\documentclass[10pt,twoside,a4paper]{amsart}

\usepackage[english]{babel}
\usepackage[utf8]{inputenc}

\usepackage{amsfonts,amsthm,amssymb,amsmath,amsthm,latexsym,wasysym,stmaryrd,mathrsfs,dsfont,txfonts}
\usepackage{accents,multirow,rotating,subfigure,graphicx,bigstrut,lscape,multicol,fancyhdr,enumerate,accents,mathtools,exscale}

\usepackage{pdftricks}
\usepackage{color}
\usepackage[pdftex,all]{xy}

\usepackage{hyperref}
\usepackage{appendix}
\usepackage[justification=centering]{caption}

\theoremstyle{theorem}
\newtheorem {theo}{Theorem}[section]
\newtheorem*{theo*}{Theorem}
\newtheorem {lemme}[theo]{Lemma}
\newtheorem*{lemme*}{Lemma}
\newtheorem {prop}[theo]{Proposition}
\newtheorem*{prop*}{Proposition}
\newtheorem {cor}[theo]{Corollary}
\newtheorem*{cor*}{Corollary}

\newtheorem*{cor_proof*}{Corollary (of the proof)}

\newtheorem*{conjecture*}{Conjecture}
\theoremstyle{definition}
\newtheorem {defi}[theo]{Definition}
\newtheorem*{defi*}{Definition}
\newtheorem {nota}[theo]{Notation}
\newtheorem*{nota*}{Notation}
\theoremstyle{remark}
\newtheorem {remarque}[theo]{Remark}
\newtheorem*{remarque*}{Remark}

\newtheorem*{warning*}{Warning}
\newtheorem {remarques}[theo]{Remarks}
\newtheorem*{remarques*}{Remarks}

\newtheorem*{warnings*}{Warnings}

\newtheorem*{convention*}{Convention}

\newtheorem*{exemple*}{Example}

\newtheorem*{exemples*}{Examples}
\newtheorem {question}[theo]{Question}
\newtheorem*{question*}{Question}

\newtheorem*{questions*}{Questions}

\newtheorem*{fact*}{Fact}
\newtheorem*{acknowledgments}{Acknowledgments}

\newcommand{\DD}{{\mathcal D}}

\newcommand{\GG}{{\mathcal G}}

\newcommand{\KK}{{\mathcal K}}

\newcommand{\R}{{\mathds R}}

\newcommand{\TT}{{\mathcal T}}

\newcommand{\e}{\varepsilon}
\newcommand{\p}{\partial}

\newcommand{\Diff}{{\textnormal{Diff}}}

\newcommand{\Id}{{\textnormal{Id}}}

\newcommand{\ie}{{\it i.e. }}

\newcommand{\psqcup}{\operatornamewithlimits{\sqcup}\limits}

\newcommand{\fract}[2]{\hbox{\leavevmode
  \kern.1em \raise .25ex \hbox{\the\scriptfont0 $#1$}\kern-.1em }\big/
  {\hbox{\kern-.15em \lower .5ex \hbox{\the\scriptfont0 $#2$}} }}

\newcommand{\ffract}[2]{\hbox{\leavevmode
  \kern.1em \raise .25ex \hbox{\the\scriptfont0 $#1$}\kern-.1em }\big/
  {\hbox{\kern-.15em \lower .5ex \hbox{\the\scriptfont0 \scriptsize $#2$}} }}

\newcommand{\fractt}[2]{\hbox{\leavevmode
  \kern.1em \raise .25ex \hbox{\the\scriptfont0 $#1$}\kern-.1em
}\lower .2ex\hbox{\Big/}
  {\hbox{\kern-.15em \lower .8ex \hbox{\the\scriptfont0 $#2$}} }}

\newcommand{\subfract}[2]{\hbox{\leavevmode
  \kern.1em \raise .25ex \hbox{\the\scriptfont0 \scriptsize $#1$}\kern-.1em }/
  {\hbox{\kern-.15em \lower .5ex \hbox{\the\scriptfont0 \scriptsize $#2$}} }}

\newcommand{\noi}{\noindent}

\newcommand{\dessin}[2]{
  \vcenter{\hbox{\includegraphics[height=#1]{#2.pdf}}}}
\newcommand{\dessinH}[2]{
  \vcenter{\hbox{\includegraphics[width=#1]{#2.pdf}}}}
\newcommand{\fdessin}[2]{\fbox{$\dessin{#1}{#2}$}}
\newcommand{\fdessinH}[2]{\fbox{$\dessinH{#1}{#2}$}}

\newcommand{\bD}{\textrm{b}\DD}
\newcommand{\bK}{\textrm{b}\KK}
\newcommand{\bT}{\textrm{b}\TT}

\newcommand{\rK}{\textrm{r}\KK}
\newcommand{\rT}{\textrm{r}\TT}

\newcommand{\vD}{\textrm{v}\DD}
\newcommand{\wK}{\textrm{w}\KK}
\newcommand{\wT}{\textrm{w}\TT}

\newcommand{\wGD}{\textrm{w}\GG\DD}
\newcommand{\wGK}{\textrm{w}\GG\KK}

\newcommand{\xB}{\textrm{B}}
\newcommand{\tB}{\widetilde{\xB}}
\newcommand{\xR}{\textrm{R}}
\newcommand{\tR}{\widetilde{\xR}}
\newcommand{\xW}{\textrm{W}}
\newcommand{\xT}{\textrm{T}}
\newcommand{\Tube}{\textrm{Tube}}

\newcommand{\tT}{\widetilde{T}}

\newcommand{\blc}{.2cm}
\newcommand{\Rar}{\hspace{\blc}\rightarrow\hspace{\blc}}
\newcommand{\Lar}{\hspace{\blc}\leftarrow\hspace{\blc}}

\begin{document}

\title{On the welded Tube map}
\author{Benjamin \textsc{Audoux}}
\date{\today}
\address{Aix Marseille Université, I2M, UMR 7373, 13453 Marseille, France
}
\email{benjamin.audoux@univ-amu.fr}
\thanks{The author is supported by the French ANR research project ``VasKho'' ANR-11-JS01-00201.}

\begin{abstract}
This paper investigates the so-called $\Tube$ map which connects welded knots, that is a quotient of the virtual knot theory, to ribbon torus-knots, that is a restricted notion of fillable knotted tori in $S^4$.
It emphasizes the fact that ribbon torus-knots with a given filling are in one-to-one correspondence with welded knots before quotient under classical Reidemeister moves and reformulates these moves and the known source of non-injectivity of the $\Tube$ map in terms of filling changes.
\end{abstract}


\maketitle

This papers investigates a known connection between two notions of distinct nature : welded knots, which are combinatorial elements of a quotient of the virtual knot theory; and ribbon torus-knots, which are a restricted notion of topological knotted surfaces in $S^4$.

\medskip

Virtual knots are a completion of usual knots, seen from the diagrammatical point of view. Knot diagrams can be thought of as abstract oriented circles with a pairing, that is a finite number of signed pairs of ordered ``merged'' points.
Indeed, as shown in Figure \ref{fig:PairCross}, every such ordered and signed pair describes a crossing and the rest of the circle prescribes how the ends of these crossings are connected; up to isotopy, this data is sufficient to recover a diagram and hence a knot in $\R^3$. However, not all pairings are realizable as a diagram since it may be impossible to connect, in the plane, all ends as prescribed without introducing some additional crossing.
A virual knot is such an abstract circle with a pairing, whether it is realizable or not, up to some relevant moves inherited from the usual knot theory. Equivalently, it can be described as a diagram with possibly some additional \emph{virtual} crossings, represented as circled crossings, which are not reported in the pairing.
\begin{figure}[h]
  \[
  \begin{array}{c}
\dessin{1.5cm}{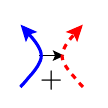}\\\downarrow\\\dessin{1.5cm}{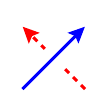}
\end{array}
\hspace{1cm}
  \begin{array}{c}
\dessin{1.5cm}{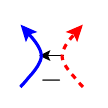}\\\downarrow\\\dessin{1.5cm}{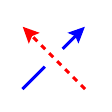}
\end{array}
\hspace{1cm}
  \begin{array}{c}
\dessin{1.5cm}{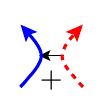}\\\downarrow\\\dessin{1.5cm}{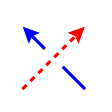}
\end{array}
\hspace{1cm}
  \begin{array}{c}
\dessin{1.5cm}{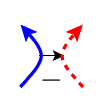}\\\downarrow\\\dessin{1.5cm}{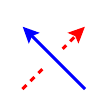}
\end{array}
  \]
  \caption{Correspondance between oriented and signed pairs and crossings}
  \label{fig:PairCross}
\end{figure}

Virtual knots can be interpreted as knots in thickened surfaces modulo handle stabilization \cite{CKS,Kuperberg}. Because of the stabilizations, this does not provide a well defined notion of complement; in spite of that, the usual notion of knot group, \ie fundamental group of the complement, can be combinatorially extended to the virtual case \cite{Kim,Silver,BB}. Knot groups happens to be invariant under the so-called \emph{over-commute move} (OC), a move (see Figure \ref{fig:OC}) which, in general, modifies virtual knots. There are other topological invariants which, similarly, can be combinatorially extended and proved to be invariant under OC.
This motivates the definition of \emph{welded knots} which are the quotient of virtual knots under these moves.
Welded knotted objects first appeared in a work of Fenn-Rimanyi-Rourke in the more algebraic context of braids \cite{FRR}.
At this stage, one can hope that welded knots admit a deeper topological interpretation than virtual knots do, yielding back a topological nature for the combinatorial extension of the invariants mentioned above. 

\medskip

The theory of knotted surfaces in 4--space takes its origins in the mid-twenties from the work of Artin \cite{artin2}. 
However, the systematic study of these objects only really began in the early sixties, notably through the work of Kervaire, Fox and Milnor \cite{KM,kervaire,FM}, 
but also in a series of papers from Kansai area, Japan (see references in \cite{Suzuki}).
From this early stage, the class of ribbon surfaces was given a particular attention.
Roughly speaking, an embedded surface in $S^4$ is ribbon if it can be filled by an immersed $3$-manifold whose singular set is a finite number of rather simple singularities called \emph{ribbon disks}.
Asking for the surface to be fillable does not change the notion of isotopy but restrict the equivalence classes we are looking at. Indeed, an ambiant isotopy transport as well any given filling, so two isotopic surfaces are simultaneously fillable or not.
On the contrary, asking for a surface to be actually filled multiplies the number of equivalence classes since a single surface may have several non isotopic fillings.
Asking for filledness better than for fillability is hence more constraining but --- at least, this is what the present paper tends to show --- it is easier to handle. Fillable surfaces are furthermore obtained as the quotient of filled surfaces under filling changes.

The case of the torus was already addressed in \cite{Yaji} under the name of ribbon torus-knots. In his paper, T. Yajima laid the foundation of the so-called $\Tube$ map which inflates usual knot into ribbon torus-knots. In \cite{Satoh}, S. Satoh shows that this map can actually be extended to all welded knots.
He proves moreover that the map is then surjective, that the welded combinatorial knot group corresponds to the fundamental group of the complement in $S^4$ of the image, and that it commutes with the orientation reversals --- what is not direct from the definition but a consequence of the torus eversion in $S^4$.
Ribbon torus-knots are hence good candidates for a topological interpretation of welded knots. However the $\Tube$ map is not injective and the lack of injectivity is not fully understood yet: the $\Tube$ map is known to
be invariant under a global reversal move (see Prop. 3.4 in \cite{Ishi} or Prop. \ref{prop:GlRevInv} in the present paper) but it is not known whether the $\Tube$ map quotiented by this move is injective.

\medskip

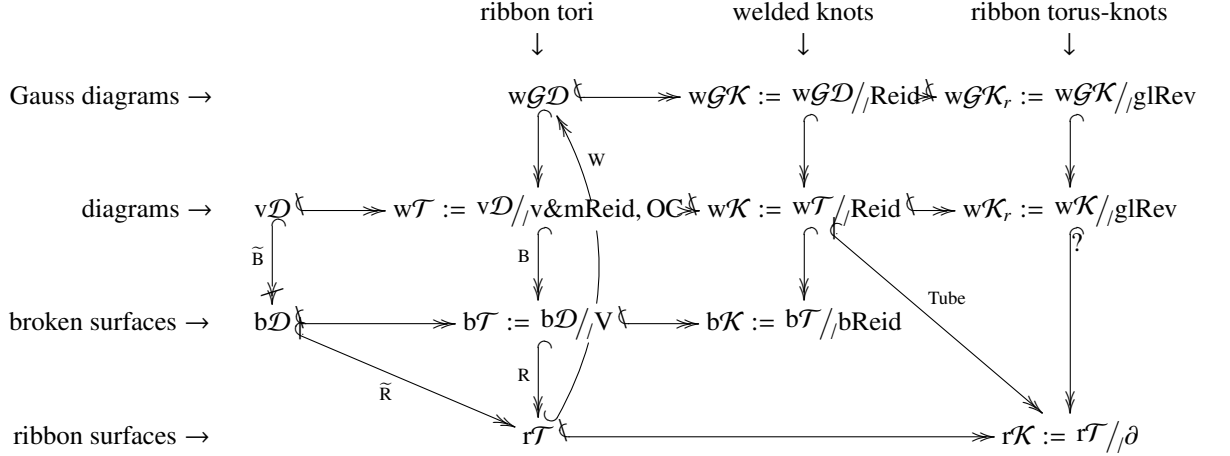
\begin{figure}
  \begin{gather*}
  \hspace{3.45cm}
  \xymatrix@!0 @C=3.5cm @R=.4cm{
&\textrm{ribbon tori}&\textrm{welded knots}&\textrm{ribbon torus-knots}\\
&\downarrow&\downarrow&\downarrow
  }\\
  \newdir{|(}{*!/.12cm $\setminus$/@{.}*@^{(}}
  \newdir{/(}{*@{ }*!/.03cm $|$/@{.}*@^{(}}
  \xymatrix@!0 @C=3.5cm @R=1.5cm{
    **[l]\textrm{Gauss diagrams}\rightarrow\\
    **[l]\textrm{diagrams}\rightarrow\\
    **[l]\textrm{broken surfaces}\rightarrow\\
    **[l]\textrm{ribbon surfaces}\rightarrow
    }
    \hspace{.3cm}
  \xymatrix@!0 @C=3.5cm @R=1.5cm{
    &
    \wGD \ar@{|(->>}[r] \ar@{^(->>}[d]&
    \wGK:=\ffract{\wGD}/{\textrm{Reid}} \ar@{|(->>}[r] \ar@{^(->>}[d]&
     \wGK_r:=\ffract{\wGK}/{\textrm{glRev}} \ar@{^(->>}[d]\\
     \vD \ar@{|(->>}[r] \ar@{^(->>}[d]|(.72){\rotatebox{90}{$\setminus$}}_(.4){\tB}&
     \wT:=\ffract{\vD}/{\textrm{v\&mReid},\textrm{OC}} \ar@{|(->>}[r] \ar@{^(->>}[d]_(.4){\xB}&
     \wK:=\ffract{\wT}/{\textrm{Reid}} \ar@{|(->>}[r] \ar@{^(->>}[d]\ar@{/(->>}[rdd]^(.45){\Tube}&
     \wK_r:=\ffract{\wK}/{\textrm{glRev}} \ar@{^(->>}[dd]^(.14){\hspace{-.1cm}\textrm{\normalsize ?}}\\
     \bD \ar@{|(->>}[r] \ar@{/(->>}[rd]_(.45){\tR}&
     \bT:=\ffract{\bD}/{\textrm{V}} \ar@{|(->>}[r]  \ar@{^(->>}[d]_(.45){\xR}&
     \bK:=\ffract{\bT}/{\textrm{bReid}} &
     \\
    &
     \rT \ar@{|(->>}[rr] \ar@<-.1cm>@{^(->>}@/_.7cm/[uuu]_(.8){\xW}|(.3){\hole}|(.37){\hole}|(.63){\hole}|(.67){\hole}&
     &
     \rK:=\ffract{\rT}/{\p}
    }
  \end{gather*}
  \caption{Summary of the combinatorial and/or topological sets and their connections}
  \label{fig:Summary}
\end{figure}

This paper contains no new result, but it reviews and reformulates the construction of the $\Tube$ map.
In particular, it emphasizes the fact that ribbon torus-knot with a given filling are in one-to-one correspondence with welded knots before quotient under classical Reidemeister moves. We interpret then invariance of the $\Tube$ map under Reidemeister moves, and notably Reidemeister move I which is the most intricated one, in terms of local filling changes.The global reversal move which let the $\Tube$ map invariant is also interpreted as a co-orientation change in the filling.
Another goal of this paper is to clarify the relationship between three kind of objects:
\begin{itemize}
\item ribbon ones which are 2 or 3--dimensional objects inside a 4--dimensional space;
\item broken ones, which are decorated 2--dimensional objects inside a 3-dimensional space;
\item welded ones which are decorated 1--dimensional objects inside a 2--dimensional space. There is also one exceptional item, namely welded Gauss diagrams, which is purely combinatorial. 
\end{itemize}
They are, respectively, denoted with a small ``r'', ``b'' or ``w'' prefix. 

The paper is organized as follows. The first section defines all the topological and/or combinatorial notions of knotted objects. There is a lot of redundancy in these notions; a summary of them is given in Figure \ref{fig:Summary}. In the second section, all maps between the different notions are defined, in particular the $\Tube$ map. The non-injectivity is discussed there.
The last section addresses the notion of \emph{wen}, which is a special portion of ribbon torus-knot embedded in $S^4$ as a Klein bottle cut along a meridional circle. Wens shall appear as items we want to avoid when projecting ribbon torus-knots in $S^3$, but also as a useful tool for addressing Reidemeister move I and the commutation of the $\Tube$ map with the orientation reversals.

\begin{acknowledgments}
This paper was initiated after a talk of the author at the conference ``Advanced school and discussion meeting on Knot theory and its applications'', organized at IISER Mohali in december 2013.
The author is grateful to Krishnendu Gongopadhyay for inviting him.
He also warmly thanks Ester Dalvit, Paolo Bellingeri, Jean-Baptiste Meilhan and Emmanuel Wagner for encouraging and stimulating conversations, and Hans Boden for pointing out a mistake in a previous version.
\end{acknowledgments}


\section{Ribbon, broken and welded objects}
\label{sec:ribbon-notions}

\subsection{Ribbon tori and ribbon torus-knots}
\label{sec:ribbon-tori}

By convention, an immersion shall actually refer to the image $\Im(\widetilde{Y}\looparrowright X)\subset X$ of the immersion, whereas the immersed space $\widetilde{Y}$ shall be referred to as the associated \emph{abstract space}. 
Throughout this paper, every immersion $Y\subset X$ shall be considered \emph{locally flat}, that is locally homeomorphic to a linear subspace $\R^k$ in $\R^m$ for some positive integers $k\leq m$, except on $\p Y$ (resp. $\p X$), where $\R^k$ (resp. $\R^m$) should be replaced by $\R_+\times\R^{k-1}$ (resp. $\R_+\times\R^{m-1}$); and every intersection $Y_1\cap Y_2\subset X$ shall be considered \emph{flatly transverse}, that is locally homeomorphic to the intersection of two linear subspaces $\R^{k_1}$ and $\R^{k_2}$ in $\R^m$ for some positive integers $k_1,k_2\leq m$, except on $\p Y_1$ (resp. $\p Y_2$, $\p X$), where $\R^{k_1}$ (resp. $\R^{k_2}$, $\R^m$) should be replaced by $\R_+\times\R^{k_1-1}$ (resp. $\R_+\times\R^{k_1-1}$, $\R_+\times\R^{m-1}$).

Here, we shall consider 3--dimensional spaces immersed in $S^4$ in
a quite restrictive way, since we allow only the following type of singularity.
\begin{defi}\label{def:Ribbon}
  An intersection $D:=Y_1\cap Y_2\subset S^4$ is a \emph{ribbon disk} if it is isomorphic to the 2--dimensional disk and satisfies $D\subset\mathring{Y}_1$, $\mathring{D}\subset\mathring{Y}_2$ and $\p D$ is an essential curve in $\p Y_2$.
\end{defi}

Before defining the main topological objects of this paper, we want to stress the fact that, besides a 3--dimensional orientation, a solid torus can be given a co-orientation, that is a 1--dimensional orientation of
its core.
Note that orientation and co-orientation are independent notions and that the 3--dimensional orientation can be equivalently given as a 2--dimensional orientation on the boundary.
We say that a solid torus is \emph{bi-oriented} if it is given an orientation and a co-orientation.

\begin{defi}\label{def:RibbonTorus}
  A \emph{ribbon torus} is a bi-oriented immersed solid torus $D^2\times S^1\subset S^4$ whose singular set consists of a finite number of ribbon disks.\\
  We define $\rT$ as the set of ribbon tori up to ambient isotopy.
\end{defi}

\begin{defi}\label{def:RibbonTorusKnot}
  A \emph{ribbon torus-knot} is an embedded oriented torus $S^1\times S^1\subset S^4$ which bounds a ribbon torus.
We say that it admits a \emph{ribbon filling}.\\
  We define $\rK$ as the set of ribbon torus-knots up to ambient isotopy.
\end{defi}

Definitions \ref{def:RibbonTorus} and \ref{def:RibbonTorusKnot} use the same notion of ambient isotopy. By forgetting its interior and keeping only its boundary, one sends hence any ribbon torus to a ribbon torus-knot. However, a given ribbon torus-knot may have several non isotopic ribbon filling. It follows that $\rK$ may be seen as the non trivial quotient $\fract{\rT}/{\p}$, where $\p$ is the equivalence relation generated by $T\stackrel{\p}{\simeq}T'\Leftrightarrow \p T=\p T'$.

\begin{remarque}
  To emphasize the connection with welded diagrams, we have chosen to deal with ribbon torus-knots, but other ribbon knotted objects in 4--dimensional spaces can be defined similarly, such as ribbon $2$--knots \cite{Yajima,Yanagawa,Suzuki,Cochran,KS}, ribbon $2$--links, ribbon torus-links or ribbon tubes \cite{WAMB}. Since our considerations shall be local, most of the material in this paper can be transposed to any of these notions.
\end{remarque}

\subsection{Broken torus diagrams}
\label{sec:broken-torus-diagrams}

\begin{figure}
  \[
 \hspace{-.3cm} 
  \fbox{$\fdessin{1.7cm}{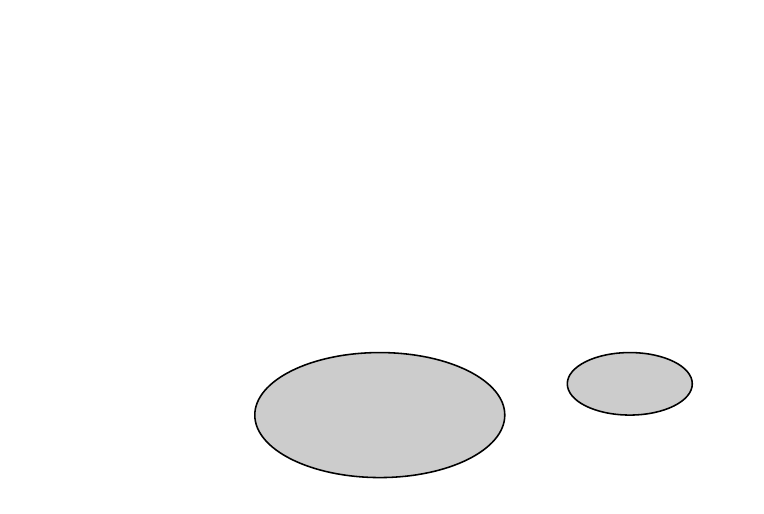}\rightarrow \fdessin{1.7cm}{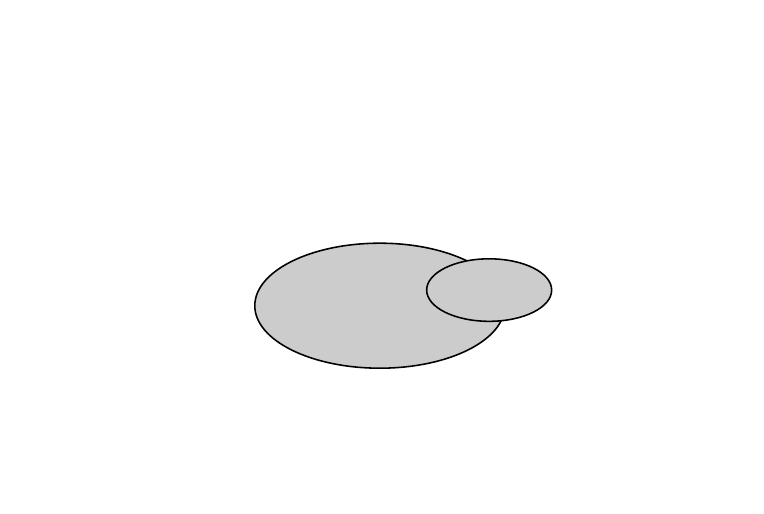}\rightarrow \fdessin{1.7cm}{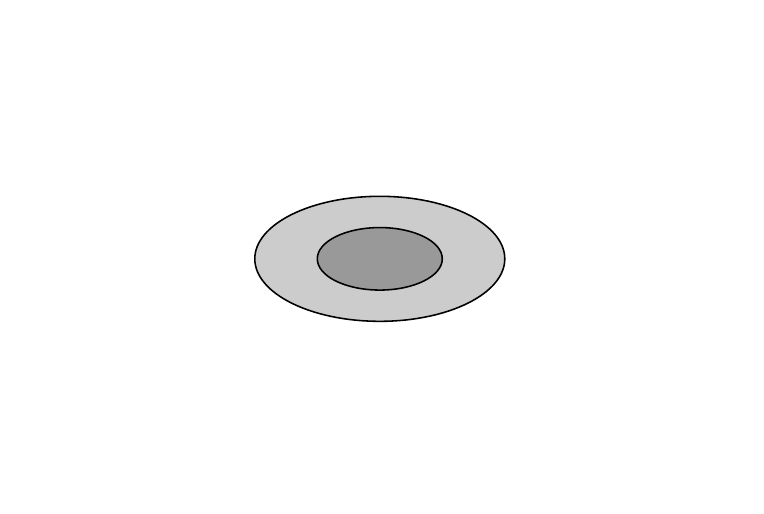}\rightarrow \fdessin{1.7cm}{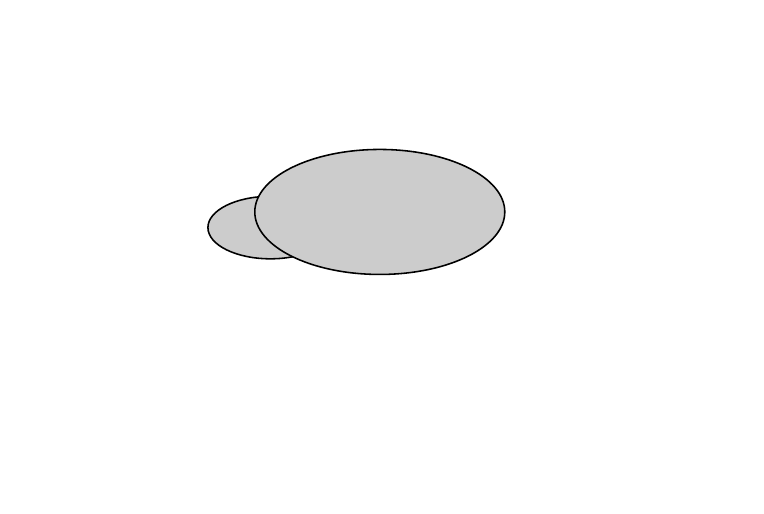}\rightarrow \fdessin{1.7cm}{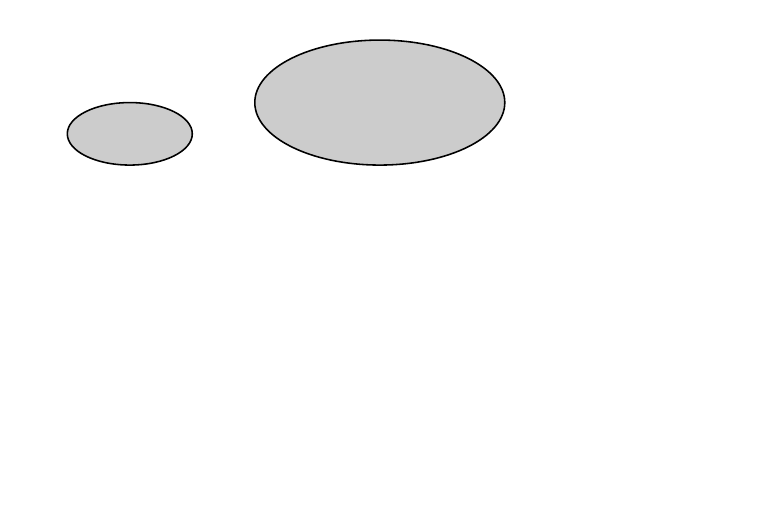}$}
  \]
\[
\begin{array}{ccc}
  \rotatebox{90}{$\ \longleftarrow\ $}&&\\
  \fdessin{4.5cm}{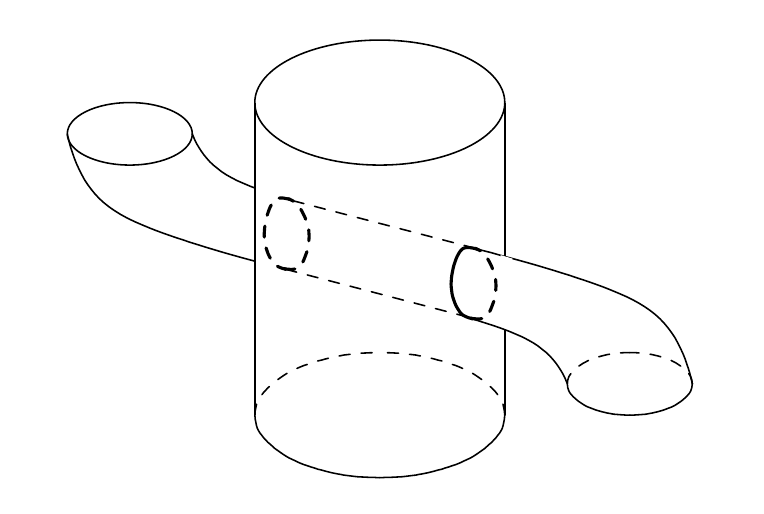}&\longrightarrow&\fdessin{4.5cm}{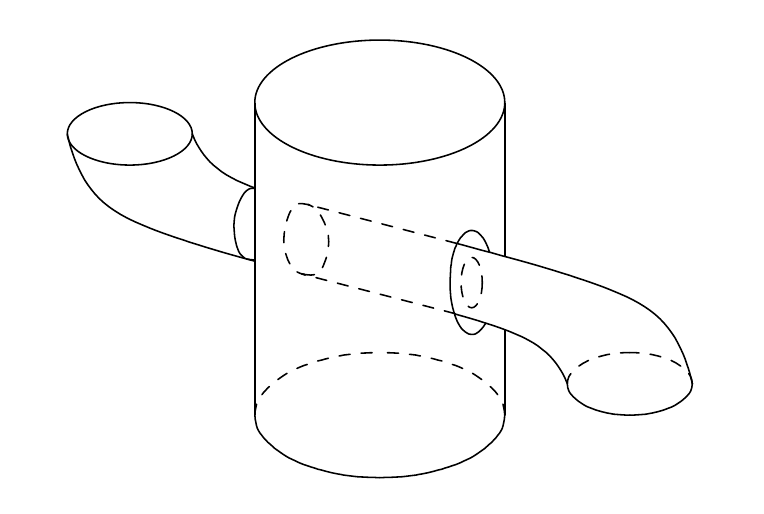}
\end{array}
\]
  \caption{Around a ribbon disk}
\label{fig:SingFly}
\end{figure}

Away from ribbon disks, pieces of a ribbon torus-knot are just annuli which can be properly projected into tubes in $\R^3$.
Around any ribbon disk, the local flatness assumption allows a local
parametrization as $B^3\times [0,1]$, where the last summand is seen as a time parameter so that the
ribbon torus-knot locally corresponds to the motion of two horizontal circles,
one flying through the other. Note that these flying circles arise naturally as
the boundary of flyings disks in $\R^3$.
By projecting along the time parameter, that is by keeping the residual track of the flying disks, we obtain two tubes,
one passing through the other.
In doing so, it creates two singular circles. But the tubes do not cross in $S^4$ so it means that for each singular circle, the time parameter separates the two preimages. There are hence one circle preimage on each tube, one having time parameters smaller than the other.
By convention, we erase from the picture a small neighborhood of the preimage with the lowest projecting parameter. See Figure \ref{fig:SingFly} for a picture. 
As we shall see, these local projections can actually be made global, and they motivate the following definition.
\begin{defi}
  A \emph{broken torus diagram} is a torus immersed in $\R^3$ whose singular set is a finite number of transverse singular circles, each of which is equipped with an order on its two preimages. As noted above and by convention, this order is specified on picture by erasing a small neighborhood of the lowest preimage.\\
  The broken surface diagram is furthermore said \emph{symmetric} if it is locally homeomorphic to either
\[
\dessin{3cm}{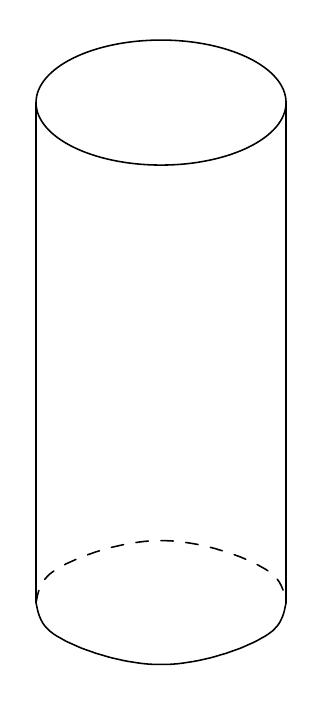}
\hspace{1cm}\textrm{or}\hspace{1cm}
\dessin{3cm}{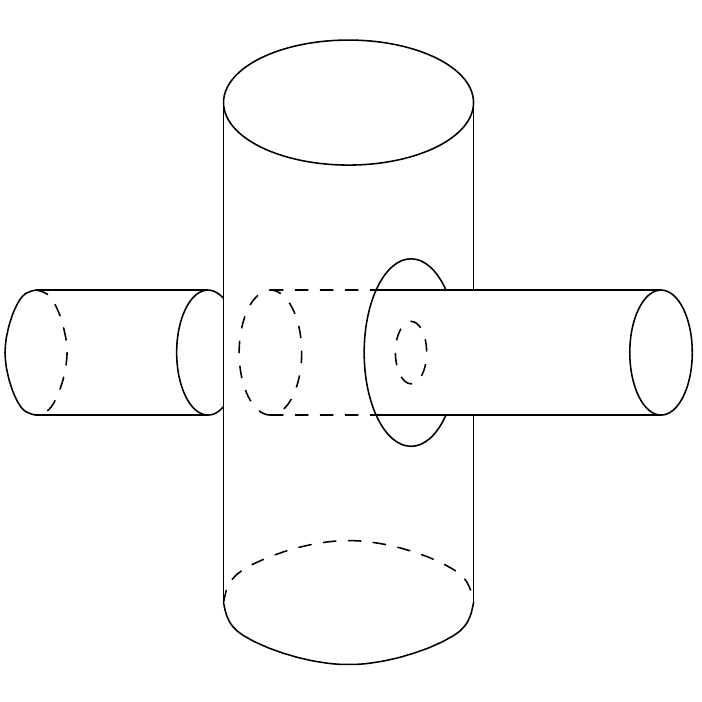}.
\]
\noi This means singular circles are pairwise matched in such a way that, for any pair, both circle has 
\begin{itemize}
\item an essential preimage: they are consecutive on a tube and the piece of tube in-between is empty of any other singular circle;
\item a non essential preimage: they are in the same connected component of the torus minus all essential preimages, and they are respectively higher and lower than their essential counterparts for the associated orders. 
\end{itemize}

As is clear from the pictures above, a symmetric broken torus diagram comes with an obvious solid torus filling which is naturally oriented by the ambient space. It is hence sufficient to enhance it with a co-orientation to provide a bi-orientation.
As a consequence, a symmetric broken torus diagram is said \emph{bi-oriented} if its natural filling is given a co-orientation.\\
We define $\bD$ the set of bi-oriented symmetric broken torus diagrams up to ambient isotopy. Unless otherwise specified, we shall assume, in the following, that all broken torus diagrams are symmetric.
\end{defi}

\begin{figure}
\[
\textrm{RI:}\xymatrix{\dessin{3.7cm}{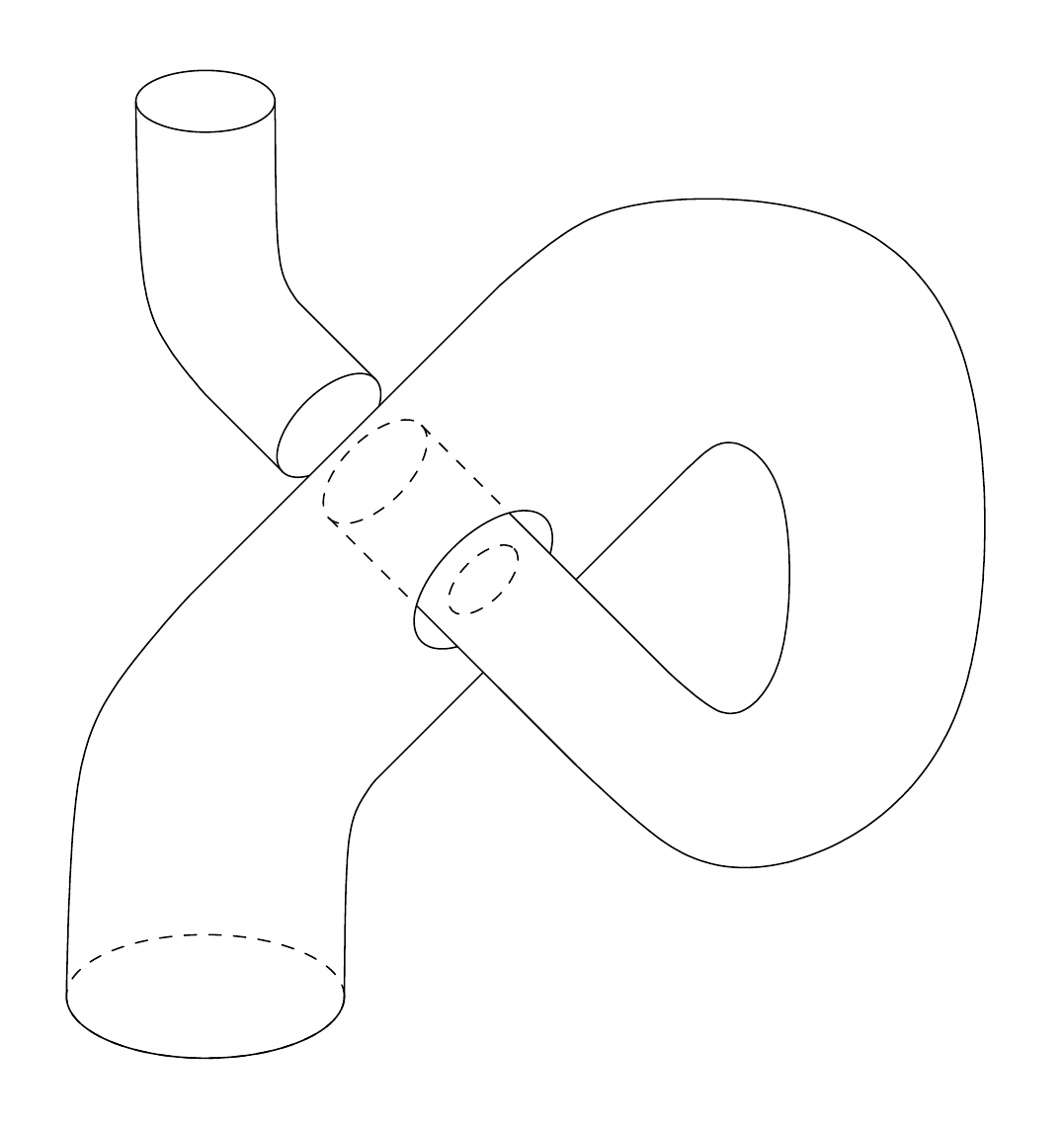}\ar@{^<-_>}[r]&\dessin{3.7cm}{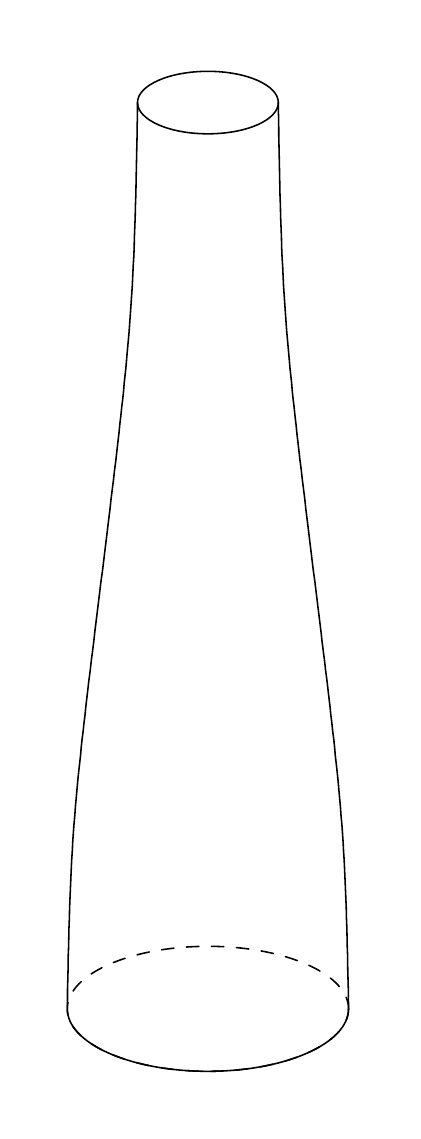}}
\hspace{1cm}
\textrm{RII:}\xymatrix{\dessin{3.7cm}{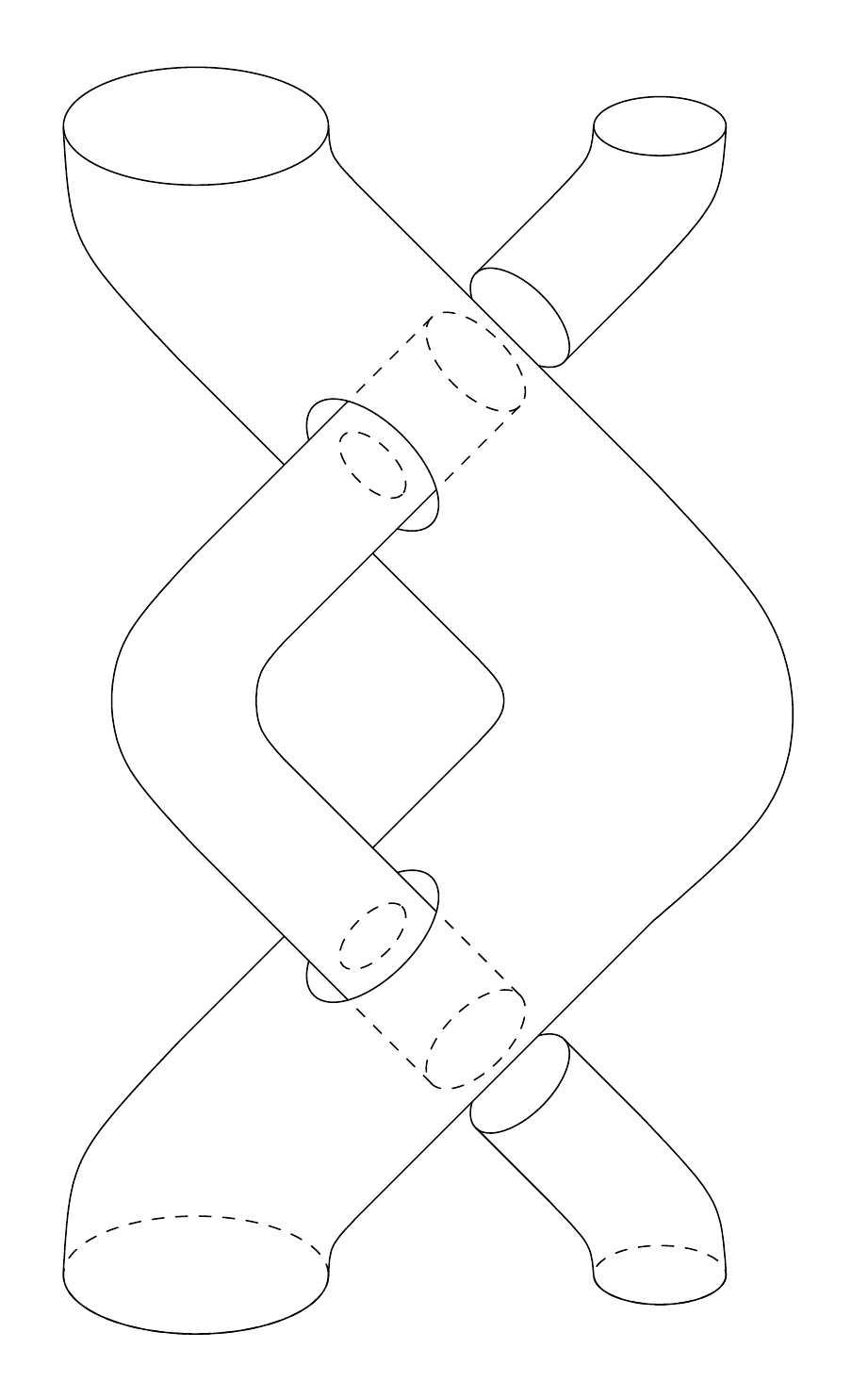}\ar@{^<-_>}[r]&\dessin{3.7cm}{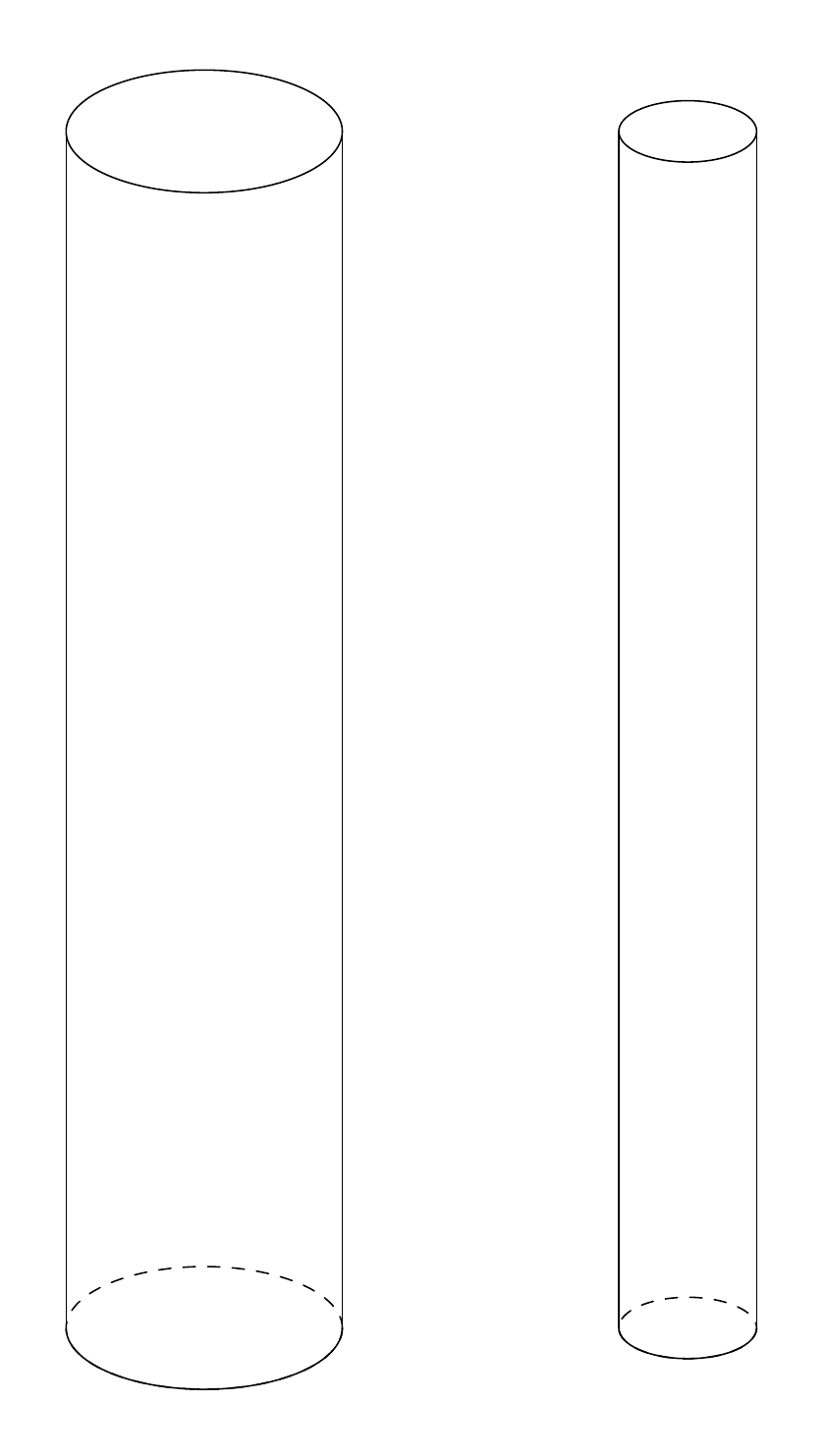}}
\]
\[
\textrm{RIII:}\xymatrix{\dessin{3.7cm}{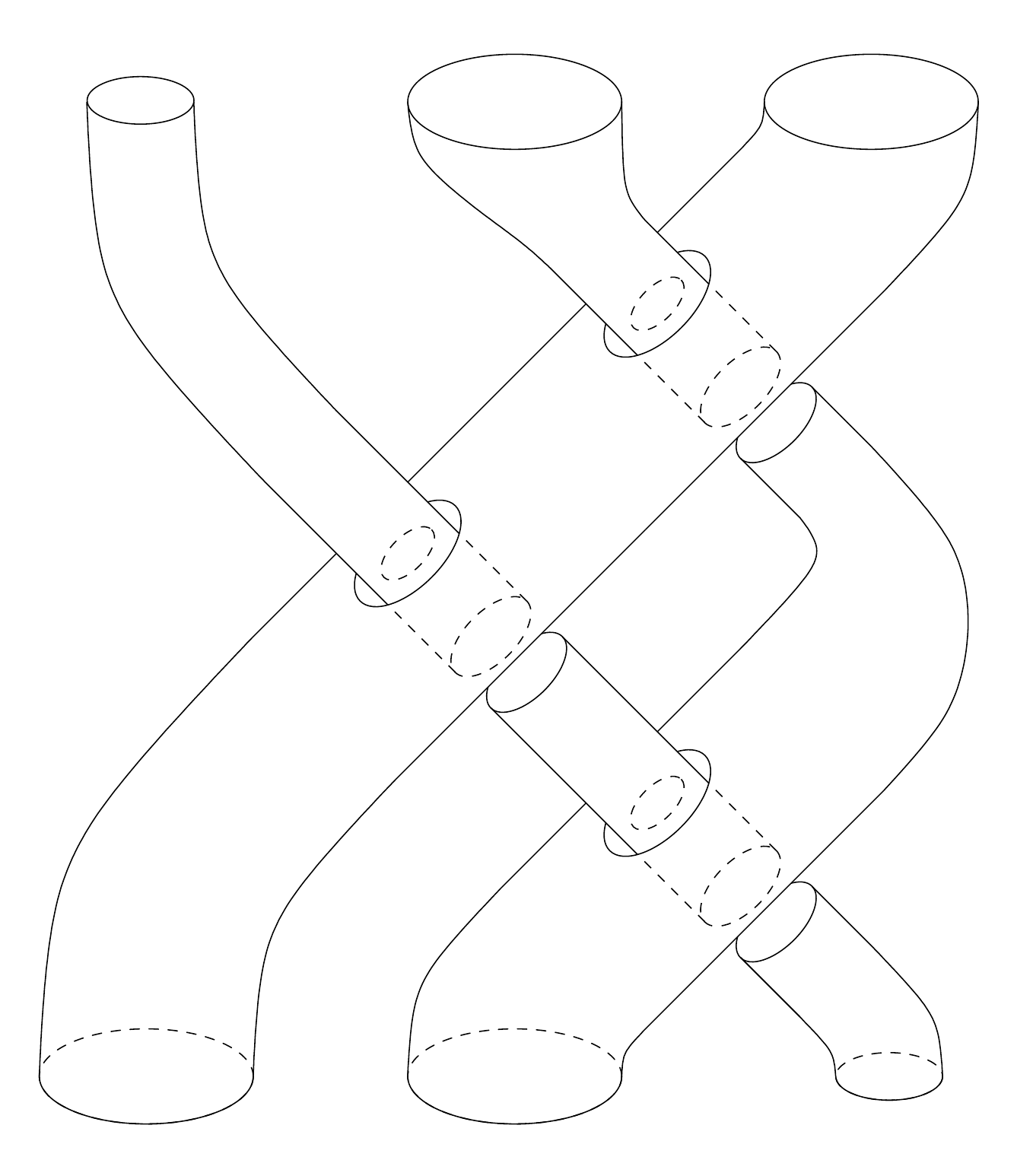}\ar@{^<-_>}[r]&\dessin{3.7cm}{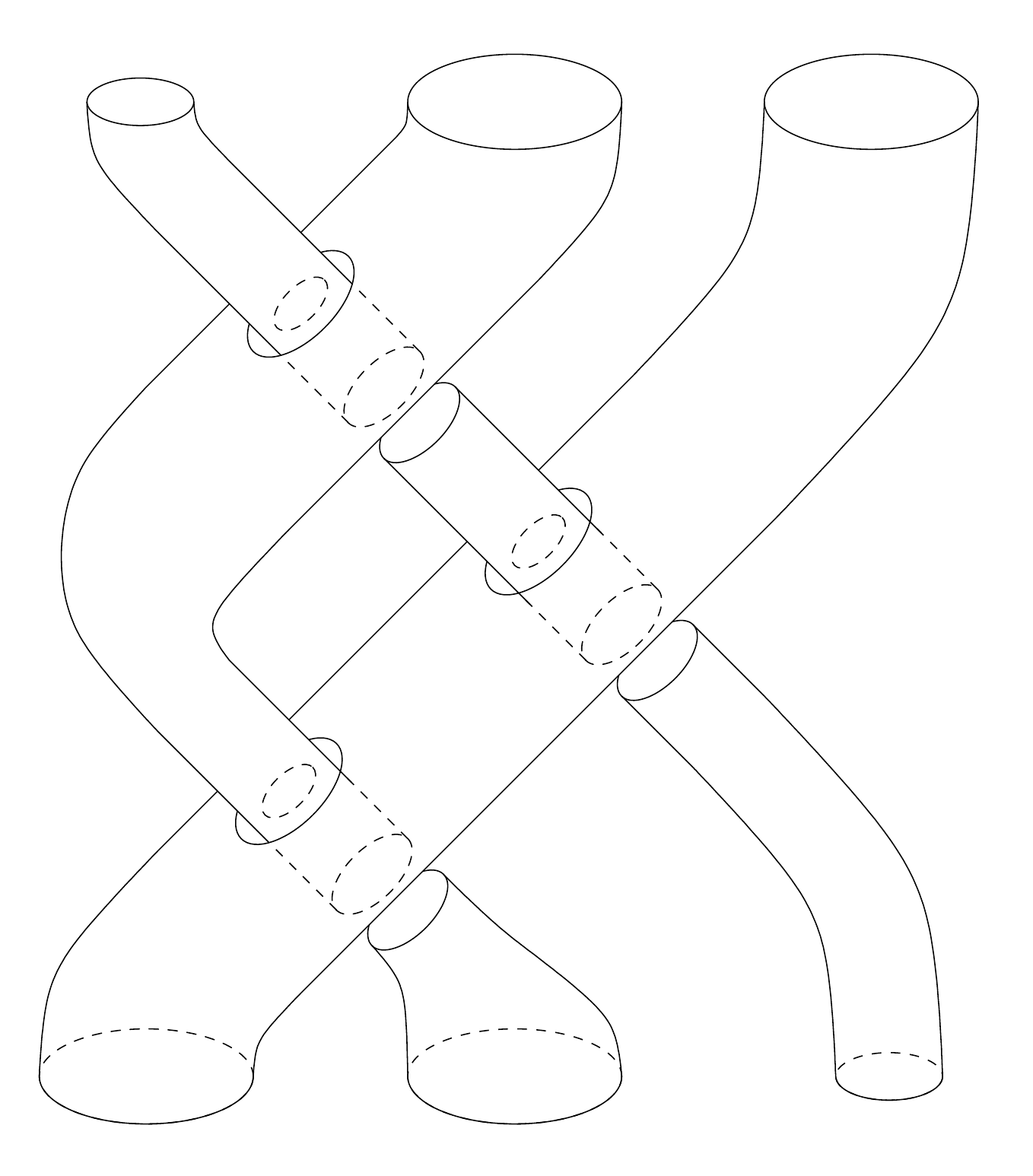}}
\]
  \caption{Broken Reidemeister moves bReid for broken torus diagrams:\\\vspace{-.25cm}{\scriptsize in these pictues, each singularity of type \protect\raisebox{-.35cm}{\protect\rotatebox{45}{$\protect\dessin{1cm}{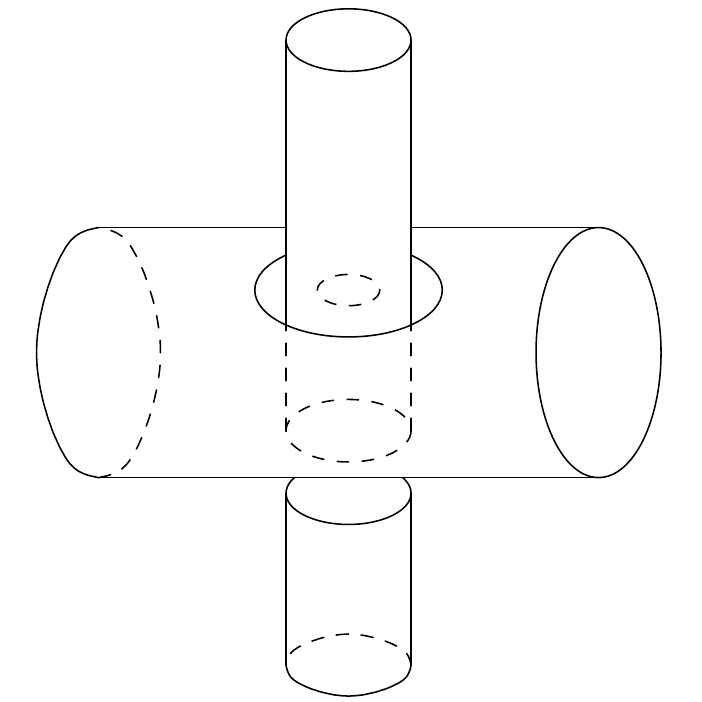}$}}
can be turned into a type \protect\raisebox{-.35cm}{\protect\rotatebox{45}{$\protect\dessin{1cm}{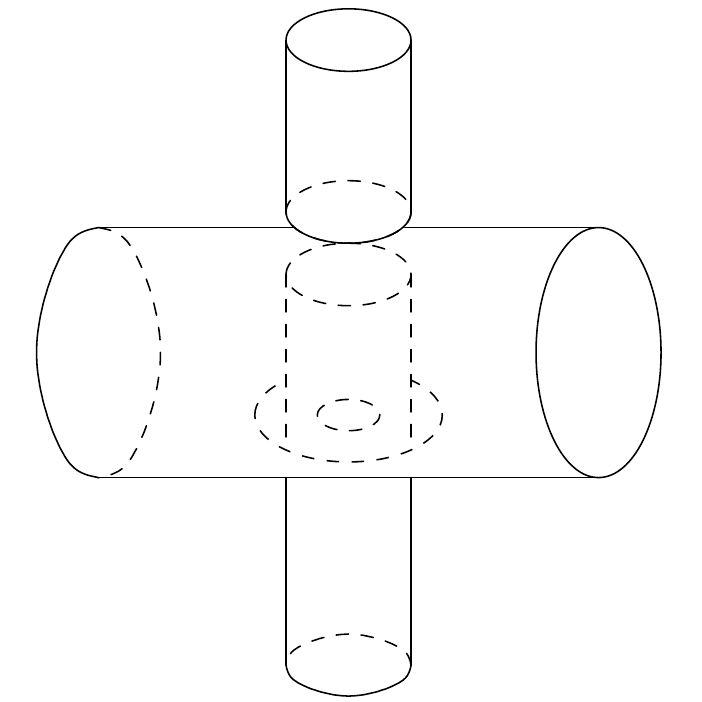}$}} one;}\\\vspace{-.25cm}{\scriptsize  however, the two singularities of move RII and the two singularities on the widest tube of move RIII}\\{\scriptsize must be then turned simultaneously}}
  \label{fig:BrokenReid}
\end{figure}
\begin{figure}
\[
\textrm{V:}\xymatrix{\dessin{3.3cm}{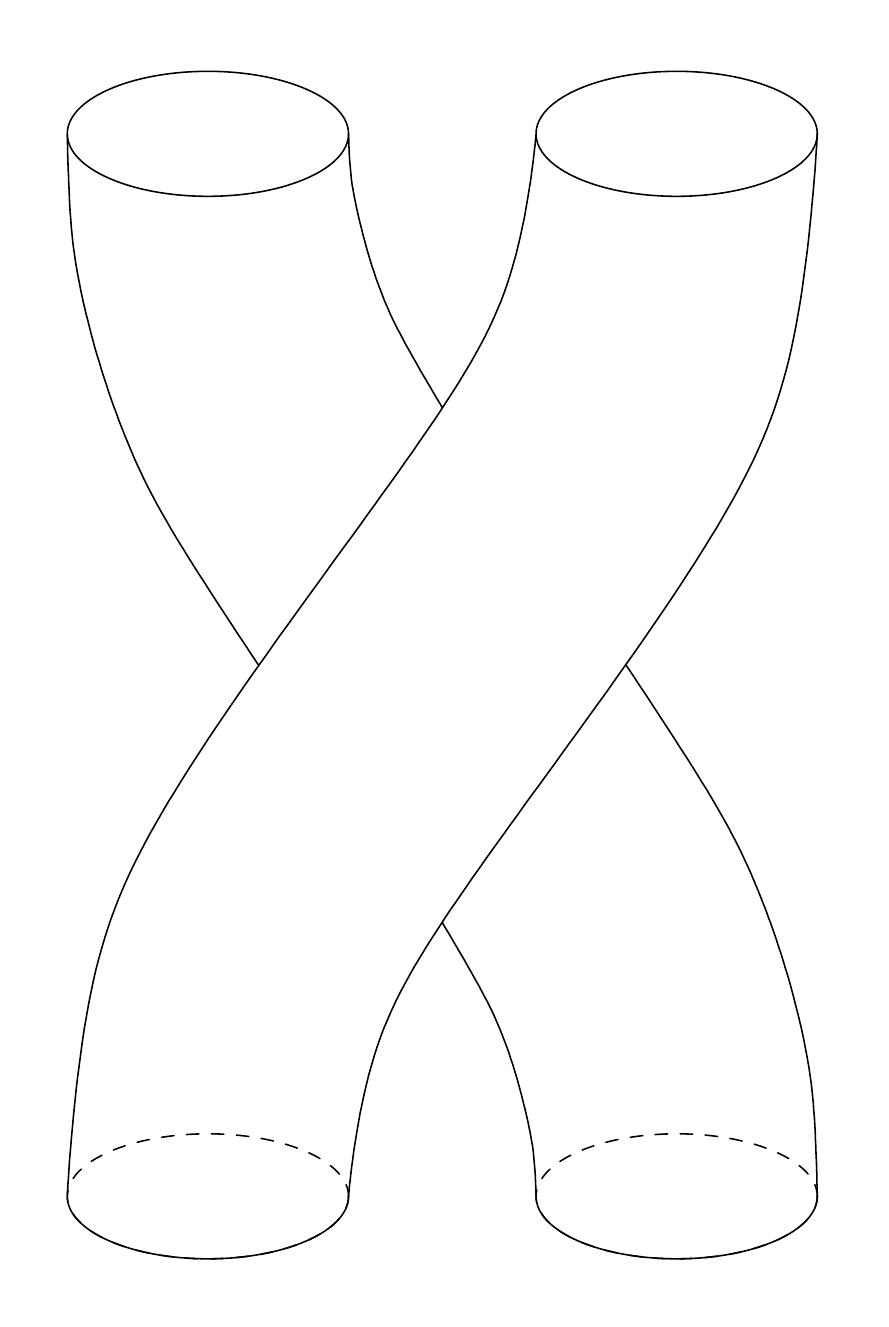}\ar@{^<-_>}[r]&\dessin{3.3cm}{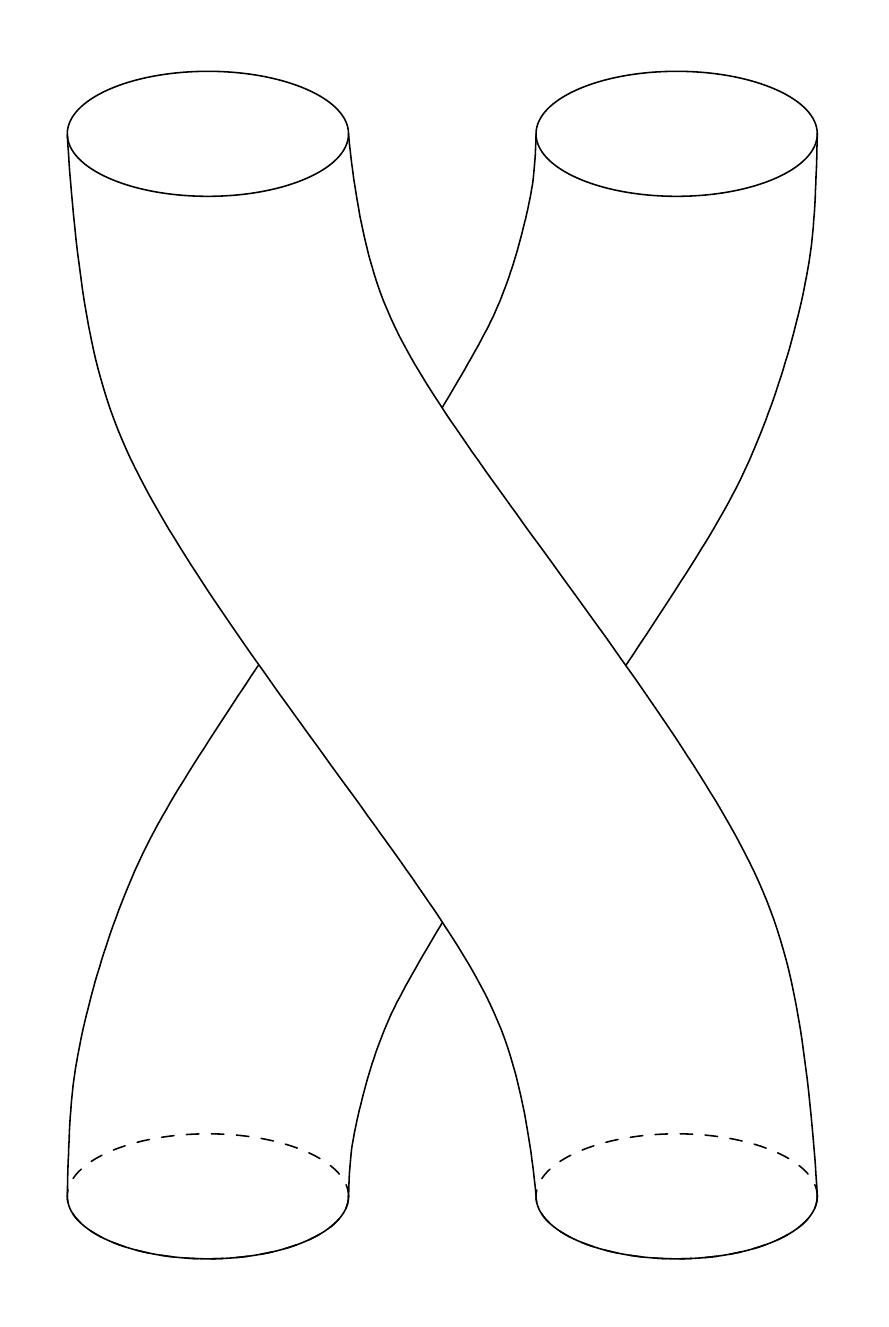}}
\]
  \caption{Virtual move for broken torus diagrams}
  \label{fig:V}
\end{figure}

A \emph{local move} is a transformation that changes a diagrammatical and/or topological object only inside a ball of the appropriate dimension. By convention, we represent only the ball where the move occurs, and the reader should keep in mind that there is a non represented part, which is identical for each side of the move.
Broken torus diagrams are higher-dimensional counterpart of usual knot diagrams.
As such, we shall quotient them by some local moves. 
\begin{defi}
  We define bReid, the \emph{broken Reidemeister moves}, and V, the \emph{virtual move}, on broken torus diagrams as the local moves shown, respectively, in Figure \ref{fig:BrokenReid} and \ref{fig:V}.
Moreover, we define the set $\bT:=\fract{\bD}/{\textrm{V}}$ of \emph{broken tori} as the quotient of broken torus diagrams under the virtual move; and the set $\bK:=\fract{\bD}/{\textrm{bReid},\textrm{V}}$ of \emph{broken knots} as the quotient of broken torus diagrams under virtual and broken Reidemeister moves, or equivalently as the quotient of broken tori under broken Reidemeister moves.
\end{defi}

\subsection{Welded knots}
\label{sec:welded-knots}
Broken torus diagrams appeared as a partially combinatorial ``$\,2\!\!\looparrowright\!\!3$''--dimensional description of ribbon objects; but, as we shall see, this description can be cut down one dimension more.  

\begin{defi}
A \emph{virtual diagram} is an oriented circle immersed in $\R^2$ whose singular set is a finite number of transverse double points, called \emph{crossing}, each of which is equipped with a partial order on its two preimages. By convention, this order is specified by erasing a small neighborhood of the lowest preimage, or by circling the crossing if the preimages are not comparable.\\
If the preimages of a crossing are comparable, then the crossing is said \emph{classical}; otherwise it is said \emph{virtual}.
Moreover, a classical crossing is said \emph{positive} if the basis made of the tangent vectors of the highest and lowest preimages is positive; otherwise, it is said \emph{negative}.\\
We define $\vD$ as the set of virtual diagrams up to ambient isotopy.
\end{defi}

\begin{figure}
  \[
  \begin{array}{c}
    \begin{array}{c}
       \textrm{R1}:\ \dessin{1.5cm}{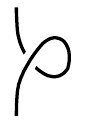}
    \leftrightarrow \dessin{1.5cm}{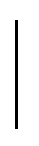} \leftrightarrow \dessin{1.5cm}{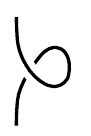}
\hspace{.7cm}
\textrm{R2}:\ \dessin{1.5cm}{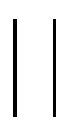} \leftrightarrow \dessin{1.5cm}{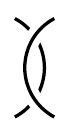}\\[1cm]
\textrm{R3a}:\ \dessin{1.5cm}{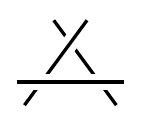} \leftrightarrow \dessin{1.5cm}{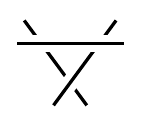}
\hspace{.7cm}
\textrm{R3b}:\ \dessin{1.5cm}{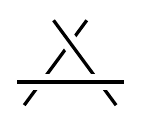} \leftrightarrow \dessin{1.5cm}{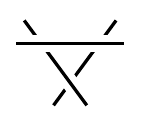}
\end{array}\\[2cm]
    \textrm{classical Reidemeister moves Reid}
  \end{array}
  \]
\vspace{.5cm}
  \[
     \begin{array}{c}
       \textrm{vR1}:\ \dessin{1.5cm}{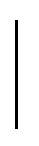}
    \leftrightarrow \dessin{1.5cm}{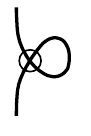}
\hspace{.7cm}
\textrm{vR2}:\ \dessin{1.5cm}{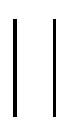} \leftrightarrow \dessin{1.5cm}{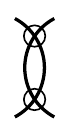}
\hspace{.7cm}
\textrm{vR3}:\ \dessin{1.5cm}{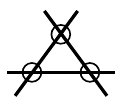} \leftrightarrow \dessin{1.5cm}{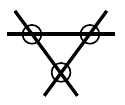}\\[1cm]
    \textrm{virtual Reidemeister moves vReid}
  \end{array}
   \]
\vspace{.5cm}
  \[
      \begin{array}{c}
       \textrm{mR3a}:\ \dessin{1.5cm}{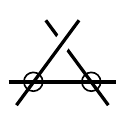}
    \leftrightarrow \dessin{1.5cm}{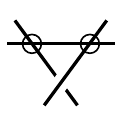}
\hspace{.7cm}
\textrm{mR3b}:\ \dessin{1.5cm}{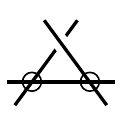} \leftrightarrow \dessin{1.5cm}{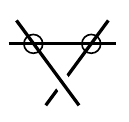}\\[1cm]
    \textrm{mixed Reidemeister moves mReid}
  \end{array}
  \]
  \caption{Generalized Reidemeister moves gReid on virtual diagrams}
  \label{fig:GenReid}
\end{figure}
\begin{figure}
 \[
\textrm{OC}:\ \dessin{1.5cm}{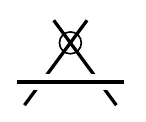}\ \longleftrightarrow \dessin{1.5cm}{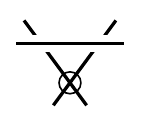}
\] 
  \caption{Over commute move on virtual diagrams}
  \label{fig:OC}
\end{figure}

\begin{defi}
  We define the \emph{generalized Reidemeister moves} gReid --- separated into classical moves Reid, virtual moves vReid and mixed moves mReid --- and the \emph{over commute move} OC on virtual diagrams as the local moves shown, respectively, in Figure \ref{fig:GenReid} and \ref{fig:OC}.
Moreover, we define the set $\wT:=\fract{\vD}/{\textrm{vReid},\textrm{mReid},\textrm{OC}}$ of \emph{welded tori} as the quotient of virtual diagrams under virtual Reidemeister, mixed Reidemeister and over commute moves; and the set $\wK:=\fract{\vD}/{\textrm{gReid},\textrm{OC}}$ of \emph{welded knots} as the quotient of virtual diagrams under generalized Reidemeister and over commute moves, that is as the quotient of welded tori under classical Reidemeister moves.
  \end{defi}

 Classical Reidemeister moves are vestiges of the diagrammatical description of the usual knot theory. Virtual and mixed Reidemeister moves essentially states that only classical crossings and the abstract connections between their endpoints are meaningful, and not the actual strands which realize these connections.
Over commute move states that two classical
crossings connected by a strand which is the highest strand for both
crossing can be commuted.
As a matter of fact, it only matters the cyclic
order of lowest preimages of classical crossings, as they are met
while running positively along the immersed oriented circle, and the unordered sets of
highest preimages which are located between two consecutive lowest
preimages.
This motivates the following definition:
\begin{defi}
  A welded Gauss diagram is a finite set $C$ given with a cyclic order
  and a map from $C$ to $C\times\{\pm1\}$.\\
  We define $\wGD$ the set of welded Gauss diagram.
\end{defi}
\begin{prop}
  There is a one-to-one correspondence between $\wGD$ and $\wT$.
\end{prop}
\begin{proof}
  To a virtual diagram $D$, we associate a welded Gauss diagram as follows.
When running along $D$, every crossing is met twice, once on the highest strand and once on the lowest one.
  So first, we consider $C$, the set of classical crossing of $D$, ordered by the order in which the crossings are met on the lowest strand while running positively along $D$.
  Then, to $c\in C$, we associate the sign of $c$ and the last crossing met on the lowest strand, while running positively along $D$, before meeting $c$ on the highest strand.

  It is straightforwardly checked that this is invariant under virtual Reidemeister, mixed Reidemeister and over commute moves and that it defines a one-to-one map from $\wT$ to $\wGD$.
This last statement can be verified by hand or, for readers who know about usual Gauss diagrams, by defining an inverse map, using item (\ref{item:GD}) in Remark \ref{rk:GD}  and the fact that Gauss diagrams up to tail commute moves are in one-to-one correspondance with welded diagrams.
See Figure \ref{fig:DwGDGD} for an illustration.
\end{proof}

\begin{figure}
  \[
  \xymatrix{
    \dessin{3cm}{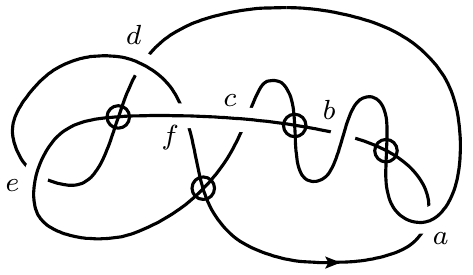} \ar[rr] &&
    {\begin{array}{ccl}
      a & \mapsto &(c,+)\\
      \rotatebox{90}{$>$}&&\\
      b & \mapsto &(c,+)\\
      \rotatebox{90}{$>$}&&\\
     c & \mapsto &(b,-)\\
      \rotatebox{90}{$>$}&&\\
     d & \mapsto &(e,-)\\
      \rotatebox{90}{$>$}&&\\
     e & \mapsto &(b,-)\\
      \rotatebox{90}{$>$}&&\\
     f & \mapsto &(b,+)\\
      \rotatebox{90}{$>$}&&\\
     a &&
    \end{array}}
\ar[rr]&&\dessin{3.3cm}{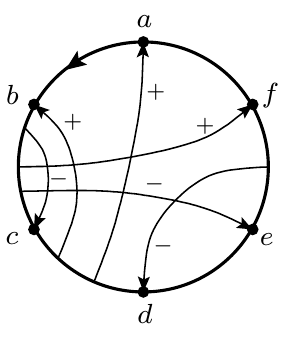} \ar `d []+<0cm,-3.2cm> `[llll] [llll]
}
  \]
  \caption{Diagrams, welded Gauss diagrams and Gauss diagrams}
  \label{fig:DwGDGD}
\end{figure}

\begin{remarques}\label{rk:GD}$\ $
  \begin{enumerate}
  \item It is quite straightforward to define relevant moves on welded
    Gauss diagrams, denoted eponymously by Reid in Figure
    \ref{fig:Summary}, that correspond to classical Reidemeister
    moves. The quotient is one-to-one with $\wK$, but since we shall
    not need them in this paper, we refer the interested reader to
    \cite[Section 4.1]{WAMB} for a formal definition.
  \item\label{item:GD} To a welded Gauss diagram $G:C\to C\times \{\pm1\}$, one can associate a more traditional Gauss diagram (see \cite{PV,GPV,Fiedler,WKO1} for definitions and Figure \ref{fig:DwGDGD} for an illustration) by
    \begin{itemize}
    \item using the cyclic order on $C$ to mark one point for each element in $C$ on an ordered circle;
    \item between the point marked by $c\in C$ and its direct successor on the oriented circle, marking one more point for each preimage in $G^{-1}\big(c,\pm1\big)$, no matter in which order;
    \item for every $c\in C$, drawing arrows from the point marked by $c'$ to the point marked by $c$ with sign $\varepsilon$ where $(c',\varepsilon)=G(c)$.
    \end{itemize}
  \end{enumerate}
\end{remarques}


\section{The Tube map}
\label{sec:tube-map}

In this section, we define some maps between ribbon and welded objects.
Broken objects shall appear as an in-between.

First, we define $\tB:\vD\to\bD$ as follows.
Let $D$ be a virtual diagram and consider it as a singular circle lying in $\R^2\times\{0\}\subset\R^3$.
Now, consider the boundary of a neighborhood of it in $\R^3$.
This is an handlebody whose boundary is the union of 4--punctured spheres, one for each crossing, attached along their boundaries.
Every such punctured sphere is hence decorated by a partial order.
We define $\tB(D)$ to be the broken torus diagram obtained by modifying locally each punctured sphere according to its partial order as shown in Figure \ref{fig:Inflating}.
It is oriented by the ambient orientation of $\R^3$ and co-oriented by the orientation of $D$.
\begin{figure}
\[
\xymatrix@C=1.35cm@R=-2cm{
&&\vcenter{\hbox{\rotatebox{45}{$\dessin{2.5cm}{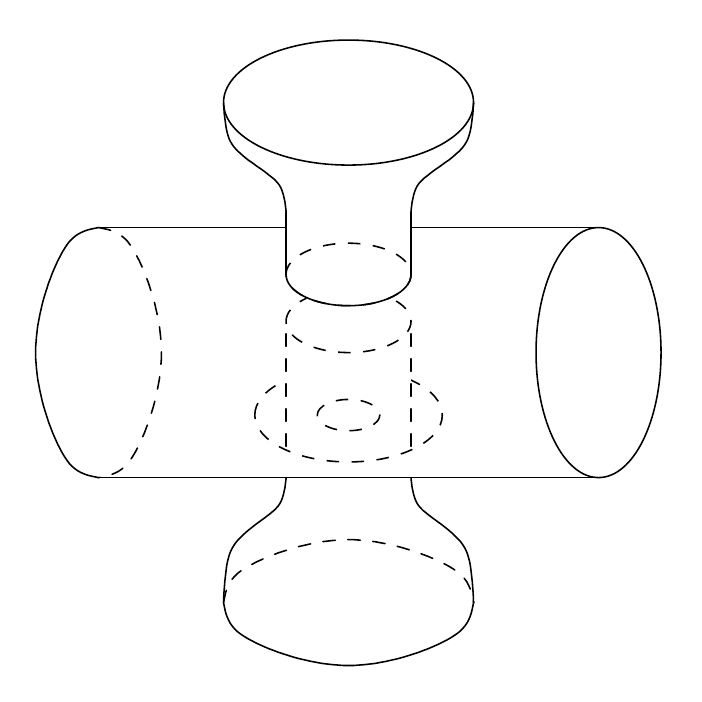}$}}}\\
\vcenter{\hbox{\rotatebox{45}{$\dessin{2.5cm}{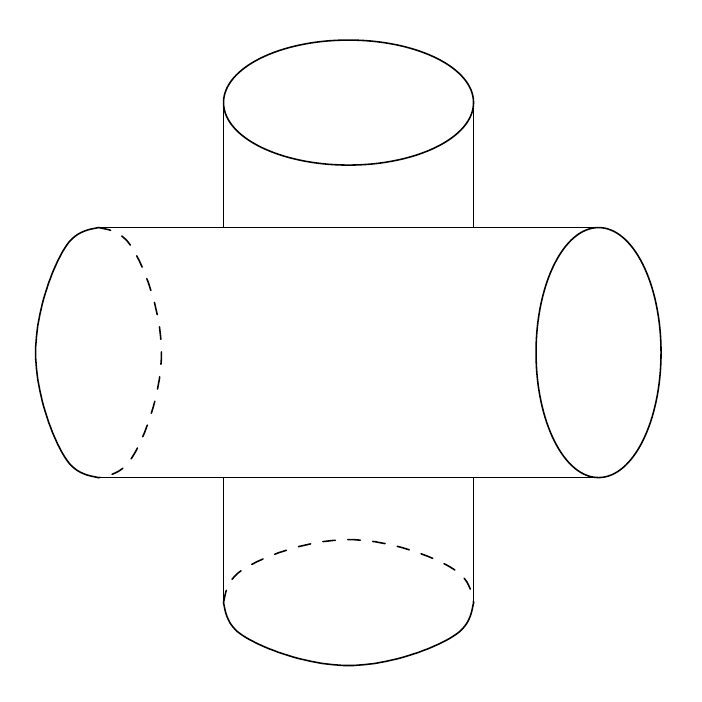}$}}}
&\vcenter{\hbox{\rotatebox{45}{$\dessin{2.5cm}{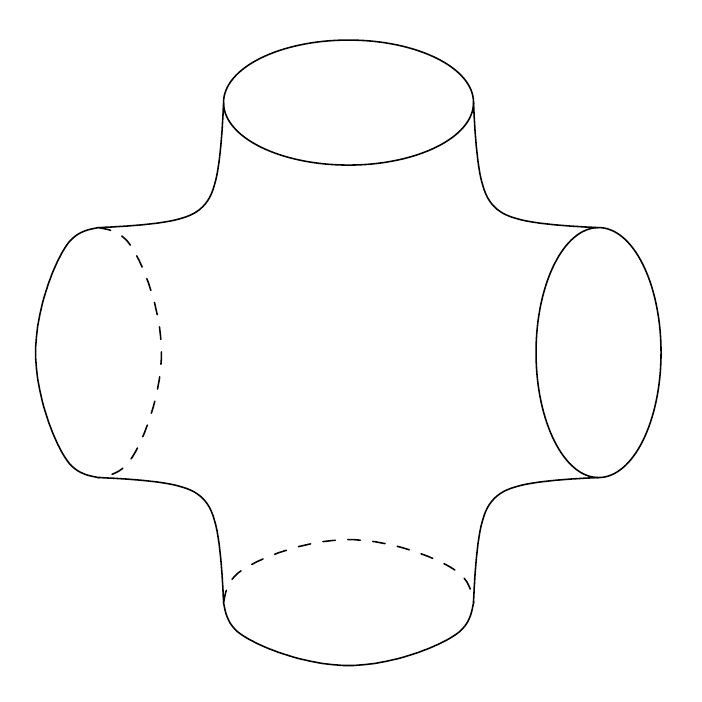}$}}}\ar[l]|(.48){\dessin{.8cm}{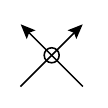}}\ar[ru]|(.48){\dessin{.8cm}{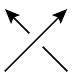}}\ar[rd]|(.48){\dessin{.8cm}{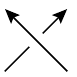}}&\\
&&\vcenter{\hbox{\rotatebox{45}{$\dessin{2.5cm}{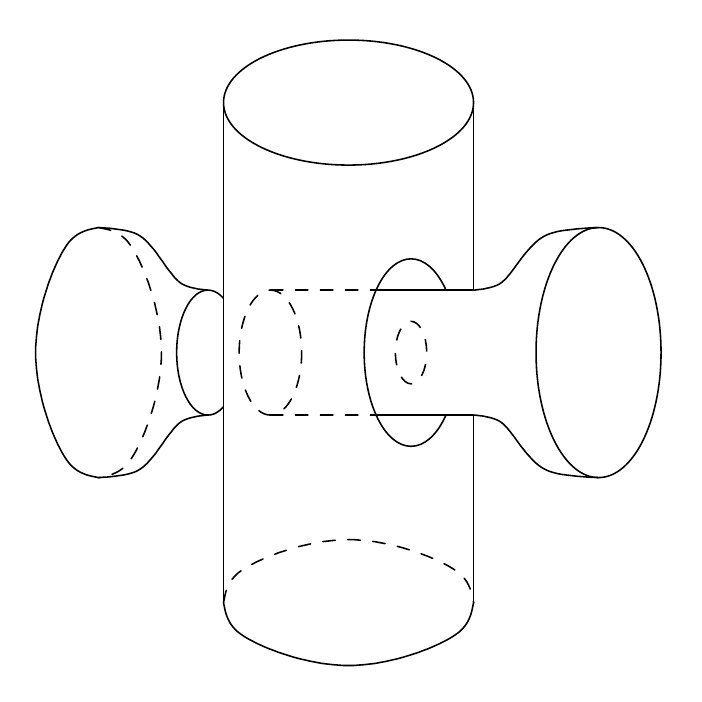}$}}}
}
\]
  \caption{Inflating classical and virtual crossings}
  \label{fig:Inflating}
\end{figure}

\begin{remarque}\label{rem:CrossChg}
  The piece of broken torus diagram associated to a classical crossing does not depend only on the crossing sign but also on the actual orientation. Indeed, reversing the orientation preserves signs but flips upside down Figure \ref{fig:Inflating}.
For instance, depending on the orientation, a positive crossing can be sent to
\[\vcenter{\hbox{\rotatebox{45}{$\dessin{1.8cm}{Infl_1}$}}}\hspace{1cm}\textrm{or}\hspace{1cm}\vcenter{\hbox{\rotatebox{45}{$\dessin{1.8cm}{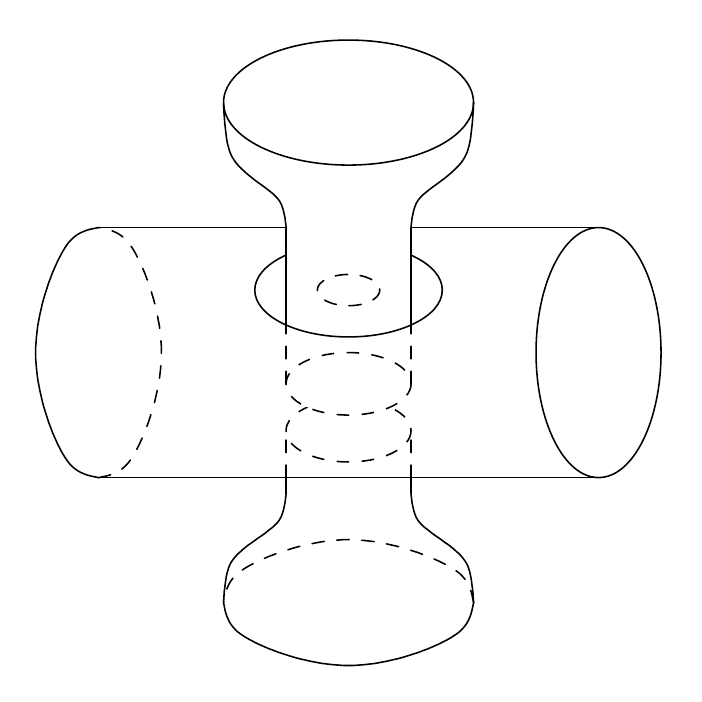}$}}},
\]
\noi which are not isotopic when fixing the boundaries.
In particular, reversing the orientation of a diagram does not just reverse, in general, the co-orientation of the associated broken torus diagram. To realize such an operation, the orientation should be reversed, but some virtual crossings should also be added before and after the crossing, and the sign of the crossing should moreover be switched.
As a matter of fact, $\dessin{.6cm}{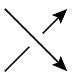}$ and $\dessin{.6cm}{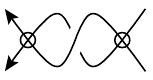}$ are send to similar pieces of broken torus diagram with opposite co-orientations. This observation shall also follow from the behavior of the map W defined below under the co-orientation reversal.
\end{remarque}

Now, we define $\tR:\bD\to\rT$ as follows.
Let $T$ be a broken torus diagram and consider it as lying in $\R^3\times\{0\}\subset S^4$.
Then, near each singular circle, the thinnest tube can be pushed slightly above or below in the fourth dimension, depending on the order associated to this singular circle.
The torus is then embedded in $S^4$.
But a symmetric broken torus diagram comes naturally with a filling by a solid torus and the pushing can be performed so that every pair of associated singular circles leads to a single ribbon disk singularity.
We define $\tR(T)$ to be this immersed solid torus with the bi-orientation induced by the one of $T$. 

\begin{lemme}
  The maps $\tB$ and $\tR$ descend to well-defined maps $\xB:\wT\to\bT$ and $\xR:\bT\to\rT$. 
\end{lemme}
\begin{proof}
  To prove the statement on $\tR$, it is suficient to prove that two broken torus diagrams which differ from a virtual move have the same image by $\tR$. Figure \ref{fig:V_isotopy} gives an explicit isotopy showing that. Another projection of this isotopy can found at \cite{Ester}

  Concerning $\tB$, it sends both sides of moves vR1 and OC on broken torus diagrams which are easily seen to be isotopic.
  For any of the other virtual and mixed Reidemeister moves, some preliminary virtual moves may be needed before the broken torus diagrams become isotopic.
  For instance, among the represented strands, one can choose one with two virtual crossings on it and, up to the virtual move, pull out its image by $\tB$ in front of the rest; then both sides of the move are sent to isotopic broken torus diagrams.
  \end{proof}
\begin{nota}
  We denote by $\xT:\wT\to\rT$ the composite map $\xR\circ\xB$. 
\end{nota}

The map $\xT$ is actually a bijection.
To prove this, we introduce $\xW:\rT\to\wGD\cong\wT$ defined as follows.
Let $T$ be a ribbon torus and $\tT$ its abstract solid torus.
We denote by $C$ the set of ribbon disks of $T$.
Each $\delta\in C$ has two preimages in $\tT$: one with essential boundary in $\p\tT$, that we shall call \emph{essential preimage} and denote by $\delta_\textrm{ess}$; and the other one inside the interior of $\tT$, that we shall call \emph{contractible preimage} and denote by $\delta_0$.
See Figure \ref{fig:TildeT} for a picture.
The co-orientation of $T$ induces a cyclic order on the set of essential preimages and hence on $C$.
Moreover, essential preimages cut $\tT$ into a union of filled cylinders.
For each $\delta\in C$, we define $h(\delta)$ as the element of $C$ such that $\delta_0$ belongs to the filled cylinder comprised between $h(\delta)_\textrm{ess}$ and its direct successor.
To $\delta$, we also associate $\varepsilon(\delta)\in\{\pm1\}$ as follows.
Let $x$ be any point in the interior of $\delta$, $x_\textrm{ess}$ its preimage in $\delta_\textrm{ess}$ and $x_0$ its preimage in $\delta_0$.
Let $(u,v,w)$ be the image in $T_xS^4$ of a positive basis for $T_{x_0}\tT$ and $z$ the image in $T_xS^4$ of a normal vector for $\delta_\textrm{ess}$ at $x_\textrm{ess}$ which is positive according to the co-orientation of $T$.
Then $(u,v,w,z)$ is a basis of $T_xS^4$ and we set $\varepsilon(\delta):=1$ if it is a positive basis and $\varepsilon(\delta):=-1$ otherwise.
See Figure \ref{fig:TildeT} for a picture.
We define $\xW(T)$ to be the welded torus corresponding to the welded Gauss diagram $G:C\to\{\pm1\}\times G$ defined by $G(\delta)=\big(\varepsilon(\delta),h(\delta)\big)$.

\begin{figure}
  \[
  \dessin{5cm}{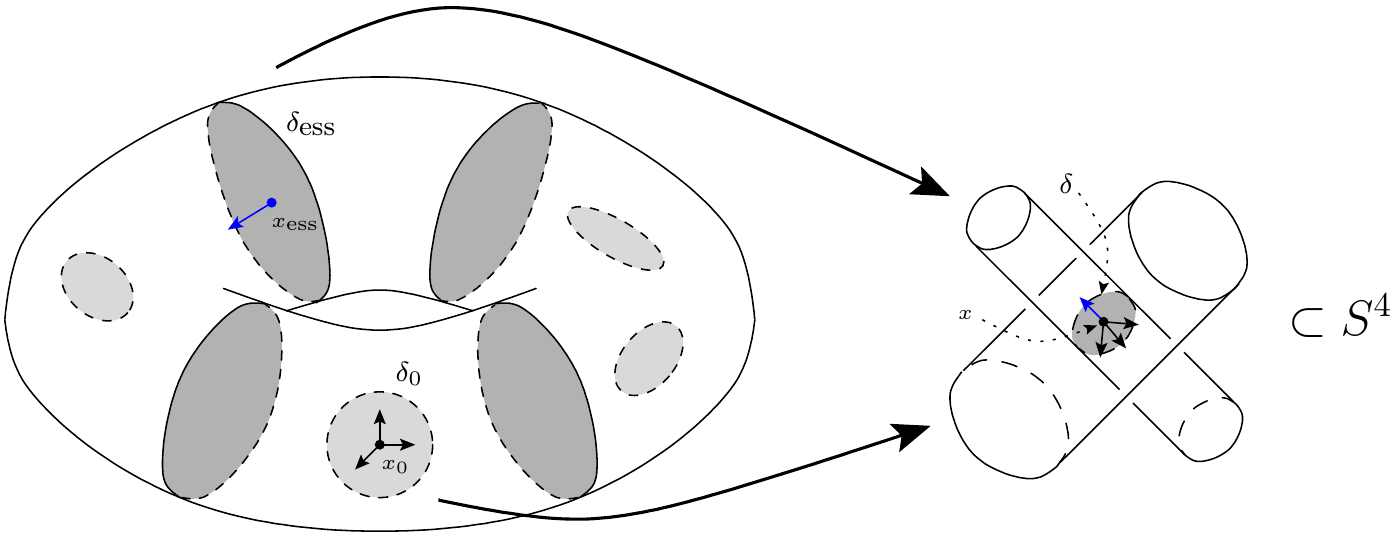}
  \]
  \caption{A picture for $\tT$}
  \label{fig:TildeT}
\end{figure}

\begin{prop}\label{prop:ToriBij}
  The map $\xT:\wT\to\rT$ is a bijection.
\end{prop}
\begin{proof}
  It is straightforwardly checked that $\xW\circ\xT=\Id_{\wT}$.
  It is hence sufficient to prove that either $\xT$ is surjective or $\xW$ is injective.
  It follows from the work of Yanagawa in \cite{Yanagawa} and Kanenobu--Shima in \cite{KS} that $\xT$ is surjective.
  More detailed references can be found in the proof of Lemmata 2.12 and 2.13 of \cite{WAMB}.
  Nevertheless, we shall sketch here an alternative proof which only uses a much less involved result --- namely, that flatly embedded 3--balls in $B^4$ can be put in a position so they project onto embedded 3--balls in $B^3$ --- and shows instead that $\xW$ is injective.

  Let $T_1$ and $T_2$ be two ribbon tori such that $\xW(T_1)=\xW(T_2)$. There is hence a bijection $\psi:C_1\to C_2$, where $C_i$ is the set of ribbon disks of $T_i$, which preserves the cyclic orders induced by the co-orientations.
  We consider $\tT_1$ and $\tT_2$ the abstract solid tori of $T_1$ and $T_2$.
Let $\delta\in C_1$, we extend $\delta_0\subset \tT_1$ into a disk $\delta'_0$ which is disjoint from the other disk extensions and from the essential preimages, and whose boundary is an essential curve in $\p\tT_1$.
Moreover, the co-orientation of $T_1$ on $\delta_\textrm{ess}$ and $\delta'_0$ provide two normal vectors for $\delta$. We choose $\delta'_0$ so that, up to isotopy, these vectors coincide.
Note that, in doing such an extension of contractible preimages, we fix an arbitrary order on the contractible preimages which are in the same filled cylinder; and, more globally, a cyclic order on the union of essential and contractible preimages.
Now, for every ribbon disk $\delta\in C_1$, we consider a small neighborhood $V(\delta)$ of the image of $\delta'_0$ in $S^4$. Locally, it looks like in Figure \ref{fig:SingFly} and $V(\delta)\cap T_1$ is the union of four disks that we shall call the \emph{ribbon disk ends} of $\delta$.
By considering a tubular neighborhood of a path which visits each $V(\delta)$ successively but avoids $T_1$ otherwise, we extend $\psqcup_{\delta\in C_1}V(\delta)$ into a 4--ball $B_*\subset S^4$ which meets $T_1$ only in $\psqcup_{\delta\in C_1}V(\delta)$.

Similarly, we extend the contractible preimages of $T_2$ in such a way that the global cyclic order on all preimages corresponds, {\it via} $\psi$, to the one of $T_1$ --- this is possible only under the assumption that $\xW(T_1)=\xW(T_2)$ --- and we perform an isotopy on $T_2$ so $T_2\cap B_*=T_1\cap B_*$ with ribbon disks identified {\it via} $\psi$.
In particular, ribbon disk ends of $T_1$ and $T_2$ coincide.
Outside $B_*$, $T_1$ and $T_2$ are now both disjoint unions of embedded 3--balls which associate pairwise all the ribbon disk ends in the same way.
We put them in a position in $B_C:=S^4\setminus\mathring{B}_*\cong B^4$ so that they project onto disjoint embedded 3--balls in $B^3$.
The fact that each 3--ball is twice attached by disks to $\p B_*$ may introduce several wens, but since all ribbon disk ends can be oriented coherently using the co-orientation, for instance, of $T_1$, there are an even number of them on each component and, as we shall see in the next section, they cancel pairwise.
The projections are then the tubular neighborhoods of two string links in $B^3$ on which the local isotopy shown in Figure \ref{fig:V_isotopy} allows us to perform crossing changes.
It follows that $T_2\cap B_C$ can be transformed into $T_1\cap B_C$, so $T_1=T_2$ in $\rT$. 
The map $\xW$ is then injective.
\end{proof}

Welded tori are then a faithfull combinatorial description of ribbon tori.
Since a contraction-to-the-core inverse shows easily that the map $\xB:\wT\to\bT$ is surjective, broken tori are also in one-to-one correspondence with ribbon and welded tori.

\begin{prop}
  The map $\xT$ descends to a well-defined surjective map $\Tube:\wK\to\rK$. 
\end{prop}
\begin{proof}
  It is sufficient to prove that any two broken tori which differ by a broken Reidemeister move only are sent through $\xT$ to ribbon tori which can be realized with same boundary (but different fillings). This is done in Figure \ref{fig:FillChg_RI}, \ref{fig:FillChg_RIb}, \ref{fig:FillChg_RII} and \ref{fig:FillChg_RIII}. There, the upper lines give 4--dimensional movies of filled tubes which project along the height axis ---time becomes then the new height
  parameter--- to the pieces of broken torus diagram shown on Figure \ref{fig:BrokenReid} as the left hand sides of the broken moves; and the lower lines give another fillings of the very same tubes, which, after a suitable deformation, project along the height axis to the right hand sides of the same broken moves. Movies are hence read from left to right, one line out of two. For instance, line 1 continues on line 3. For Reidemeister moves II and III, we have dropped the left--right motion of the disks since it plays no role, complicates uselessly the pictures and can be easily added in the reader's mind.

Surjectivity is immediate by Proposition \ref{prop:ToriBij}.
\end{proof}

Now, we can address the question of the Tube map injectivity.
It is directly checked that classical Reidemeister moves on welded tori are in one-to-one correspondance with broken Reidemeister moves on broken tori. It follows that the notions of welded and broken knots do coincide.
So, basically, it remains to understand whether Reidemeister moves on broken tori are sufficient to span the whole $\p$--equivalence relation on $\rT$.
As noticed in \cite{Winter}, the answer is no.
Indeed, since it forgets everything about the filling, the $\p$--equivalence is insensitive to a change of co-orientation. However, as noted in Remark \ref{rem:CrossChg}, it is not sufficient to reverse the orientation of a welded torus for its image under $\xB$ to be identical but with reversed co-orientation: one also need to reverse all the crossing signs and add virtual crossings on both sides of each classical crossing.
\begin{nota}\label{nota:OR}
  For any virtual diagram $D$, we define $-D$ to be the virtual diagram obtained by reversing the orientation of $D$ and by $\overline{D}$ the diagram obtained by applying to each classical crossing of $D$ the operation described in Figure \ref{fig:SignRev}.
These operations on $\vD$ factor through the quotients $\wT$ and $\wK$.
From the welded Gauss diagram point of view, the second is nothing but to the reversal of all signs.\\
We define a \emph{global reversal} as the simultaneous reversing of both the orientation and the signs; it is denoted by glRev in Figure \ref{fig:Summary}.
\end{nota}
\begin{figure}
  \[
  \dessinH{1cm}{Cpos}\ \longrightarrow\ \dessinH{1cm}{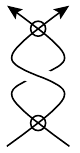}
  \hspace{1cm}
  \dessinH{1cm}{Cneg}\ \longrightarrow\ \dessinH{1cm}{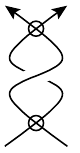}
  \]
  \caption{Signs reversal on virtual diagrams}
  \label{fig:SignRev}
\end{figure}

It follows from the discussion above that:
\begin{prop}\label{prop:GlRevInv}
  The $\Tube$ map is invariant under global reversal, that is $\Tube(K)=\Tube(-\overline{K})$ for every welded knot $K$.
\end{prop}


\section{A digression on wens}
\label{sec:digression-wens}

In this section, we focus on a particular portion of ribbon torus-knot, called wen in the literature.
A detailed treatment of them can be found in \cite{KS}, \cite[Sec. 2.5.4]{WKO1} and \cite[Sec. 4.5]{WKO2}.
In general, a wen is an embedded Klein bottle cut along a meridional circle.
Locally, it can be described as a circle which makes a half-turn inside a 3--ball with an extra time parameter.
\begin{defi}
  A \emph{wen} is an embedded annulus in $S^4$ which can be locally parametrized as
\[
\big(\gamma_t(C),t\big)_{t\in[0,1]}\subset B^3\times [0,1]\subset S^4
\]
\noi where $C\subset B^3$ is a fixed circle and $(\gamma_t)_{t\in[0,1]}$ a path in $\Diff(B^3;\p B^3)$, the set of smooth diffeomorphisms of $B^3$ which fix the boundary, such that $\gamma_0$ is the identity and $\gamma_1$ a diffeomorphism which sends $C$ on itself but with reversed orientation.
\end{defi}
A wen by itself is rather pointless since the reparametrization $\varphi(x,y,z,t):=\big(\gamma_t^{-1}(x,y,z),t\big)$ sends it on a trivially embedded annuli and, reciprocally, any annuli can be reparametrized as a wen.
However, as a portion of a ribbon torus-knots, we shall be interested in modifying a wen without altering the rest of the ribbon torus-knot. In this prospect, we shall consider only isotopies of $S^4$ which fix the boundary components of wens. We shall say then that the wens are \emph{$\p$--isotopic}.

Generically, a wen projects in $B^3$ into one of the four pieces of non symmetric broken torus diagram shown in Figure \ref{fig:wens}.
\begin{figure}
  \[
\fdessin{4cm}{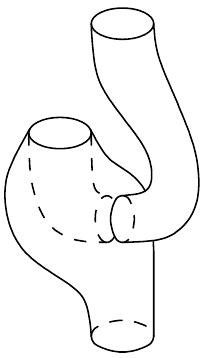}  \hspace{1cm} \fdessin{4cm}{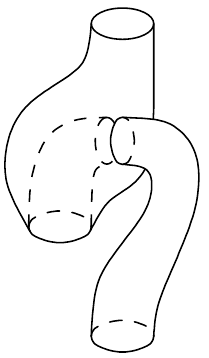} \hspace{1cm} \fdessin{4cm}{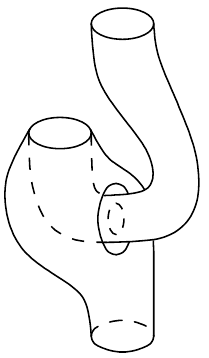} \hspace{1cm} \fdessin{4cm}{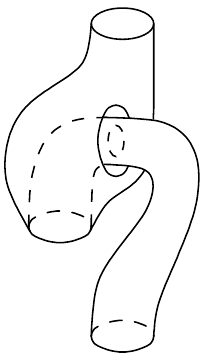}
  \]
  \caption{The four projections of wens}
  \label{fig:wens}
\end{figure}
However, we can speak of a wen unambiguously since the following result holds:
\begin{prop}\cite[Lem. 3.1]{KS}\label{prop:AllWensAreOne}
  Any two wens are $\p$--isotopic in $S^4$.
\end{prop}
It has the following corollary, depicted in the last four pictures of Figure \ref{fig:RI_isotopy}:
\begin{cor}\label{cor:square}
  The gluing of any two wens along one of their boundary circles is $\p$--isotopic to a trivially embedded cylinder.
\end{cor}
Although it is a direct consequence of T. Kanenobu and A. Shima's result, we shall sketch an alternative proof which strongly relies on the Smale conjecture proved by A. Hatcher in \cite{Hatcher}.
\begin{proof}
 We consider two paths $(\gamma^1_t)_{t\in[0,1]}$ and $(\gamma^2_t)_{t\in[0,1]}$ in $\Diff(B^3;\p B^3)$ which parametrizes, respectively,  the first and the second wen.
 By composing the second path with $\gamma^1_1$ and gluing it to the first one, we obtain a parametrization of the two wens as $\big(\gamma^0_t(C),t\big)_{t\in[0,1]}$ where $\gamma^0$ is a path in $\Diff(B^3;\p B^3)$ which starts at the identity and ends at $\gamma^2_1\circ\gamma^1_1$ which is a smooth diffeomorphism of $B^3$ sending $C$ to itself with same orientation.

To define an ambient isotopy of $S^4$, it would be more convenient that $\gamma^0$ also ends at the identity. This can be achieved by slightly packing down the two wens from the top, and reparametrizing the small piece of trivially embedded annulus that it creates.
In this prospect, we define now a path $(\gamma_t)_{t\in[0,1]}$ of elements in $\Diff(B^3;\p B^3)$ which starts at $\gamma^0_1=\gamma^2_1\circ\gamma^1_1$, ends at the identity and preserves globally $C$ for each $t\in[0,1]$.
Since the set $\Diff^+(S^1)$ of orientation preserving smooth diffeomorphisms of $S^1$ is path connected, there is a path between the restriction of $\gamma^0_1$ to $C$ and the identity on $C$. It can be used to deform $\gamma^0_1$ in a neigborhood of $C$ into a diffeomorphism $\gamma_{\frac{1}{3}}$ which fixes both $\p B^3$ and $C$.
Then, we choose an arbitrary path $\lambda$ from a point of $C$ to a point of $\p B^3$. Possibly up to an infinitesimal deformation of $\gamma_{\frac{1}{3}}$, we may assume that $\lambda$ and $\gamma_{\frac{1}{3}}(\lambda)$ meet only at their extremities, so $\lambda\cup\gamma_{\frac{1}{3}}(\lambda)$ is a closed curve embedded in $B^3$. Let $D$ be any disk in $B^3$ bounded by this curve. The diffeomorphism $\gamma_{\frac{1}{3}}$ can now be deformed in a neighborhood of $D$ so the image of
$\lambda$ is pushed along $D$ to $\lambda$. It results a diffeomorphism which fixes $\p B^3$, $C$ and $\lambda$. 
We can deform it furthermore so we obtain a diffeomorphism $\gamma_{\frac{2}{3}}$ which fixes a tubular neighborhood $N$ of $C\cup\lambda\cup\p B^3$.
But $S^3\setminus N$ is a solid torus --- gluing another 3--ball $B$ to $B^3$ gives a 3--sphere, and the union of $B$ with $N$ is an unknotted torus whose complement in the 3--sphere is also a torus --- and since $\Diff(S^1\times D^2;\p S^1\times D^2)$ is contractible, see formulation (9) of the Smale conjecture in \cite[Appendix]{Hatcher}, there is path from $\gamma_{\frac{2}{3}}$ to the identity.

After being packed down, the two wens are now reparametrized as $\big(\gamma^0_{\frac{t}{1-\e}}(C),t\big)_{t\in[0,1-\e]}\cup \big(\gamma_{\frac{t-1+\e}{\e}}(C),t\big)_{t\in[1-\e,1]}$, for some small $\e>0$.
This defines a loop of smooth diffeomorphisms of $B^3$ which fixes $\p B^3$.
But since $\Diff(B^3;\p B^3)$ is contractible, see formulation (1) of the Smale conjecture in \cite[Appendix]{Hatcher}, it is homotopic to the trivial loop.
This provides an homotopy of $B^3\times [0,1]\subset S^4$ which deforms the two glued wens into a trivially embedded annulus.
Since it fixes the boundary of $B^3\times [0,1]$, it can be extended by identity to an homotopy of $S^4$. 
\end{proof}

Now we enumerate three consequences of Corollary \ref{cor:square}.
\begin{description}
\item[Untwisting twisted 3--balls] Wens are no obstruction for a torus in $S^4$ to admit a ribbon filling. Indeed, one can similarly flip disks instead of circles, and all the arguments above apply to wens filled in this way.
This is used in the proof of Proposition \ref{prop:ToriBij}: when rectifying 3--balls so they project properly in $B^3$, some twisting may arise near the fixed ribbon disk ends. This can be untwisted using filled wens, and those cancel pairwise thanks to the filled version of Corollary \ref{cor:square}.
\item[Reidemeister I move for ribbon surfaces] Since it cannot be interpreted in terms of flying rings, it may be worthwhile to pay more attention to move RI on broken torus diagrams. So, besides its realization as a change of filling given in Figure \ref{fig:FillChg_RI} and \ref{fig:FillChg_RIb}, we provide in Figure \ref{fig:RI_isotopy} an isotopy in $S^4$, seen from the (non symmetric) broken torus diagram point of view, which realizes a move RI with the help of Proposition
  \ref{prop:AllWensAreOne}. The left handside of move RI in Figure \ref{fig:BrokenReid} can indeed be seen to be isotopic to the last picture in Figure \ref{fig:RI_isotopy} by pulling one singular circle next to the other. Note that the last five pictures can be replaced by a single use of Corollary \ref{cor:square}.
\item[Commutation of the $\Tube$ map with the orientation reversals] A consequence of Corollary \ref{cor:square} is that a trivially embedded torus can be inverted in $S^4$. Indeed, one can choose any portion of the torus and create {\it ex nihilo} a pair of wens. Then, one of the wens can travel all along the torus, inverting on its way the orientation. Once done, wens can cancel back one each other.\footnote{An illustration can be found at \url{https://www.youtube.com/watch?v=kQcy5DvpvlM}}

The same isotopy can be performed even if the torus is knotted, but then the passage of the wen
  will modify the neigborhood of each ribbon disk. Indeed, as explained in \cite[Sec. 4.5]{WKO2}, it will reverse the sign associated to each ribbon disk. Starting with $\Tube(K)$, for $K$ any welded knot, the isotopy will hence end at $-\Tube(\overline{K})$.
The torus inversion process in $S^4$ induces then the following:
\end{description}
\begin{prop}\label{prop:Inv1}
  For every welded knot $K$, $\Tube(K)=-\Tube(\overline{K})$.
\end{prop}
Together with Proposition \ref{prop:GlRevInv}, it implies: 
\begin{prop}\label{prop:Inv2}
  For every welded knot $K$, $\Tube(-K)=\Tube(\overline{K})=-\Tube(K)$.
\end{prop}
\begin{question}
  Is there a one-to-one correspondence between ribbon torus-knots and welded knots up to global reversal? Equivalently, do the classical Reidemeister local moves and the co-orientation reversal generate all possible change of filling? 
\end{question}



\bibliographystyle{amsalpha}
\bibliography{OnTubeMap}
\addcontentsline{toc}{part}{Bibliography}

\begin{figure}[!h]
  \[
\vcenter{\hbox{\fbox{
  $\hspace{-.25cm}\begin{array}{c}
    \fdessin{2.4cm}{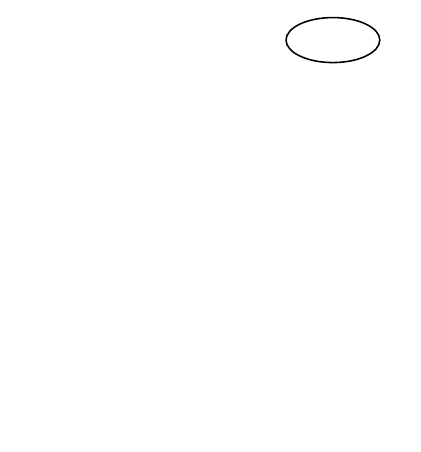}\\
    \uparrow\\
    \fdessin{2.4cm}{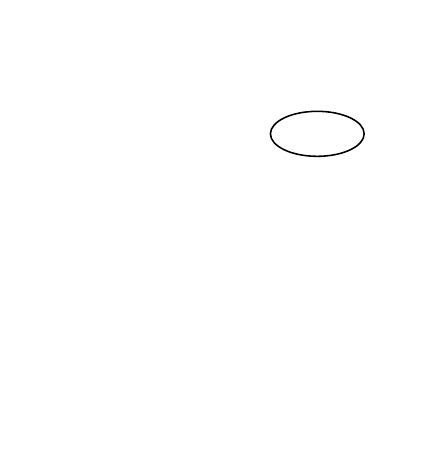}\\
    \uparrow\\
    \fdessin{2.4cm}{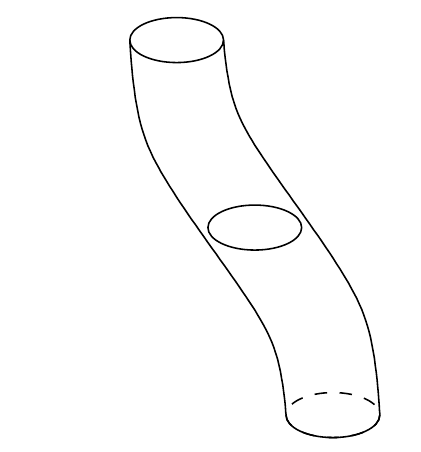}\\
    \uparrow\\
    \fdessin{2.4cm}{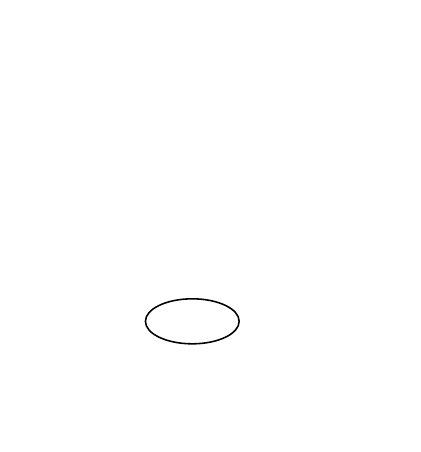}\\
    \uparrow\\
    \fdessin{2.4cm}{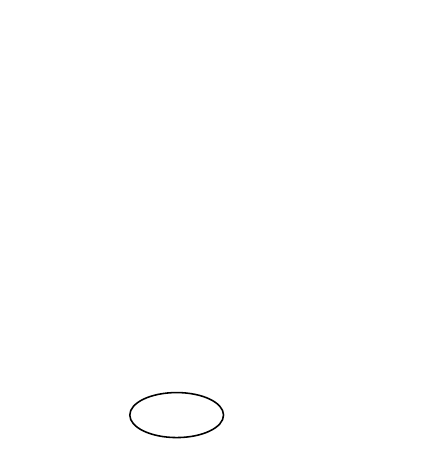}
  \end{array}\hspace{-.25cm}$
}}}
\rightarrow
\vcenter{\hbox{\fbox{
  $\hspace{-.25cm}\begin{array}{c}
    \fdessin{2.4cm}{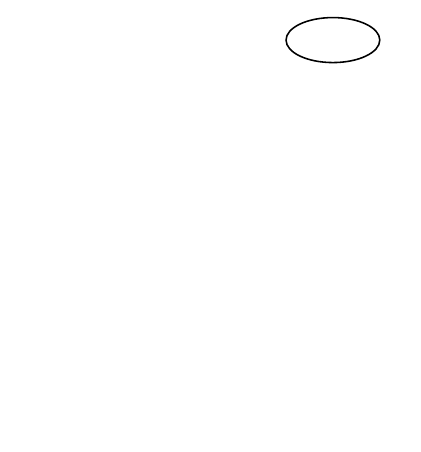}\\
    \uparrow\\
    \fdessin{2.4cm}{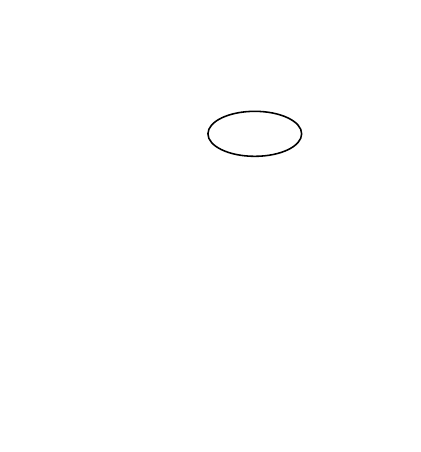}\\
    \uparrow\\
    \fdessin{2.4cm}{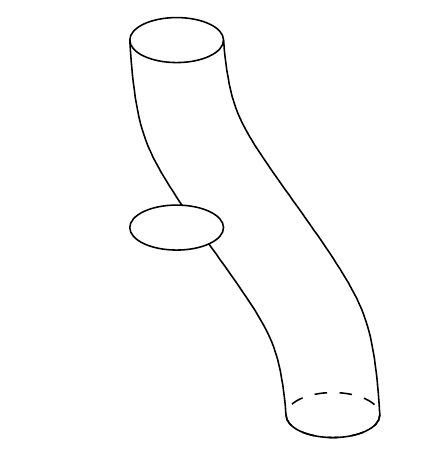}\\
    \uparrow\\
    \fdessin{2.4cm}{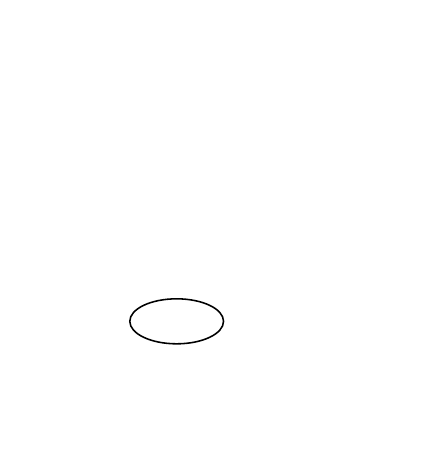}\\
    \uparrow\\
    \fdessin{2.4cm}{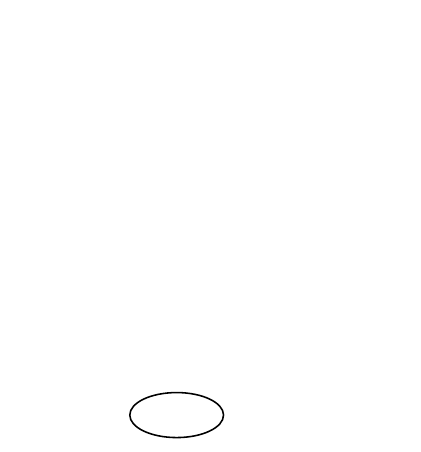}
  \end{array}\hspace{-.25cm}$
}}}
\rightarrow
\vcenter{\hbox{\fbox{
  $\hspace{-.25cm}\begin{array}{c}
    \fdessin{2.4cm}{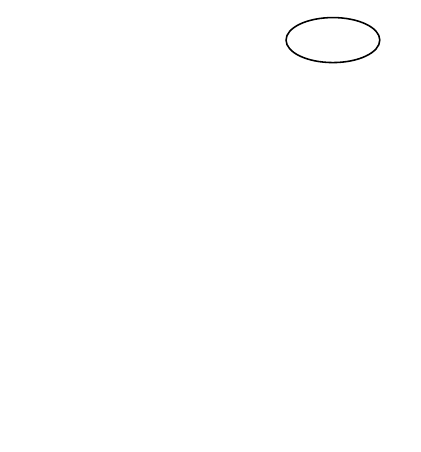}\\
    \uparrow\\
    \fdessin{2.4cm}{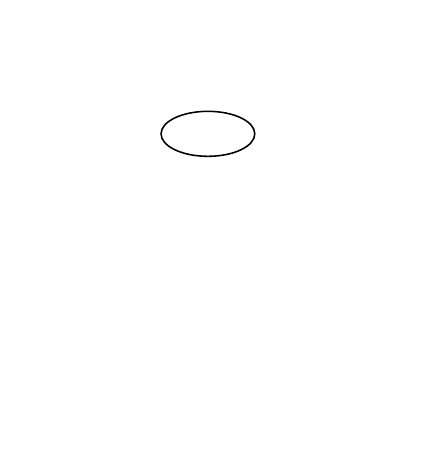}\\
    \uparrow\\
    \fdessin{2.4cm}{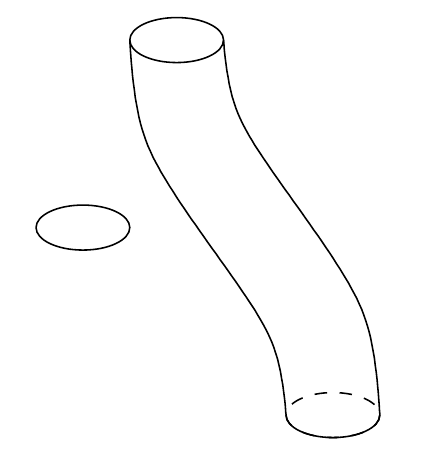}\\
    \uparrow\\
    \fdessin{2.4cm}{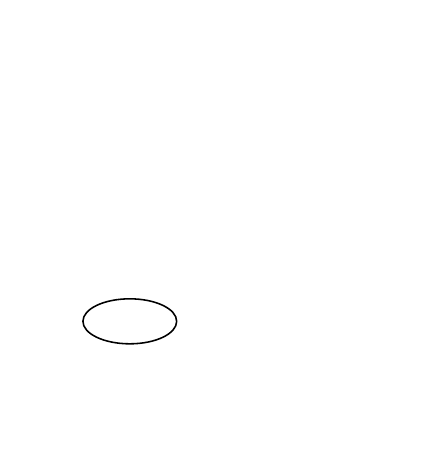}\\
    \uparrow\\
    \fdessin{2.4cm}{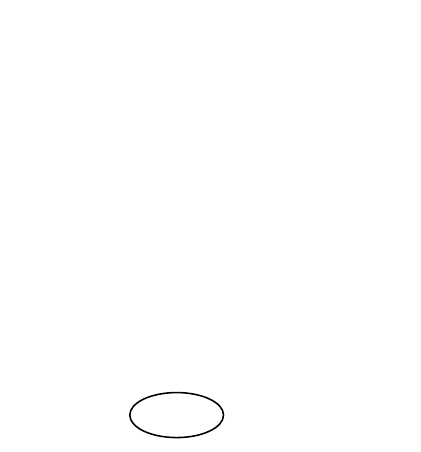}
  \end{array}\hspace{-.25cm}$
}}}
\rightarrow
\vcenter{\hbox{\fbox{
  $\hspace{-.25cm}\begin{array}{c}
    \fdessin{2.4cm}{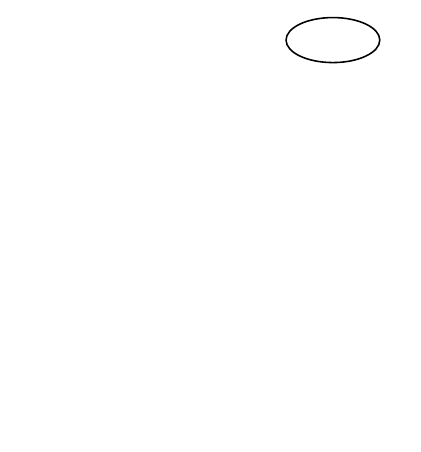}\\
    \uparrow\\
    \fdessin{2.4cm}{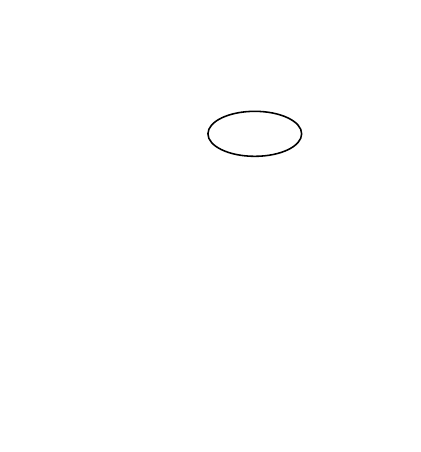}\\
    \uparrow\\
    \fdessin{2.4cm}{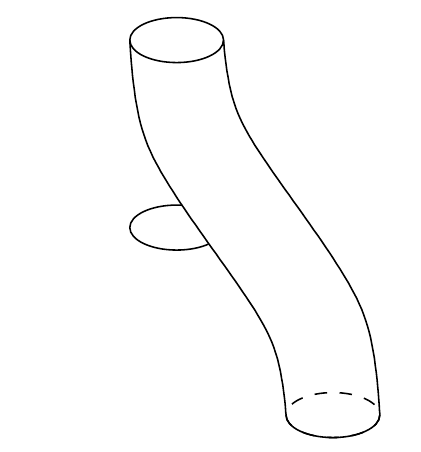}\\
    \uparrow\\
    \fdessin{2.4cm}{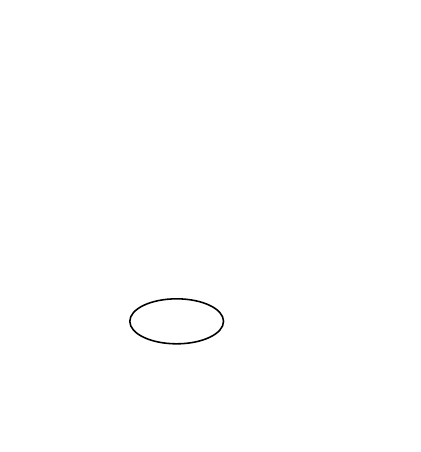}\\
    \uparrow\\
    \fdessin{2.4cm}{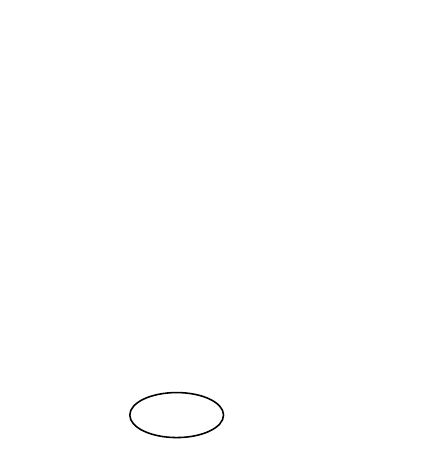}
  \end{array}\hspace{-.25cm}$
}}}
\rightarrow
\vcenter{\hbox{\fbox{
  $\hspace{-.25cm}\begin{array}{c}
    \fdessin{2.4cm}{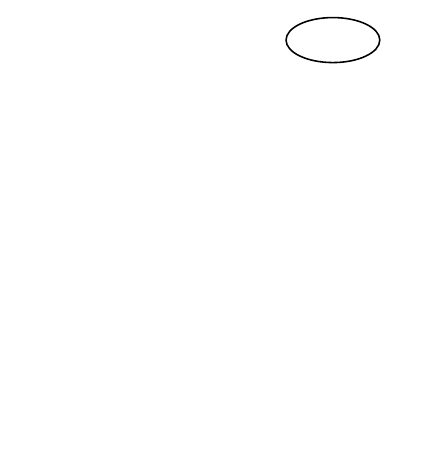}\\
    \uparrow\\
    \fdessin{2.4cm}{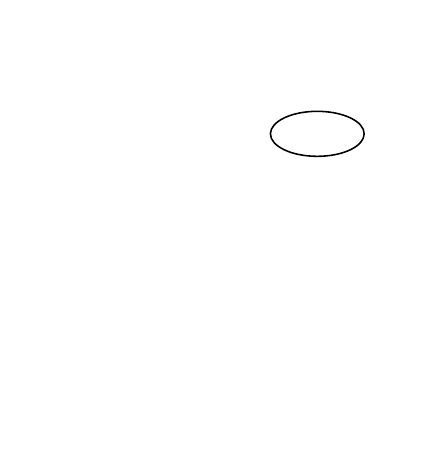}\\
    \uparrow\\
    \fdessin{2.4cm}{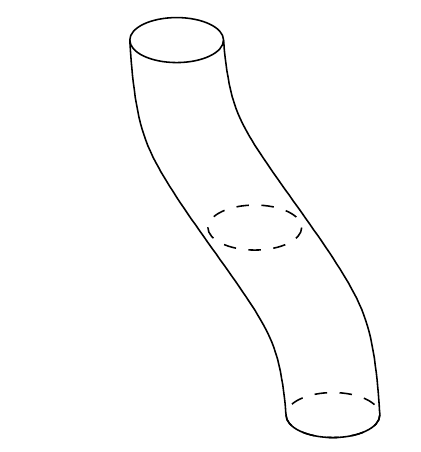}\\
    \uparrow\\
    \fdessin{2.4cm}{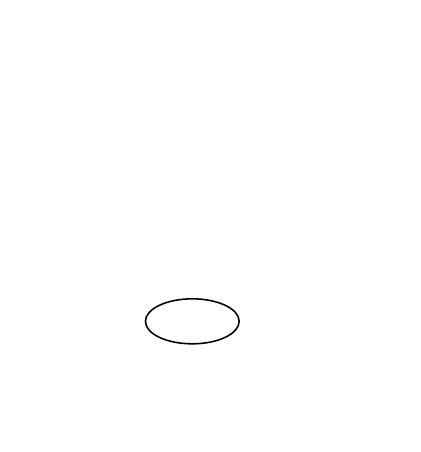}\\
    \uparrow\\
    \fdessin{2.4cm}{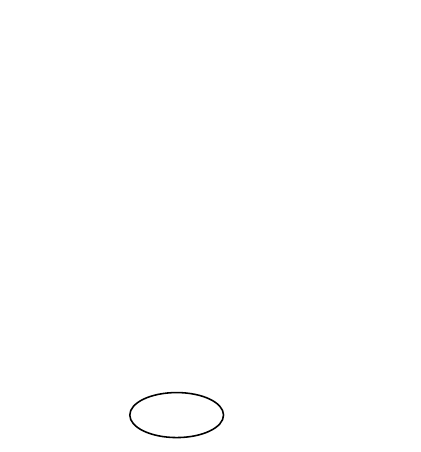}
  \end{array}\hspace{-.25cm}$
}}}
  \]
\[
\xrightarrow[\hspace{1cm}\textrm{\tiny isotopy}\hspace{1cm}]{}
\]
\caption{Isotopy for the virtual move}\label{fig:V_isotopy}
\end{figure}

\begin{figure}
\[
\begin{array}{ccccccc}
  &\fbox{$\dessin{2.5cm}{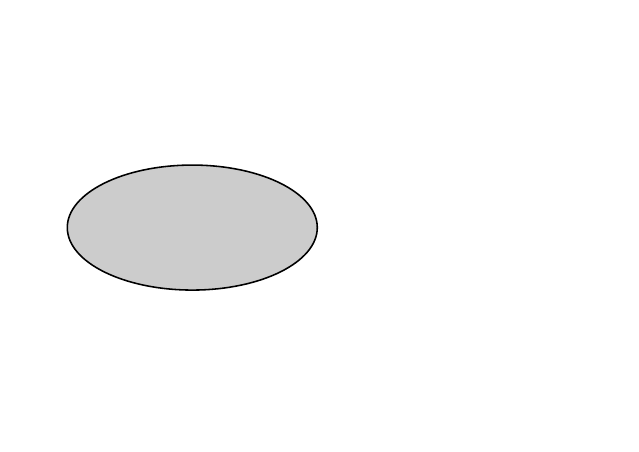}$}
  &\rightarrow&
  \fbox{$\dessin{2.5cm}{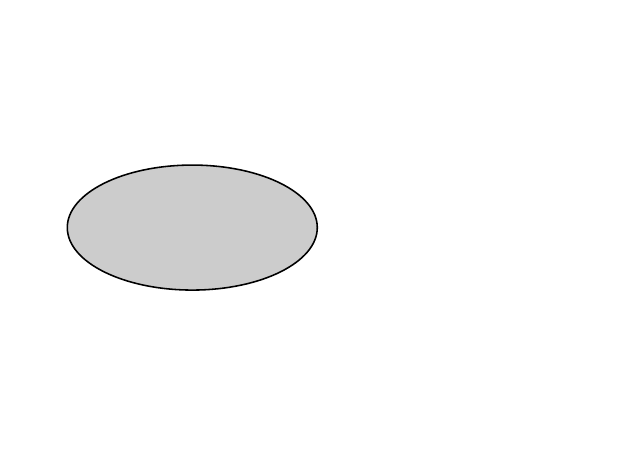}$}
  &\rightarrow&
  \fbox{$\dessin{2.5cm}{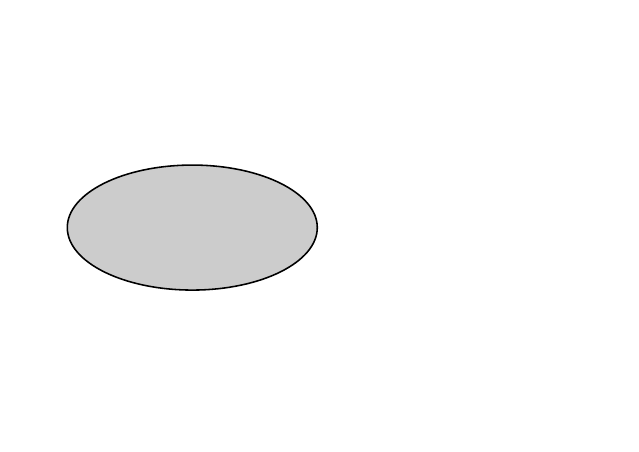}$}
  &\rightarrow\\[2cm]
  &\fbox{$\dessin{2.5cm}{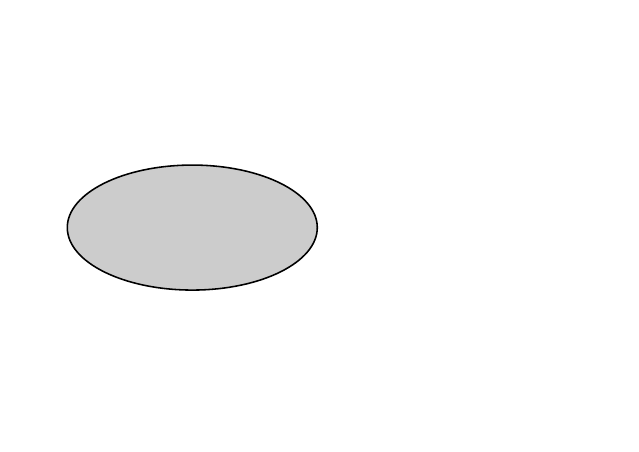}$}
  &\rightarrow&
  \fbox{$\dessin{2.5cm}{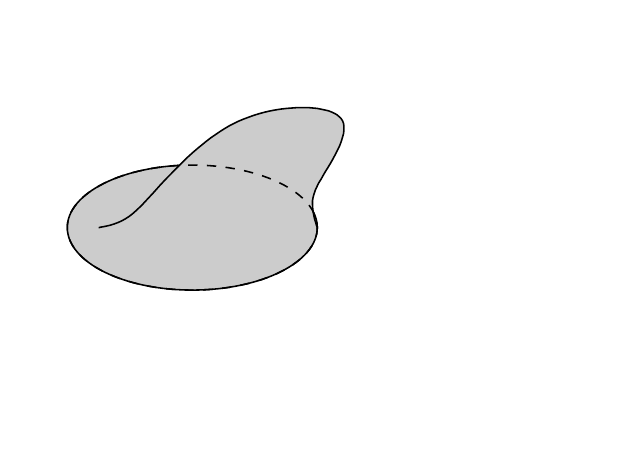}$}
  &\rightarrow&
  \fbox{$\dessin{2.5cm}{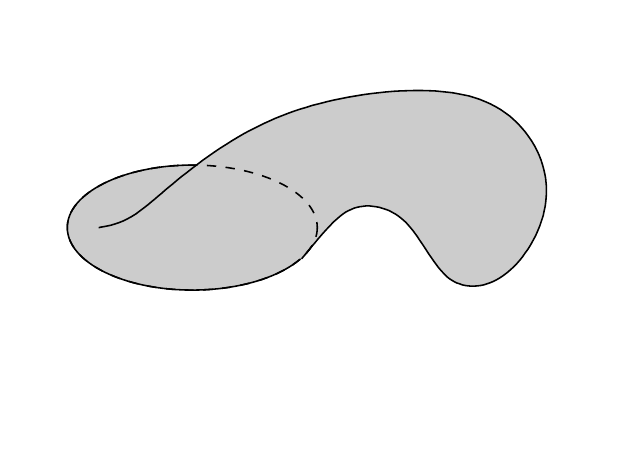}$}
  &\rightarrow\\
    &&&&&&\\
   \hline
   \hline
   &&&&&&\\
   \rightarrow&
   \fbox{$\dessin{2.5cm}{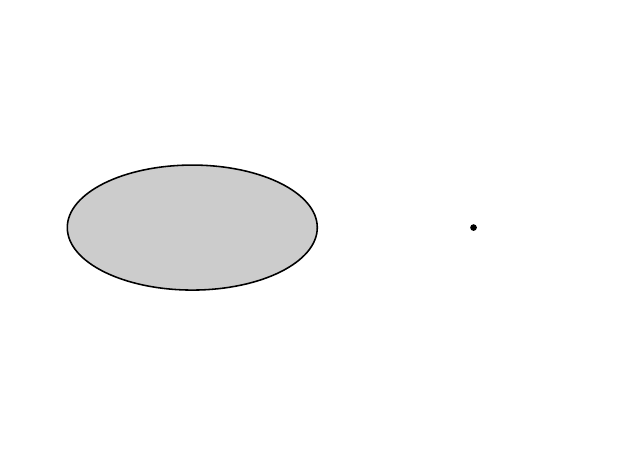}$}
  &\rightarrow&
  \fbox{$\dessin{2.5cm}{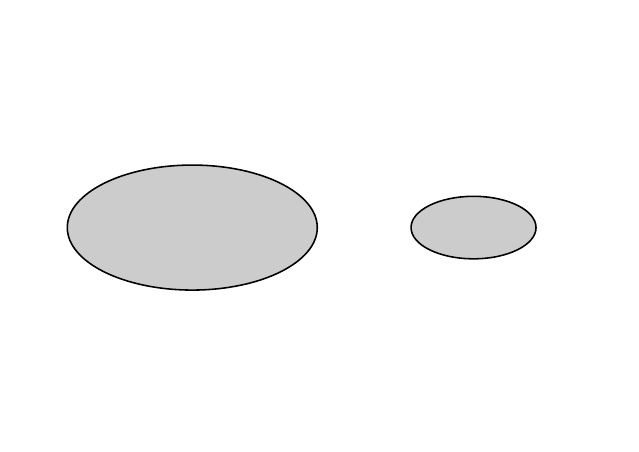}$}
  &\rightarrow&
  \fbox{$\dessin{2.5cm}{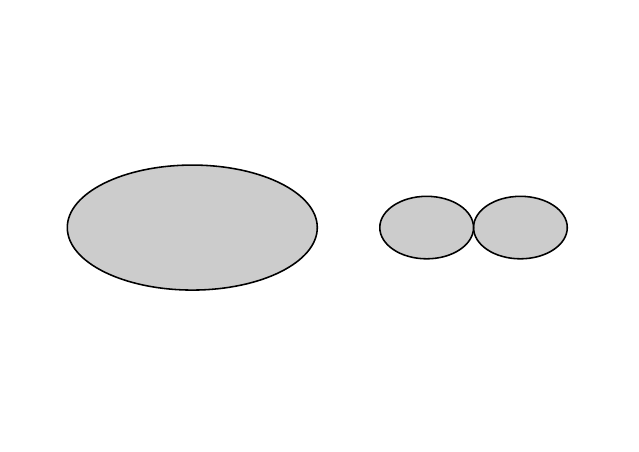}$}
  &\rightarrow\\[2cm]
     \rightarrow&
   \fbox{$\dessin{2.5cm}{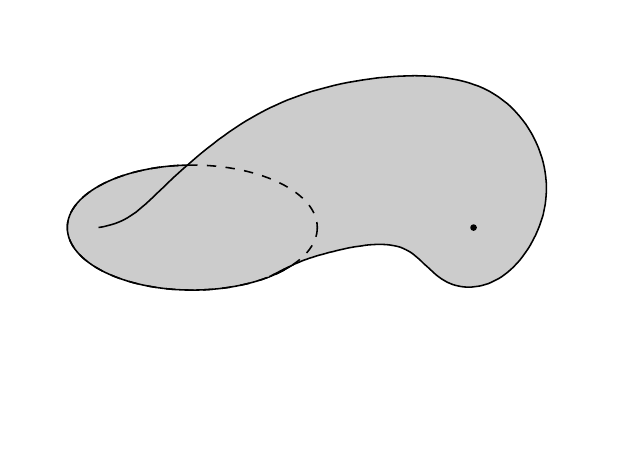}$}
  &\rightarrow&
  \fbox{$\dessin{2.5cm}{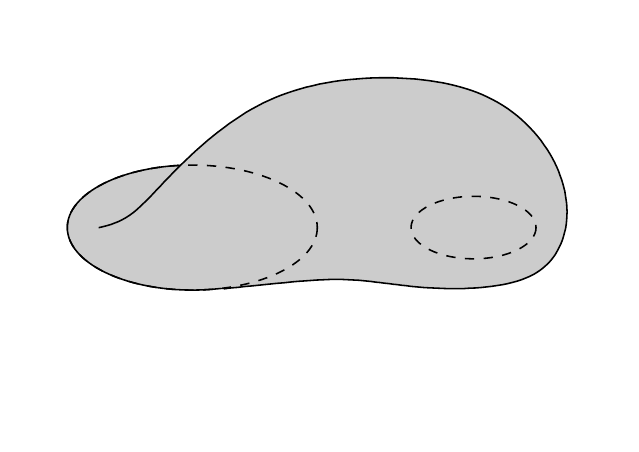}$}
  &\rightarrow&
  \fbox{$\dessin{2.5cm}{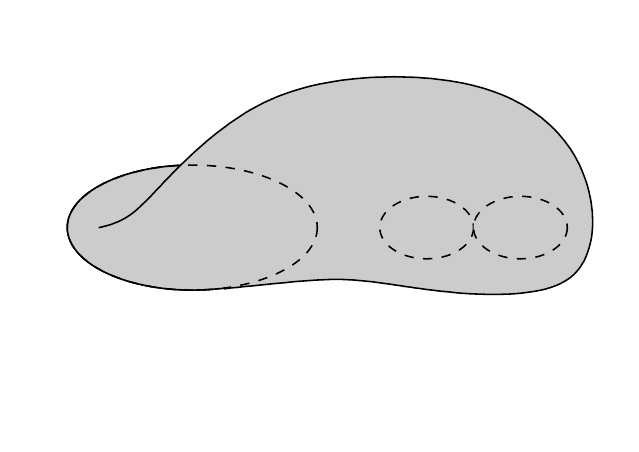}$}
  &\rightarrow\\
    &&&&&&\\
   \hline
   \hline
   &&&&&&\\
   \rightarrow&
   \fbox{$\dessin{2.5cm}{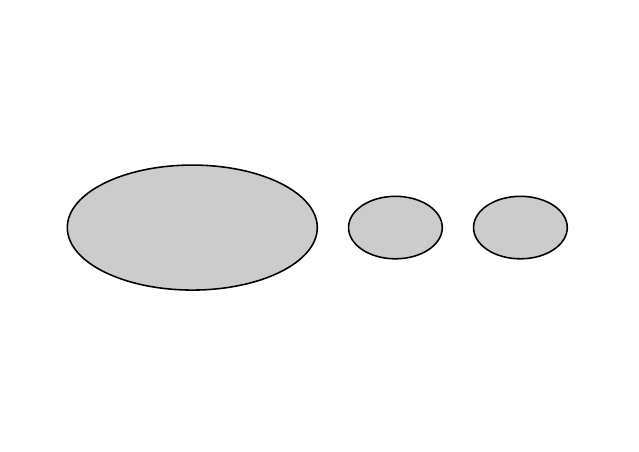}$}
  &\rightarrow&
  \fbox{$\dessin{2.5cm}{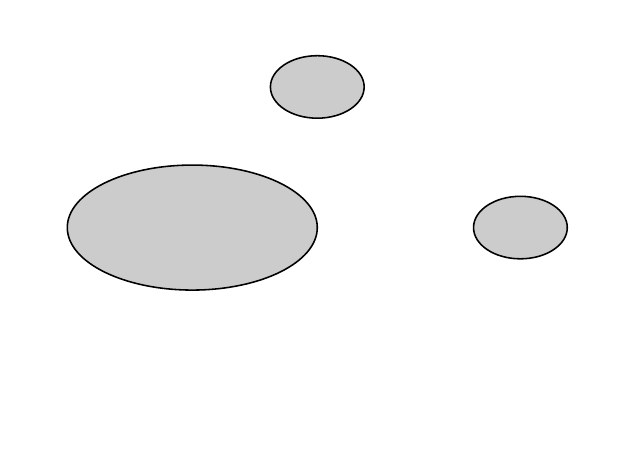}$}
  &\rightarrow&
  \fbox{$\dessin{2.5cm}{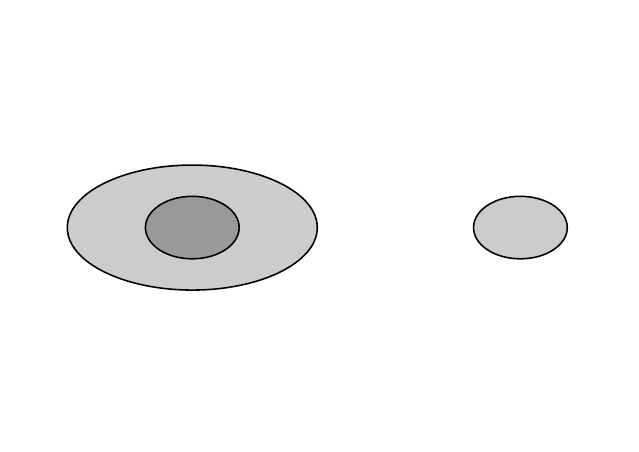}$}
  &\rightarrow\\[2cm]
     \rightarrow&
   \fbox{$\dessin{2.5cm}{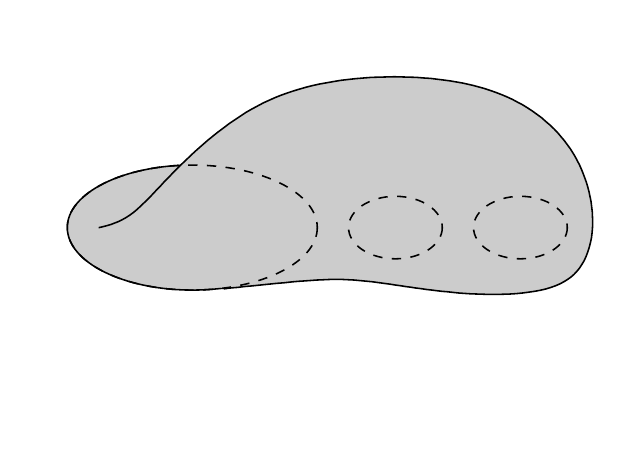}$}
  &\rightarrow&
  \fbox{$\dessin{2.5cm}{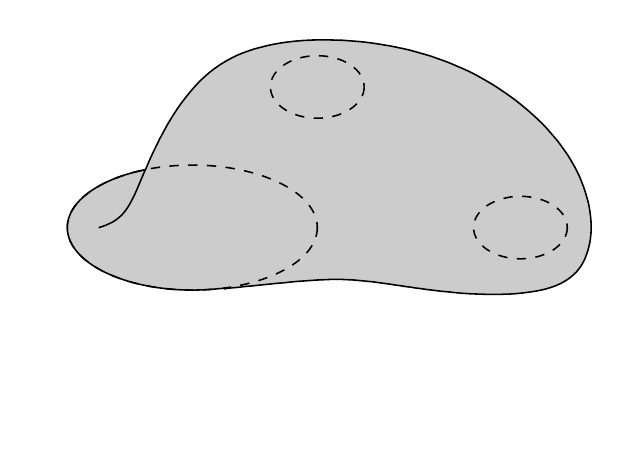}$}
  &\rightarrow&
  \fbox{$\dessin{2.5cm}{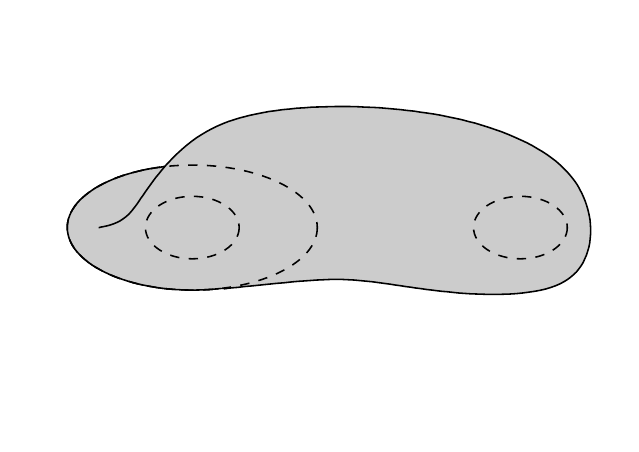}$}
  &\rightarrow
\end{array}
\]  
  \caption{RI seen as a filling change (part 1)}
  \label{fig:FillChg_RI}
\end{figure}
\begin{figure}
\[
\begin{array}{ccccccc}
  \rightarrow&\fbox{$\dessin{2.5cm}{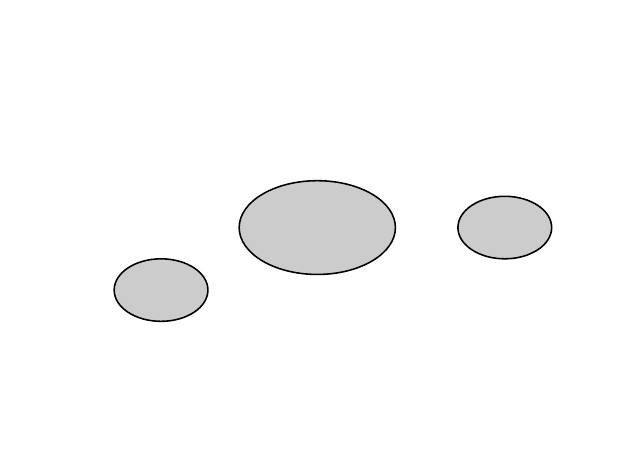}$}
  &\rightarrow&
  \fbox{$\dessin{2.5cm}{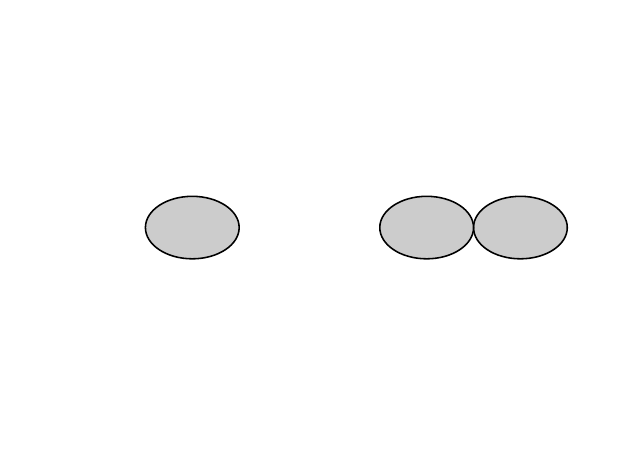}$}
  &\rightarrow&
  \fbox{$\dessin{2.5cm}{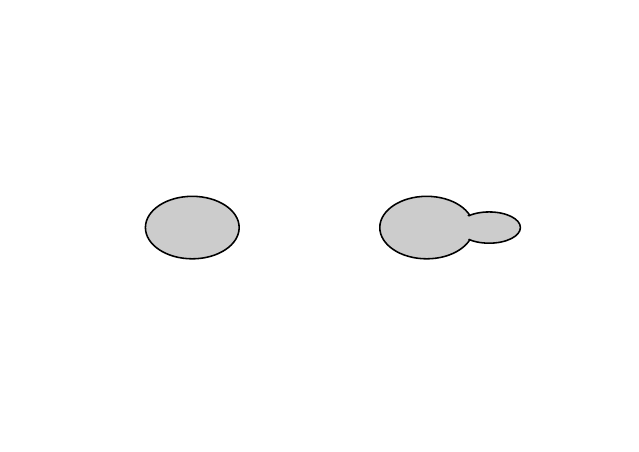}$}
  &\rightarrow\\[2cm]
  \rightarrow&\fbox{$\dessin{2.5cm}{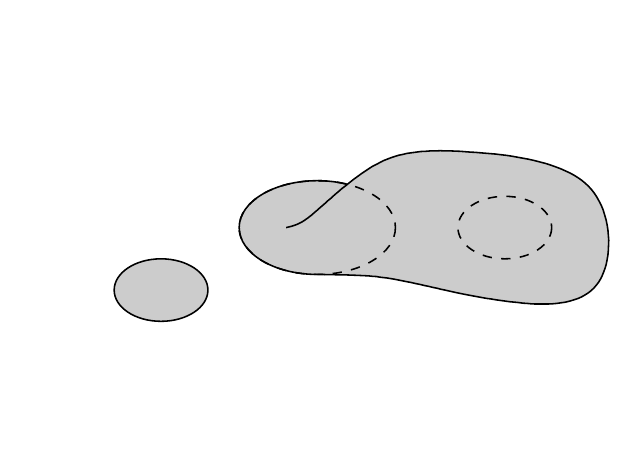}$}
  &\rightarrow&
  \fbox{$\dessin{2.5cm}{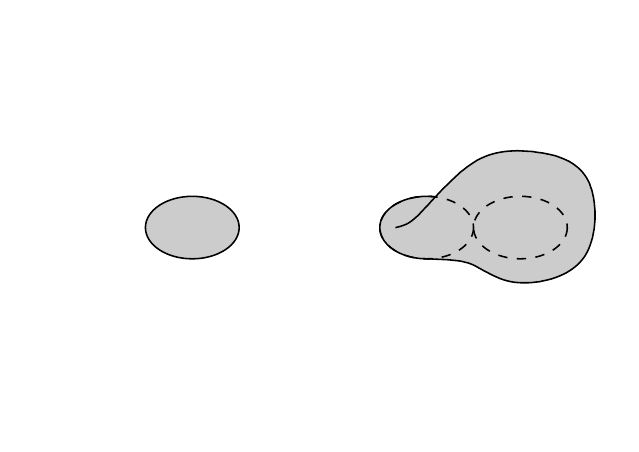}$}
  &\rightarrow&
  \fbox{$\dessin{2.5cm}{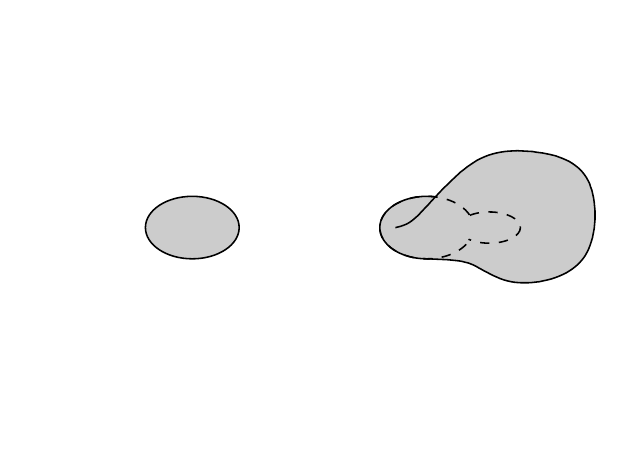}$}
  &\rightarrow\\
    &&&&&&\\
   \hline
   \hline
   &&&&&&\\
   \rightarrow&
   \fbox{$\dessin{2.5cm}{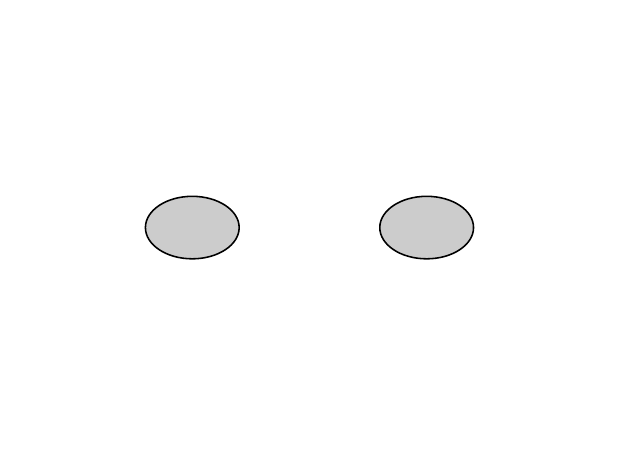}$}
  &\rightarrow&
  \fbox{$\dessin{2.5cm}{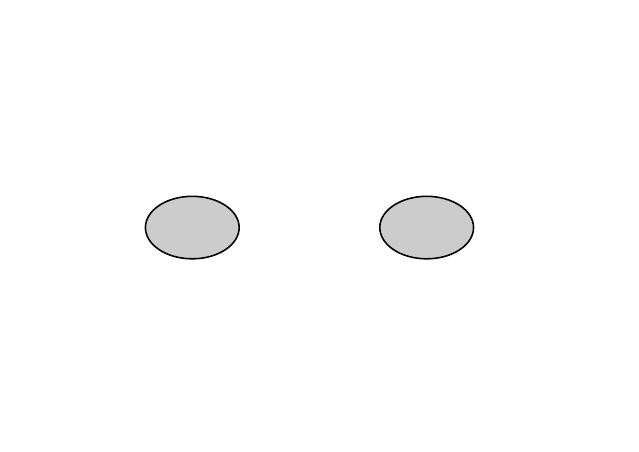}$}
  &\rightarrow&
  \fbox{$\dessin{2.5cm}{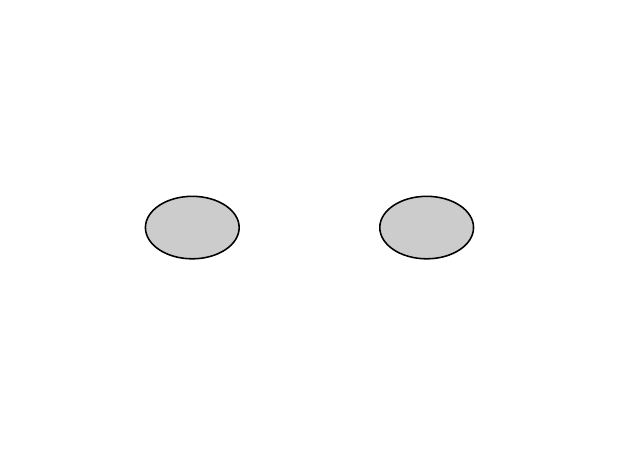}$}
  &\rightarrow\\[2cm]
     \rightarrow&
   \fbox{$\dessin{2.5cm}{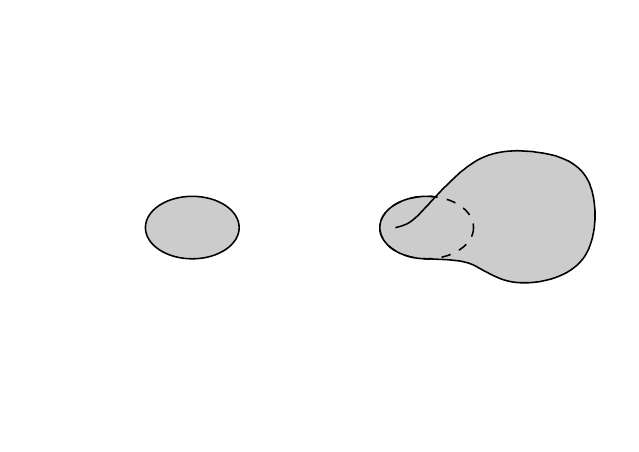}$}
  &\rightarrow&
  \fbox{$\dessin{2.5cm}{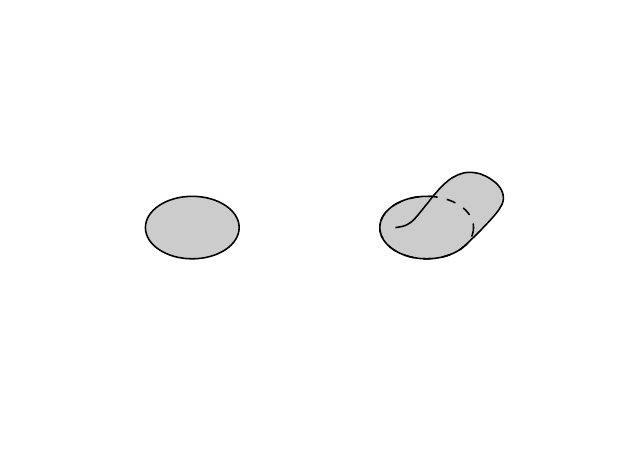}$}
  &\rightarrow&
  \fbox{$\dessin{2.5cm}{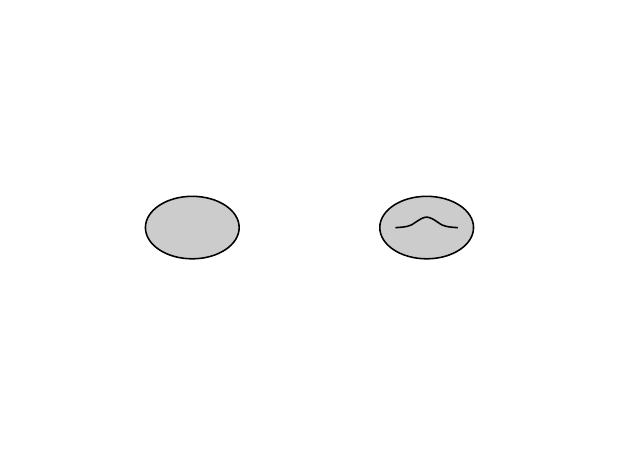}$}
  &\rightarrow\\
    &&&&&&\\
   \hline
   \hline
   &&&&&&\\
   \rightarrow&
   \fbox{$\dessin{2.5cm}{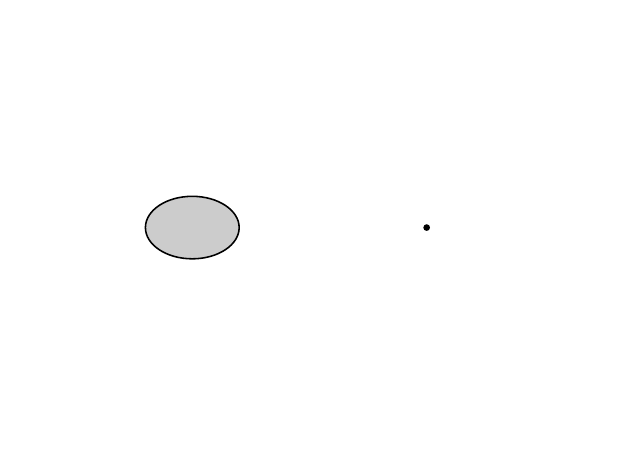}$}
  &\rightarrow&
  \fbox{$\dessin{2.5cm}{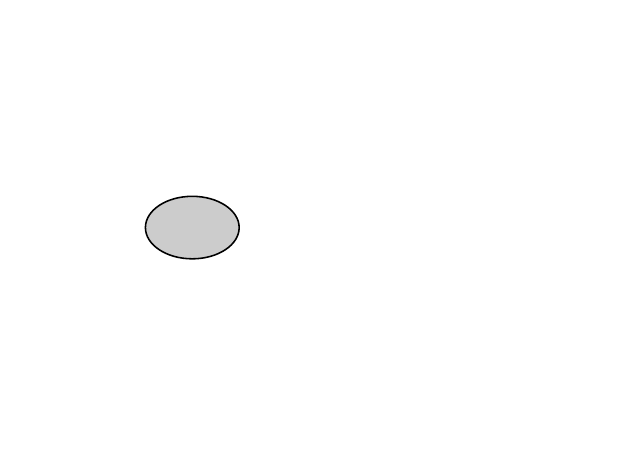}$}
  &&&\\[2cm]
     \rightarrow&
   \fbox{$\dessin{2.5cm}{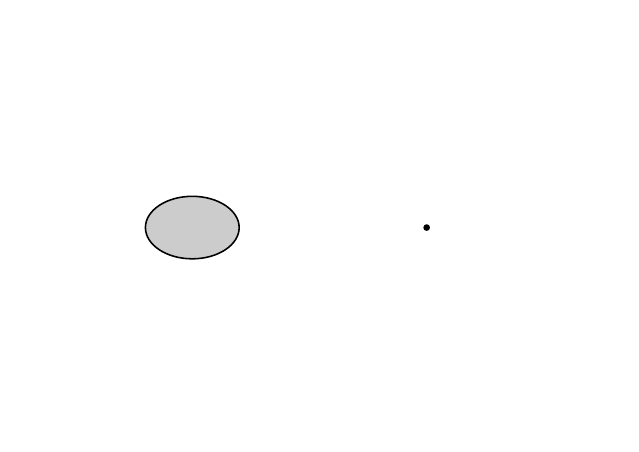}$}
  &\rightarrow&
  \fbox{$\dessin{2.5cm}{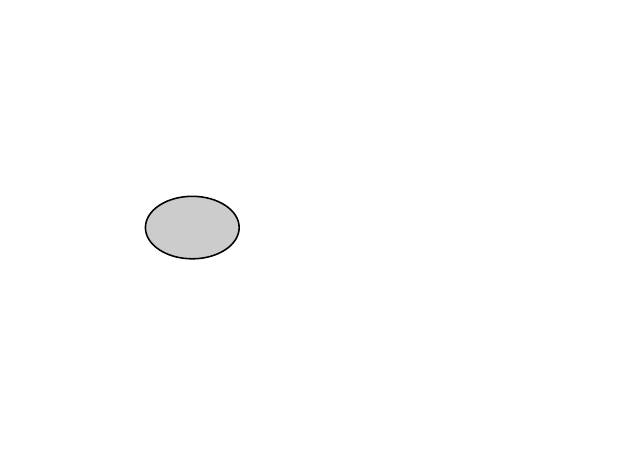}$}
  &&&
\end{array}
\]  
  \caption{RI seen as a filling change (part 2)}
  \label{fig:FillChg_RIb}
\end{figure}

\begin{figure}
  \[
  \begin{array}{ccccc}
    \phantom{\rightarrow}
    &\fbox{$\dessin{3cm}{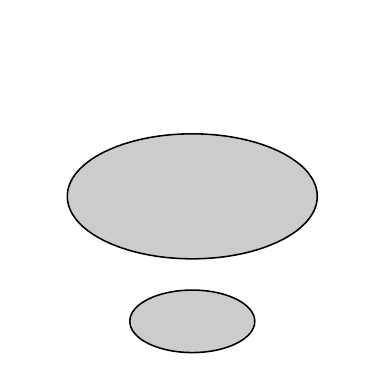}$}
    &\rightarrow&
    \fbox{$\dessin{3cm}{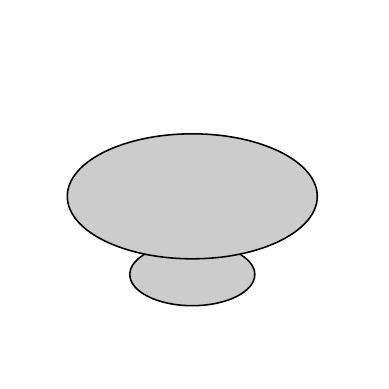}$}
    &\rightarrow\\[2cm]
    \phantom{\rightarrow}
    &\fbox{$\dessin{3cm}{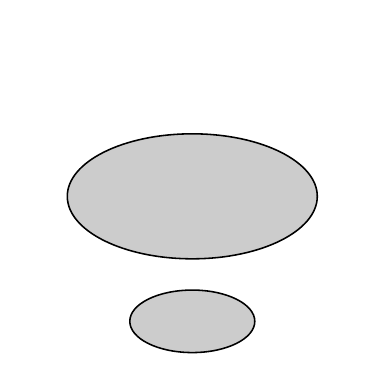}$}
    &\rightarrow&
    \fbox{$\dessin{3cm}{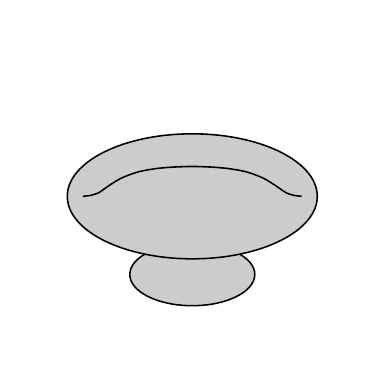}$}
    &\rightarrow
\end{array}
\]
\[
\begin{array}{ccccccc}
    &&&&&&\\[-.1cm]
   \hline
   \hline
   &&&&&&\\
   \rightarrow&
   \fbox{$\dessin{3cm}{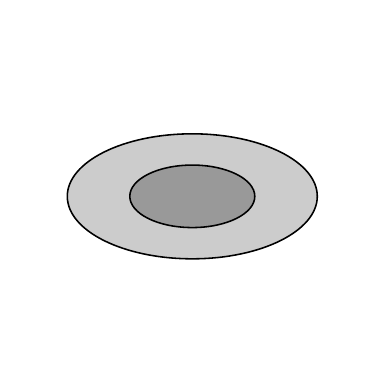}$}
    &\rightarrow&
    \fbox{$\dessin{3cm}{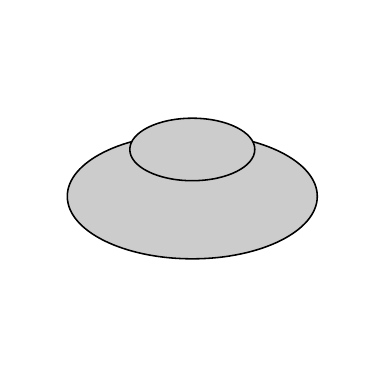}$}
    &\rightarrow&
    \fbox{$\dessin{3cm}{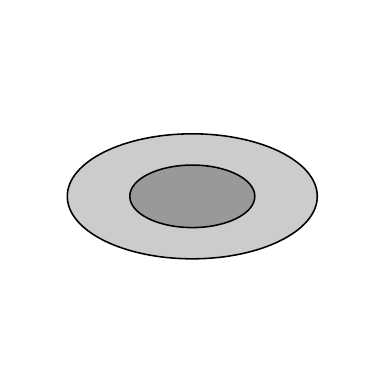}$}
    &\rightarrow\\[2cm]
    \rightarrow&\fbox{$\dessin{3cm}{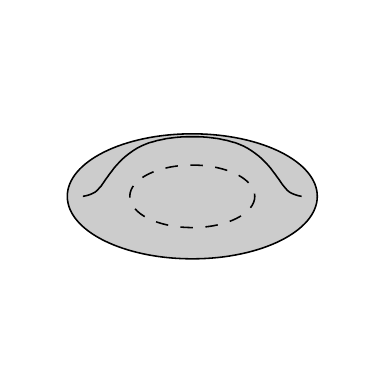}$}
    &\rightarrow&
    \fbox{$\dessin{3cm}{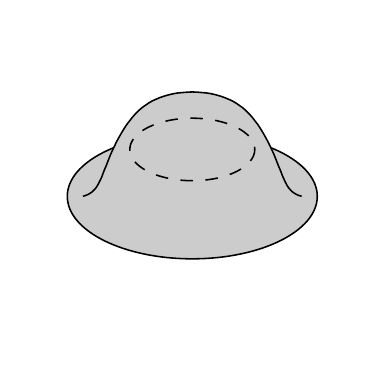}$}
    &\rightarrow&
    \fbox{$\dessin{3cm}{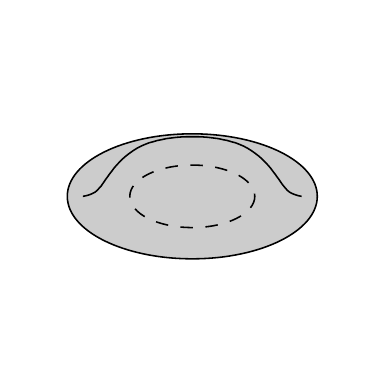}$}
    &\rightarrow\\
    &&&&&&\\
   \hline
   \hline
    &&&&&&\\[-.1cm]
\end{array}
\]
\[
\begin{array}{ccccc}
   \rightarrow&
   \fbox{$\dessin{3cm}{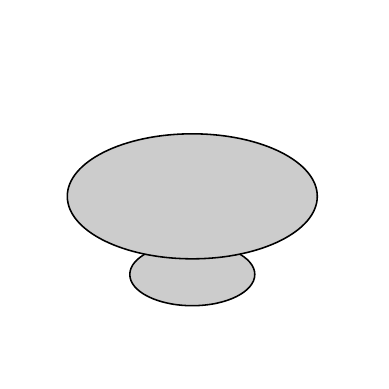}$}
   &\rightarrow&\fbox{$\dessin{3cm}{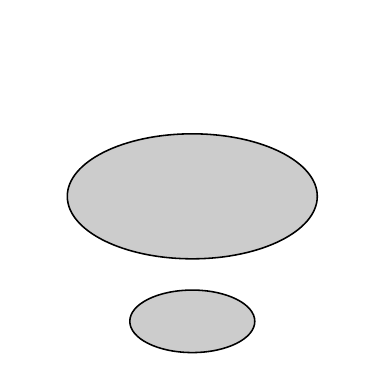}$}&\phantom{\rightarrow}\\[2cm]
   \rightarrow&
   \fbox{$\dessin{3cm}{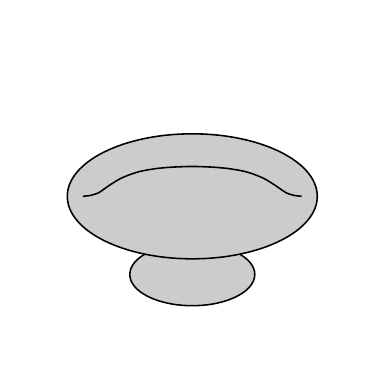}$}
   &\rightarrow&\fbox{$\dessin{3cm}{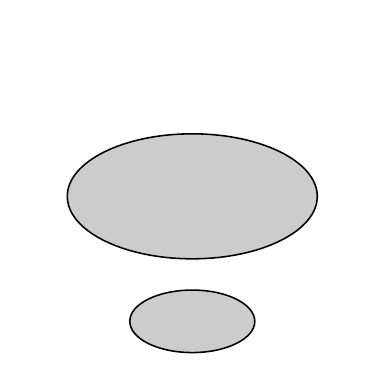}$}&\phantom{\rightarrow}
\end{array}
\]
  \caption{RII seen as a filling change}
  \label{fig:FillChg_RII}
\end{figure}


\begin{figure}
    \[
  \begin{array}{ccccccccc}
    &\fbox{$\dessin{3cm}{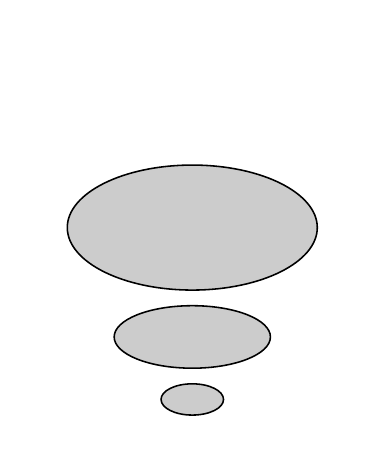}$}
    &\rightarrow&
    \fbox{$\dessin{3cm}{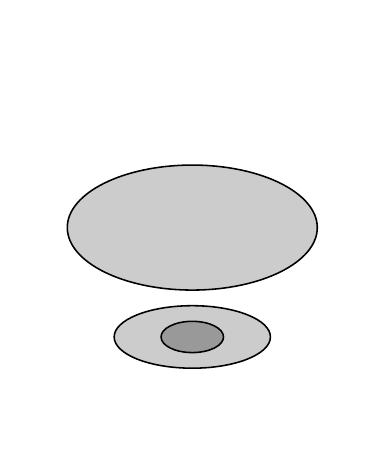}$}
    &\rightarrow&
    \fbox{$\dessin{3cm}{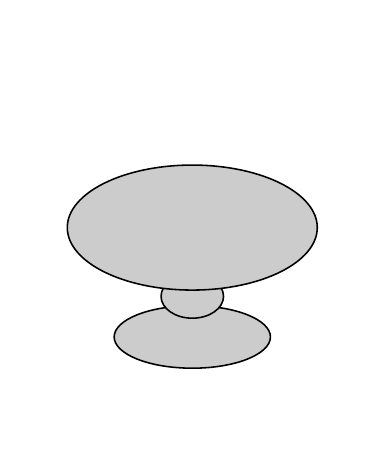}$}
    &\rightarrow&
    \fbox{$\dessin{3cm}{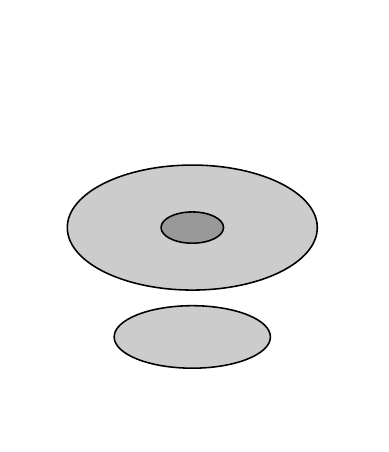}$}
    &\rightarrow\\[2cm]
    &\fbox{$\dessin{3cm}{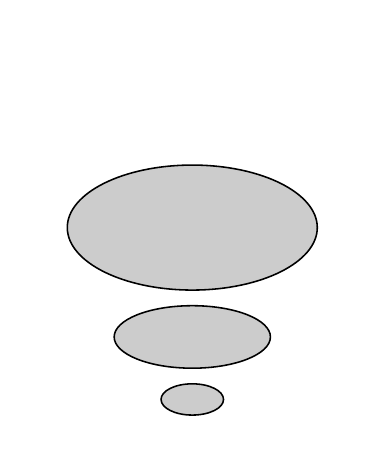}$}
    &\rightarrow&
    \fbox{$\dessin{3cm}{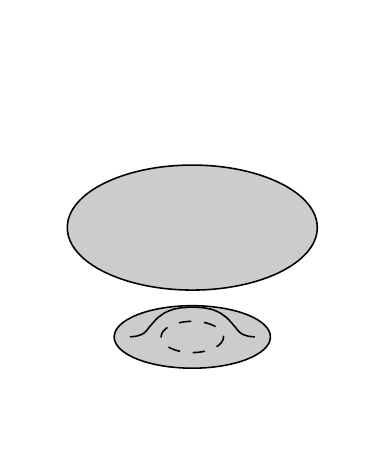}$}
    &\rightarrow&
    \fbox{$\dessin{3cm}{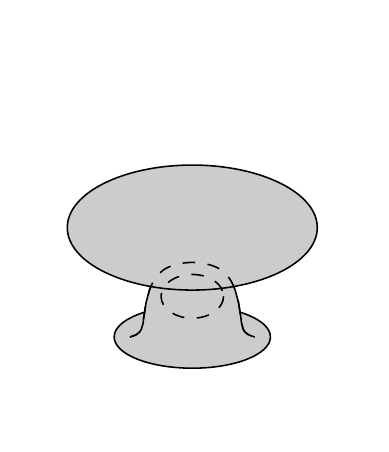}$}
    &\rightarrow&
    \fbox{$\dessin{3cm}{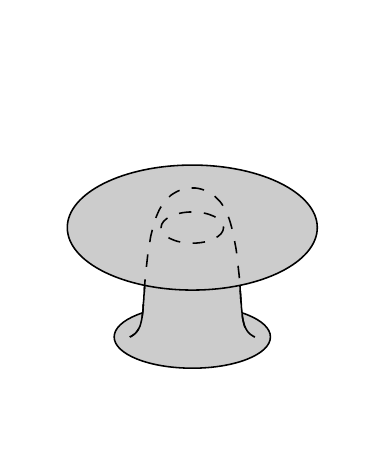}$}
    &\rightarrow\\
    &&&&&&&&\\
   \hline
   \hline
   &&&&&&&&\\
   \rightarrow&
   \fbox{$\dessin{3cm}{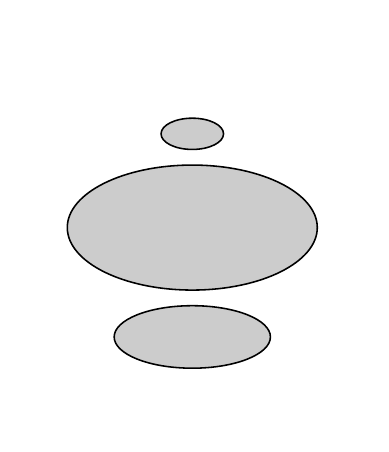}$}
    &\rightarrow&
    \fbox{$\dessin{3cm}{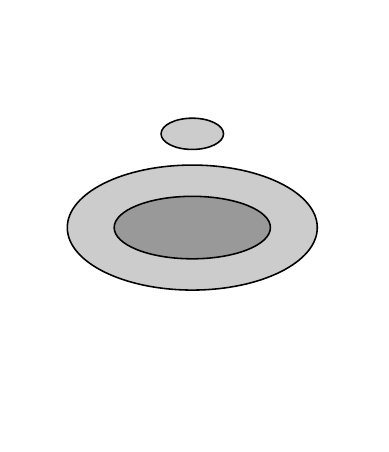}$}
    &\rightarrow&
    \fbox{$\dessin{3cm}{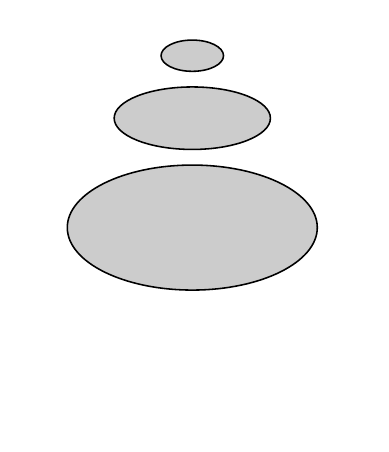}$}
    &\rightarrow&
    \fbox{$\dessin{3cm}{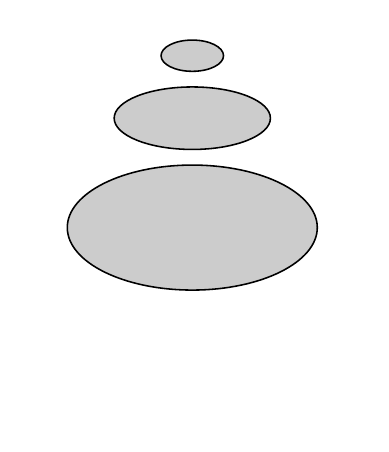}$}
    &\rightarrow\\[2cm]
    \rightarrow&\fbox{$\dessin{3cm}{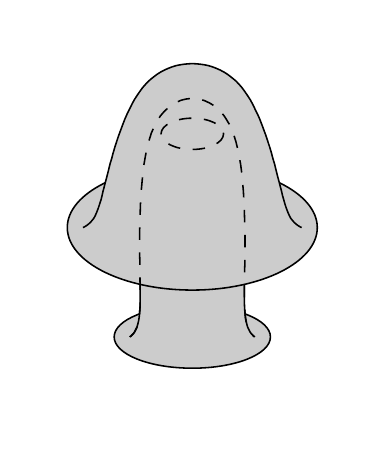}$}
    &\rightarrow&
    \fbox{$\dessin{3cm}{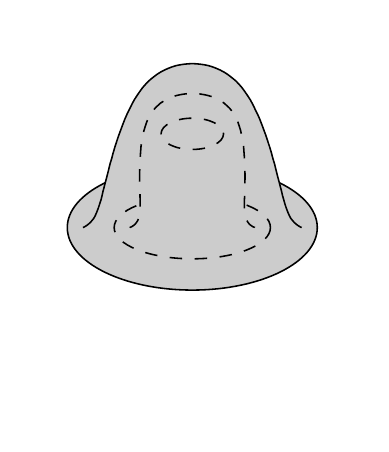}$}
    &\rightarrow&
    \fbox{$\dessin{3cm}{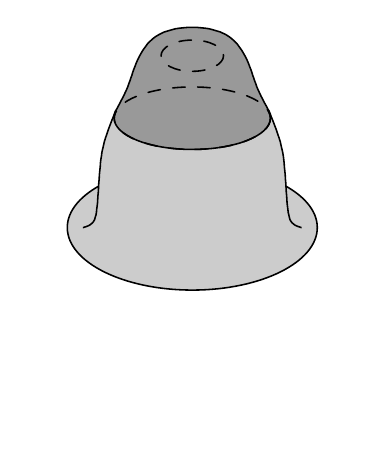}$}
    &\rightarrow&
    \fbox{$\dessin{3cm}{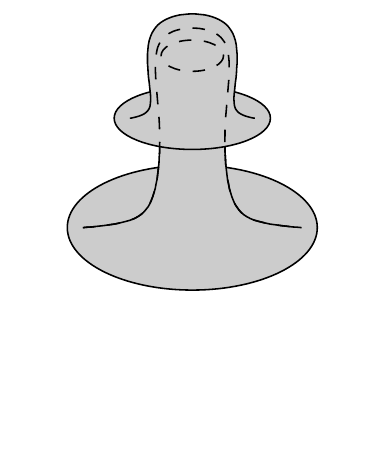}$}
    &\rightarrow\\
    &&&&&&&&\\
   \hline
   \hline
   &&&&&&&&\\
   \rightarrow&
   \fbox{$\dessin{3cm}{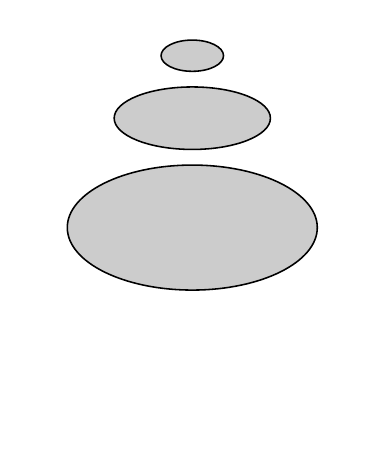}$}
    &\rightarrow&
    \fbox{$\dessin{3cm}{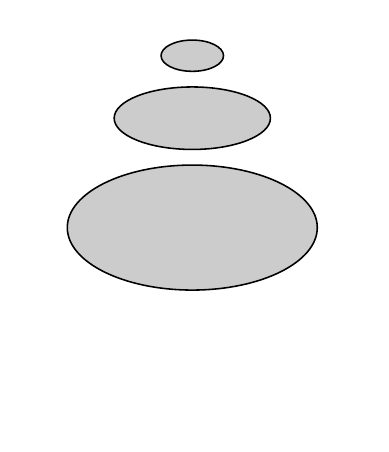}$}
    &\rightarrow&
    \fbox{$\dessin{3cm}{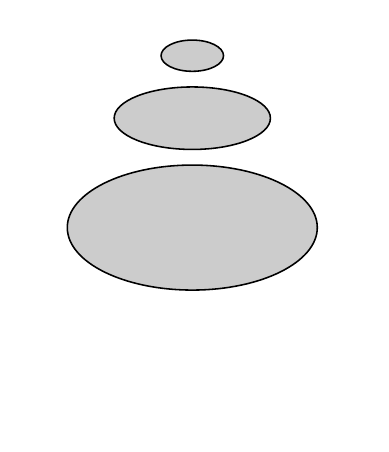}$}
    &\rightarrow&
    \fbox{$\dessin{3cm}{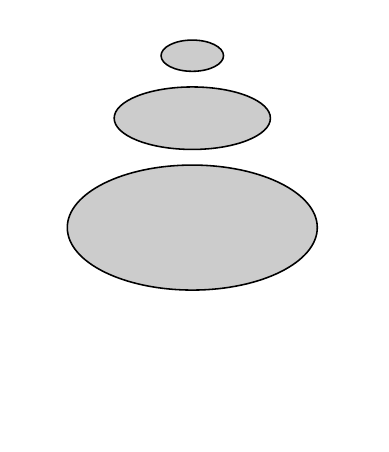}$}
    &\\[2cm]
    \rightarrow&\fbox{$\dessin{3cm}{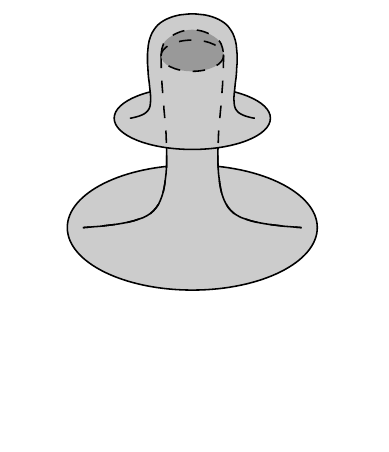}$}
    &\rightarrow&
    \fbox{$\dessin{3cm}{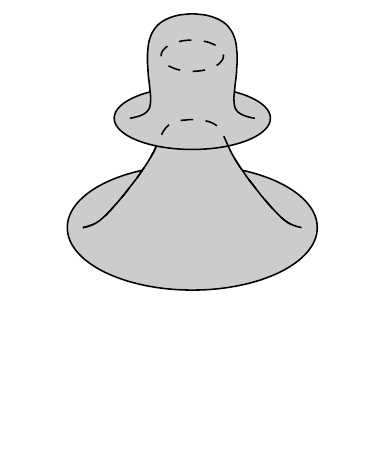}$}
    &\rightarrow&
    \fbox{$\dessin{3cm}{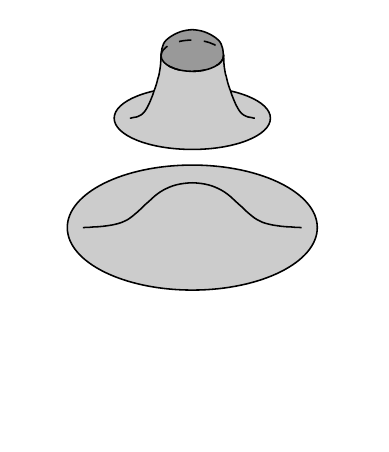}$}
    &\rightarrow&
    \fbox{$\dessin{3cm}{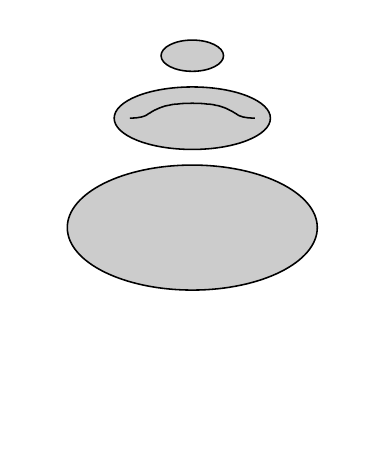}$}
    &
\end{array}
\]
  \caption{RIII seen as a filling change}
  \label{fig:FillChg_RIII}
\end{figure}


\begin{figure}
  \begin{gather*}
    \fdessinH{4cm}{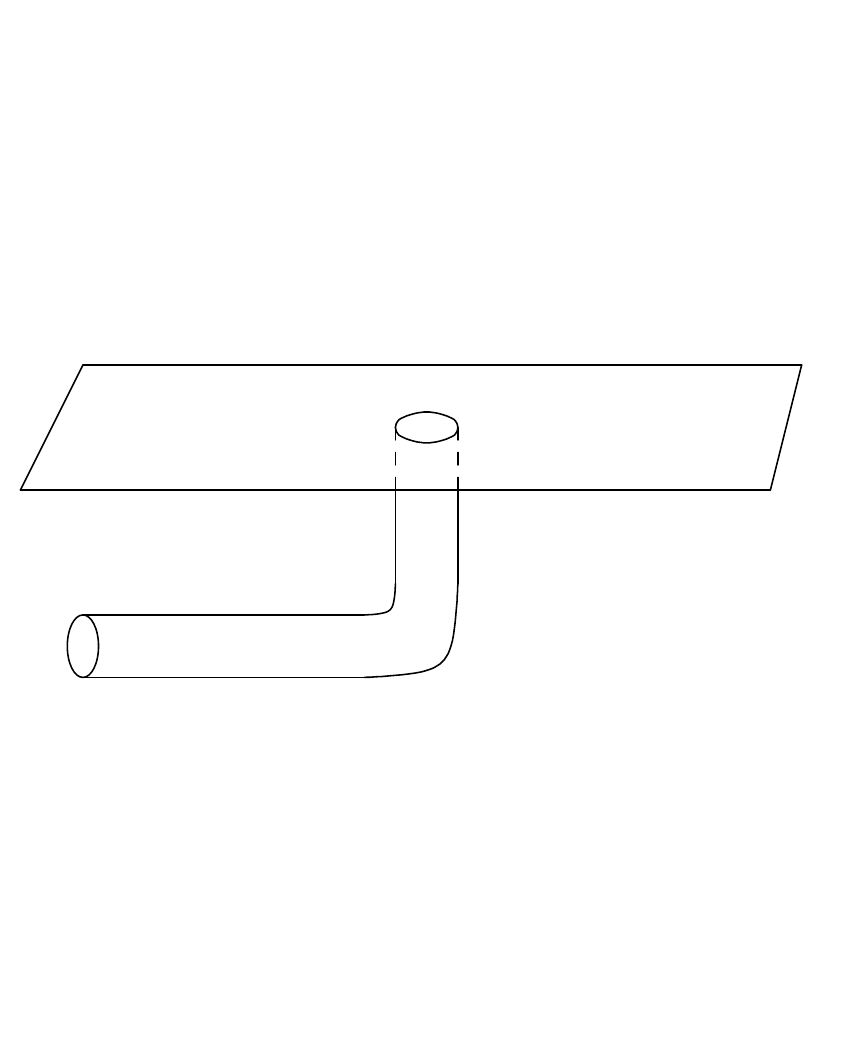}\Rar\fdessinH{4cm}{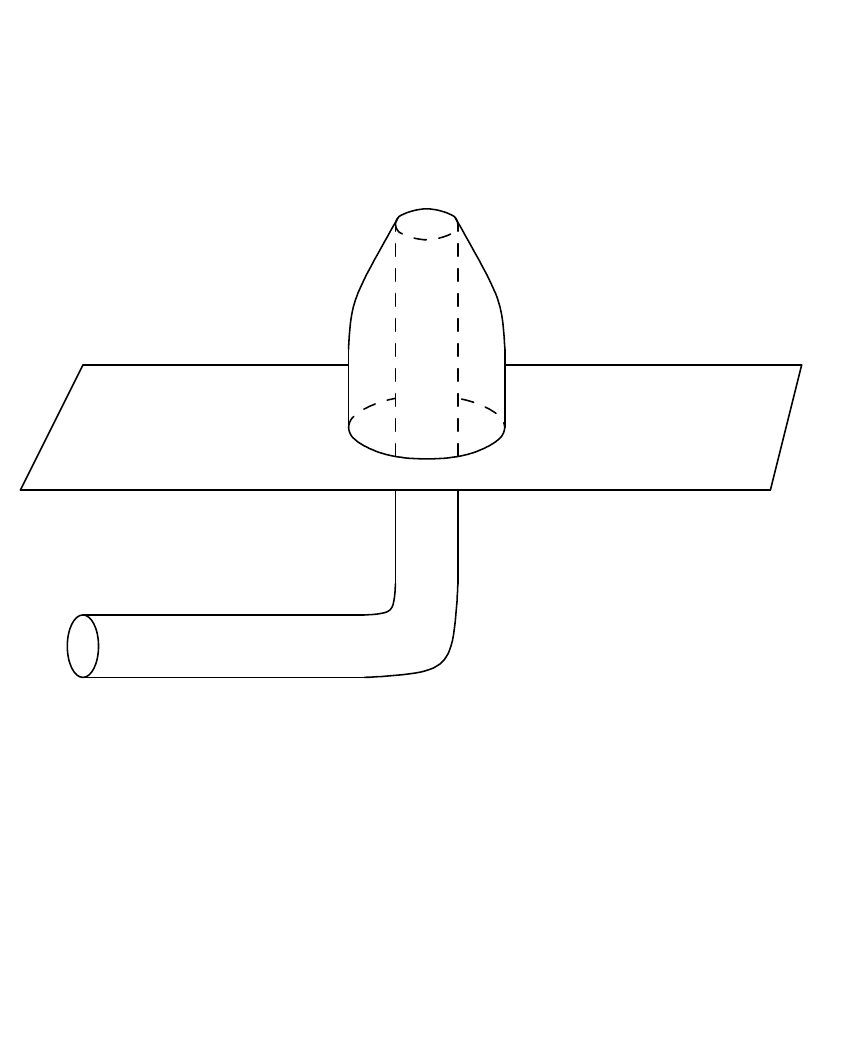}\\
    \hspace{7.5cm}\searrow{\huge \strut}\\
    \fdessinH{4cm}{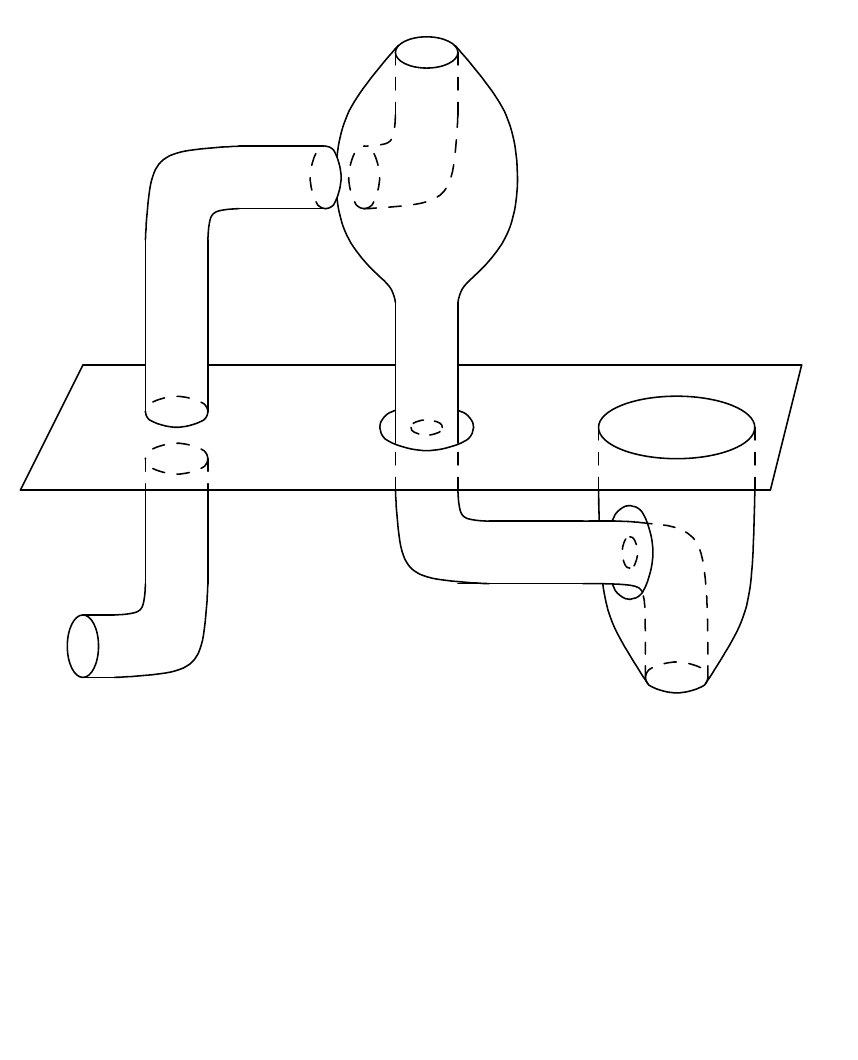}\Lar\fdessinH{4cm}{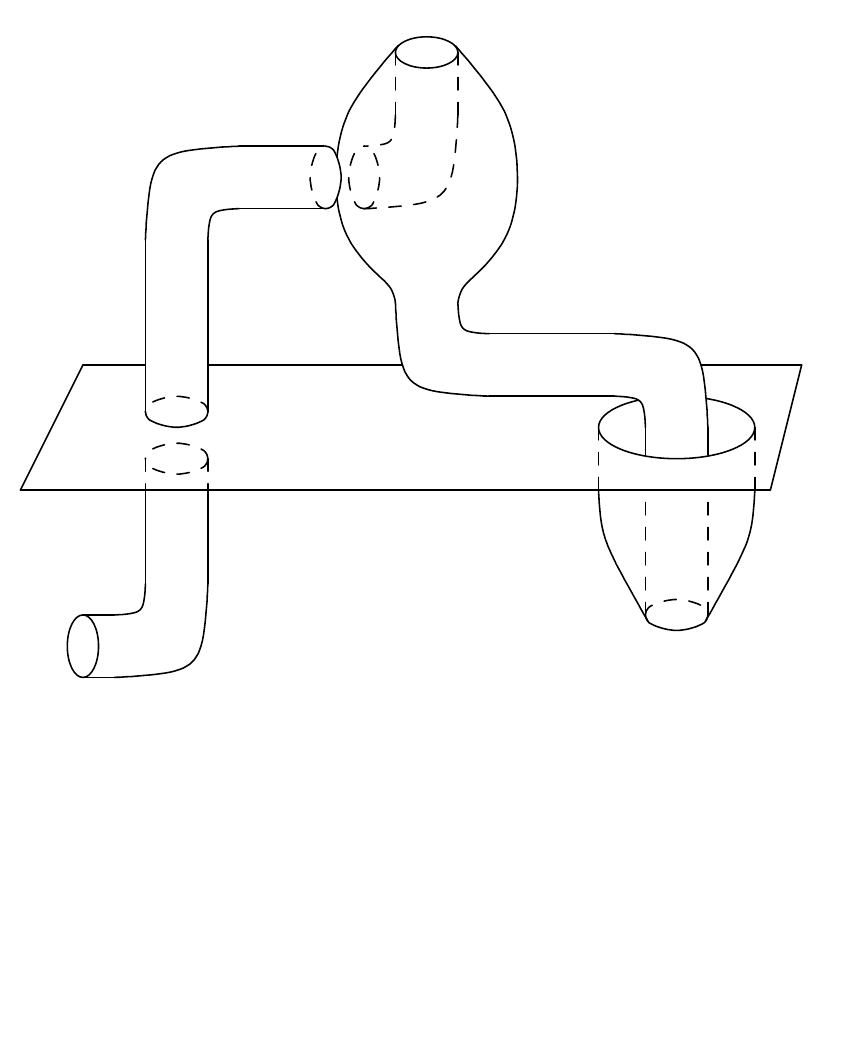}\Lar\fdessinH{4cm}{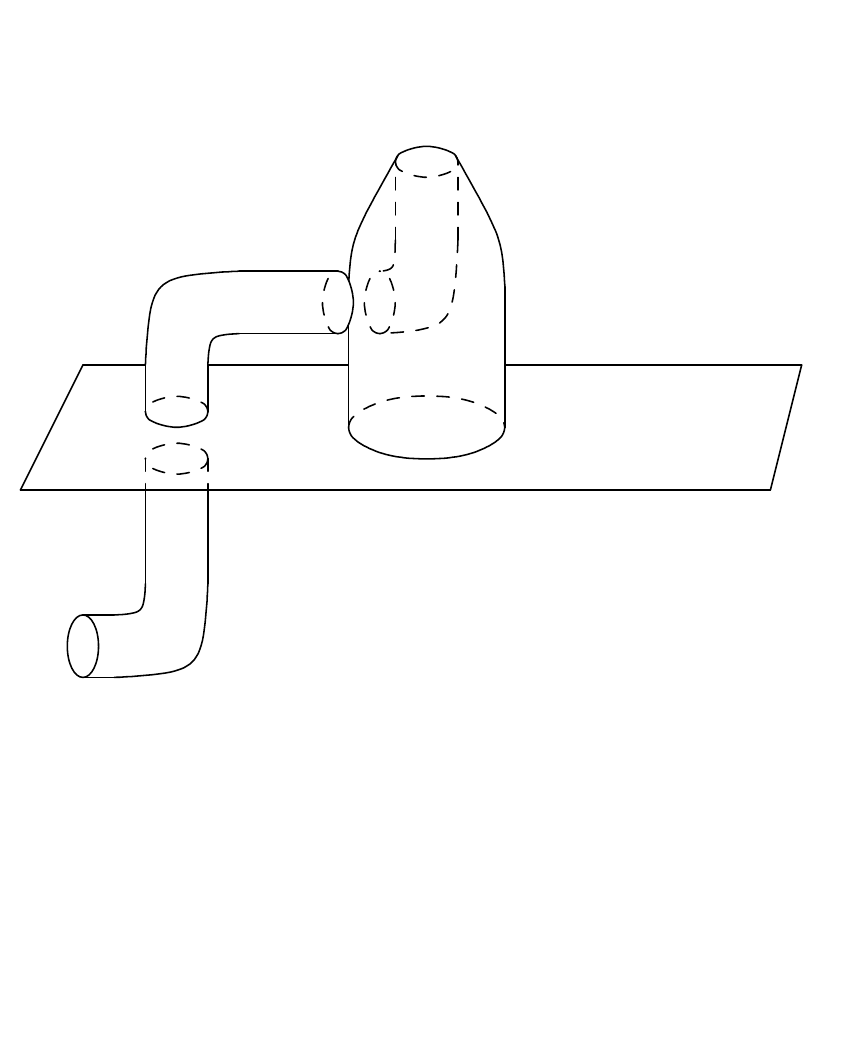}\\
    \downarrow{\huge \strut}\hspace{10cm}\\
    \fdessinH{4cm}{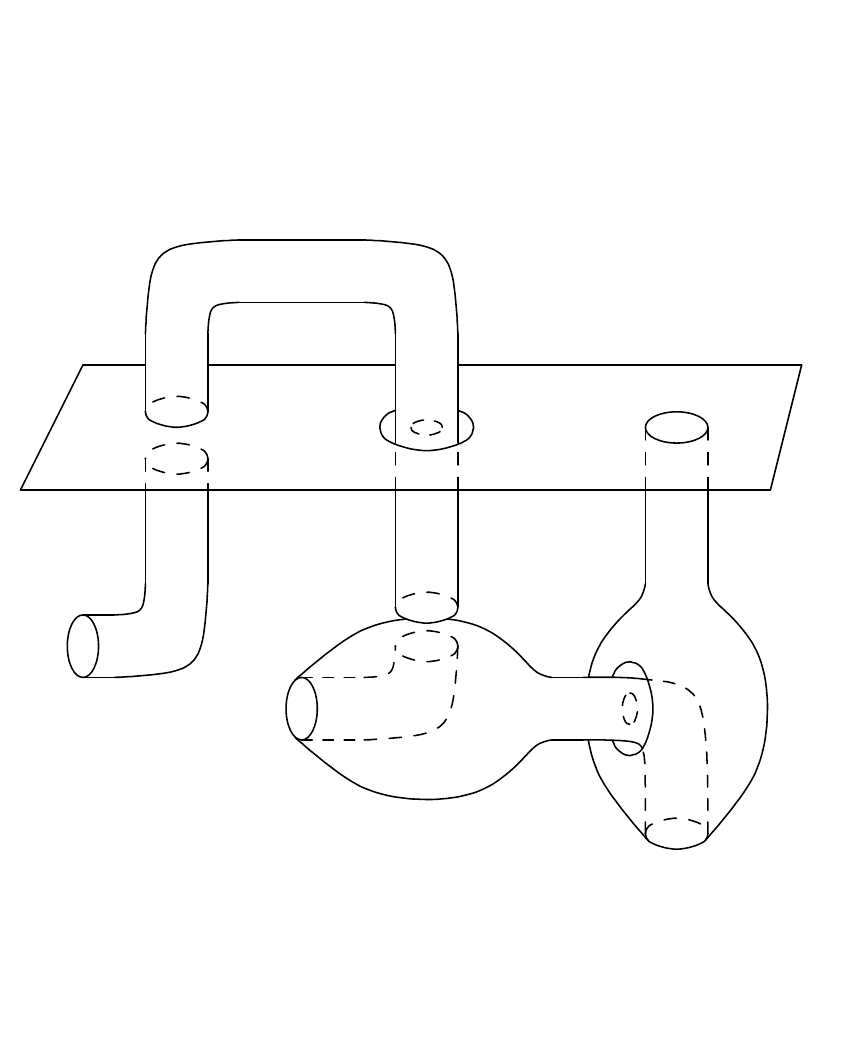}
\makebox[0cm]{\hspace{.9cm}$\rotatebox{90}{\small \hspace{.35cm}Prop. \ref{prop:AllWensAreOne}}$}\Rar
\fdessinH{4cm}{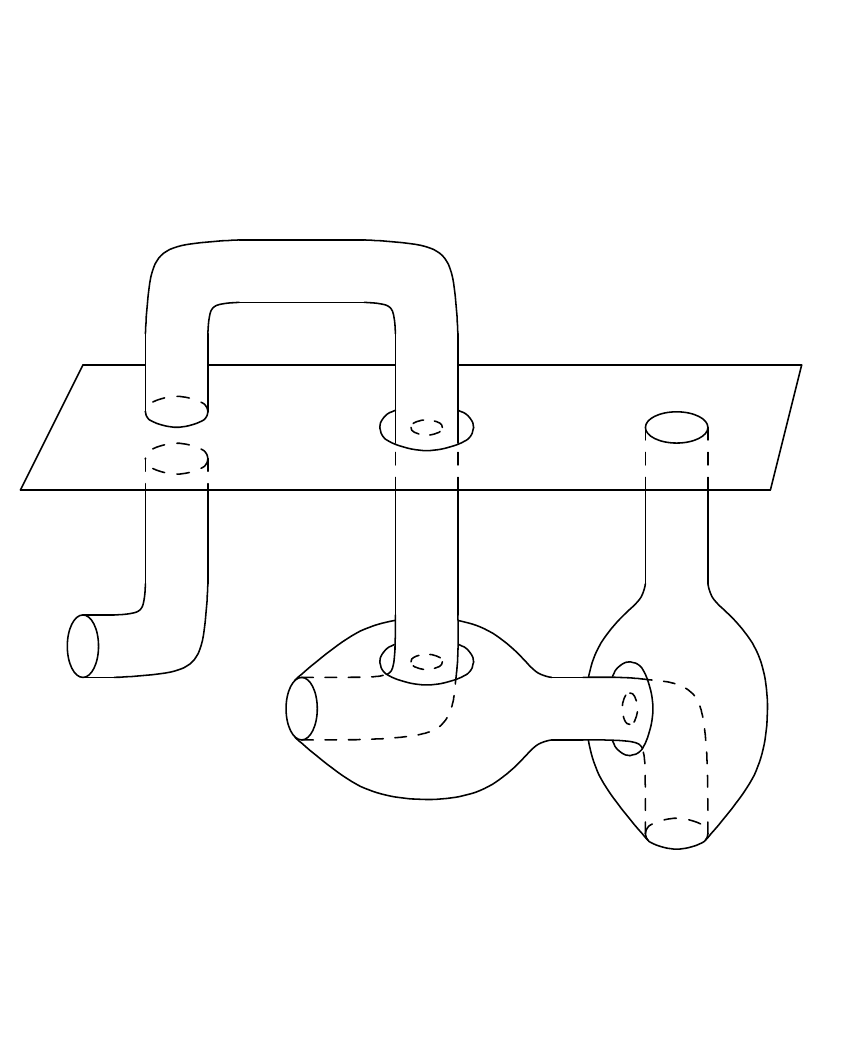}\Rar\fdessinH{4cm}{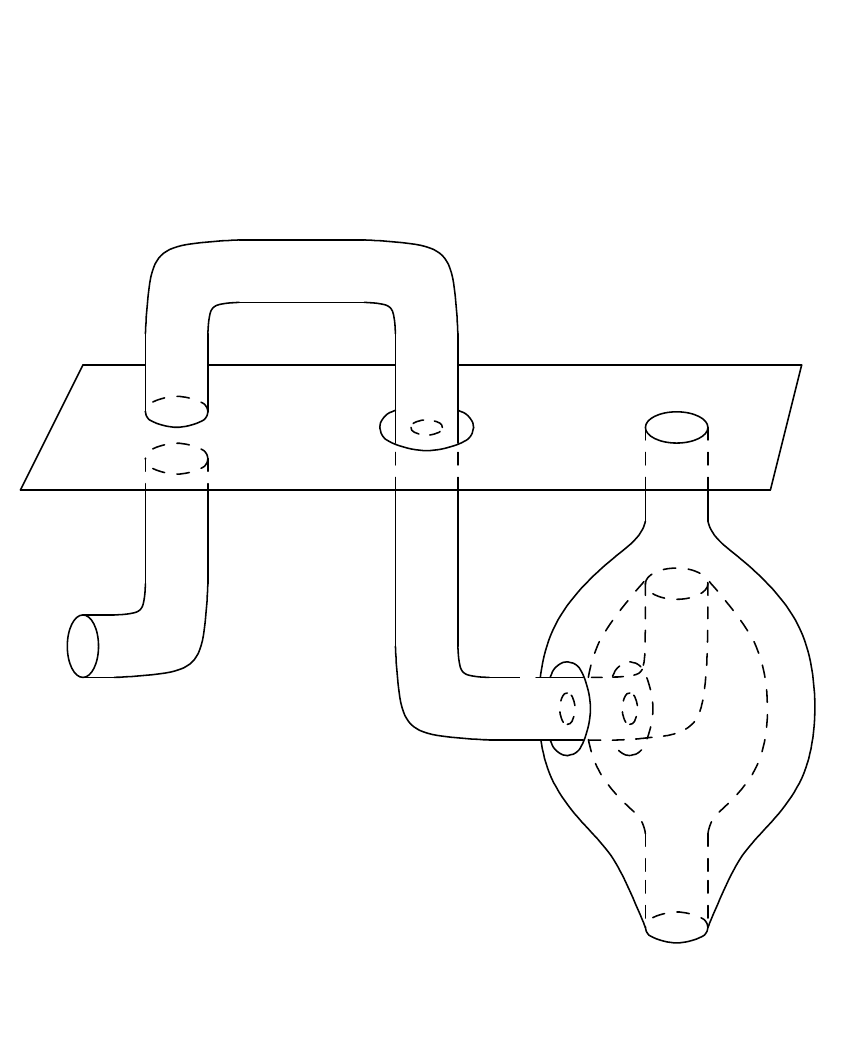}\\
    \hspace{7.5cm}\swarrow{\huge \strut}\\
   \fdessinH{4cm}{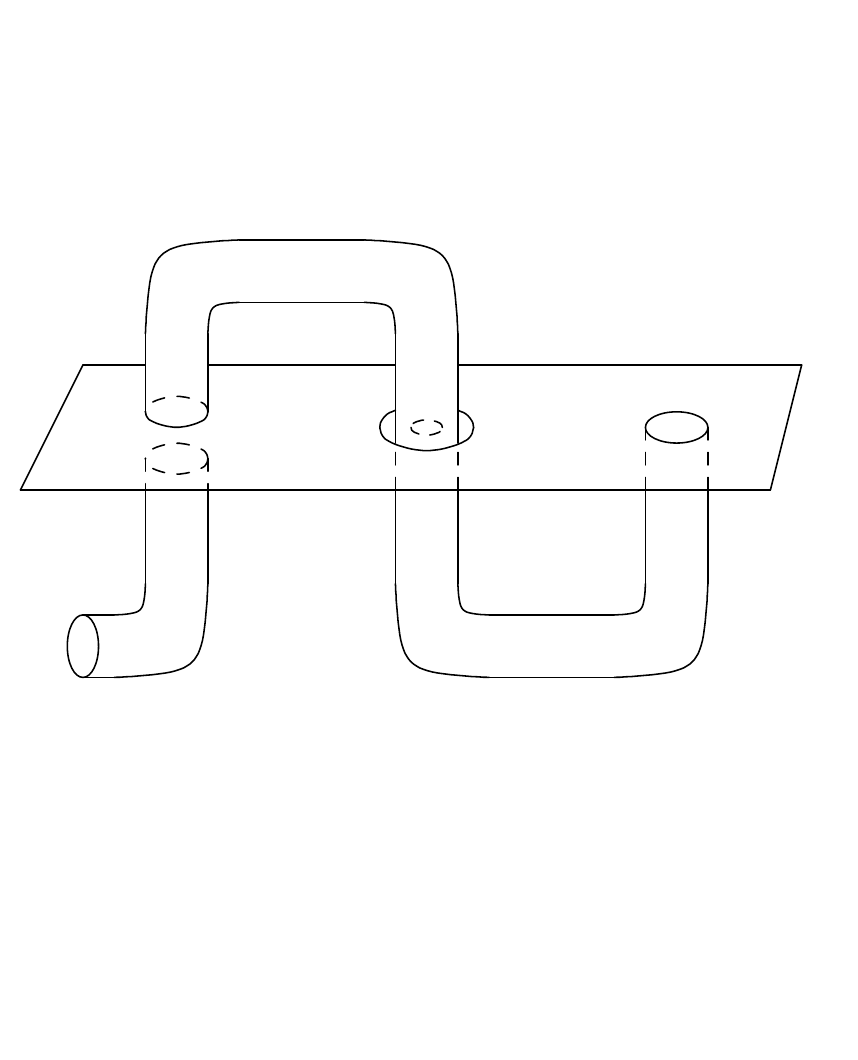}\Lar\fdessinH{4cm}{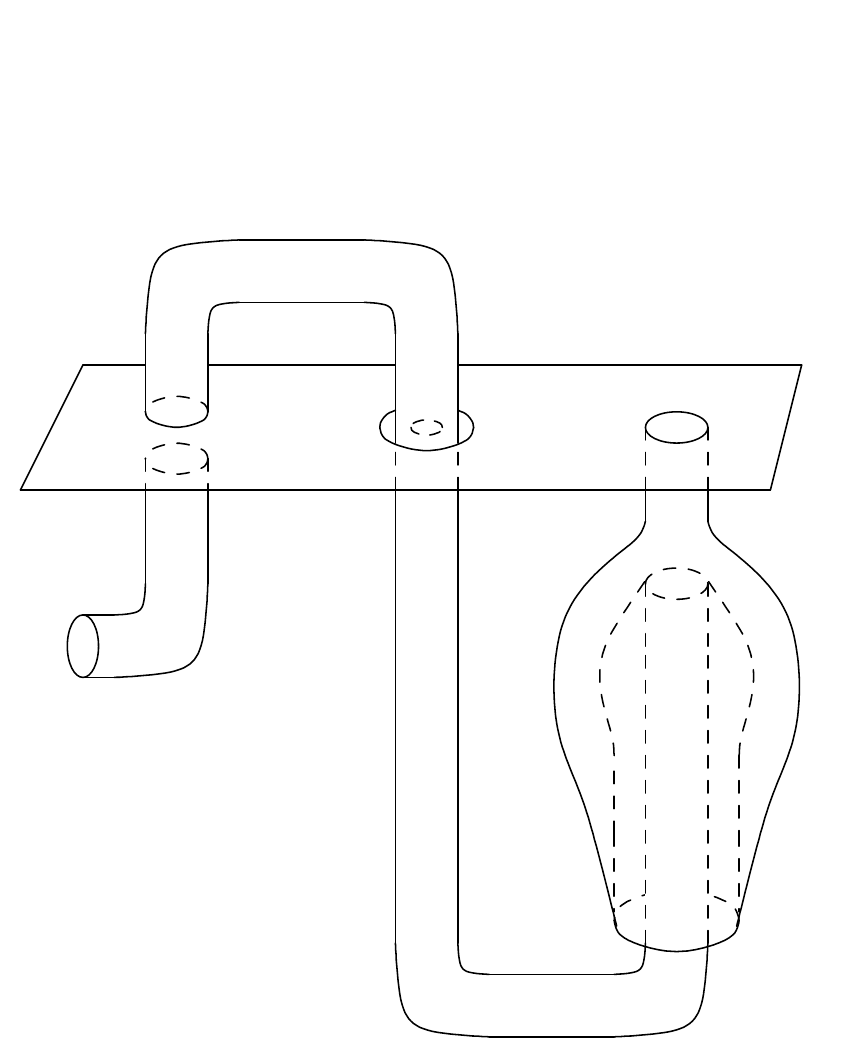}
  \end{gather*}
  \caption{{RI seen as a wen cancellation}}
  \label{fig:RI_isotopy}
\end{figure}

\end{document}